\documentclass[10pt]{elsarticle}
\usepackage{amsfonts,amssymb,amsmath,amsthm,indentfirst}

\usepackage{epsfig}
\usepackage{epstopdf}
\usepackage[latin5]{inputenc}
\usepackage{subfigure}
\usepackage{lineno}
\usepackage[mathscr]{eucal}
\usepackage{hyperref}
\usepackage{color}
\usepackage{graphics}
\usepackage{csquotes}
\usepackage{fullpage}
\modulolinenumbers[5]
\newcommand{\scal}[1]{\left\langle #1 \right\rangle}
\newcommand{\sett}[1]{\left\{   #1   \right\}}
\newcommand{\norm}[1]{\left\|   #1   \right\|}
\newcommand{\abso}[1]{\left|   #1   \right|}
\newcommand{\paar}[1]{\left(   #1   \right)}

\usepackage{color}

\numberwithin{equation}{section}

\newtheorem{lemma}{Lemma}

\newtheorem{theorem}{Theorem}
\newtheorem{cor}{Corollary}
\newtheorem{remark}{Remark}

\numberwithin{equation}{section}
\numberwithin{theorem}{section}
\numberwithin{lemma}{section}
\numberwithin{prop}{section}
\numberwithin{cor}{section}
\numberwithin{remark}{section}
\numberwithin{defn}{section}

\definecolor{darkolivegreen}{rgb}{0.333333, 0.419608, 0.1843140}

\DeclareMathOperator{\sech}{sech}

\newcommand{\la}{\langle}
\newcommand{\ra}{\rangle}
\newcommand{\hs}{{H^s(\rr)}}
\newcommand{\lt}{{L^2(\rr)}}
\newcommand{\ff}{\varphi}
\newcommand{\al}{\alpha}
\newcommand{\dd}{{\;\rm d}}
\newcommand{\ee}{{\rm e}}
\newcommand{\rr}{\mathbb{R}}
\newcommand{\ii}{{\rm i}}
\newcommand{\lam}{{\lambda}}
\newcommand{\what}{\widehat}
\newcommand{\vu}{\vec{u}}

\newcommand{\lc}{ \mathcal{L}}
\newcommand{\eps}{\epsilon}
\newcommand{\vdu}{\vec{u}_0}
\newcommand{\x}{\mathscr{X}}
\newcommand{\vtt}{\vartheta}
\renewcommand{\th}{\theta}
\renewcommand{\ss}{\mathfrak{S}}

\begin{document}
	
	\title{Long time behavior of solutions to the generalized Boussinesq equation}
	
	\author[a]{Amin Esfahani}
	\ead{saesfahani@gmail.com, amin.esfahani@nu.edu.kz}

	\author[b,c]{Gulcin M. Muslu}
	\ead{gulcin@itu.edu.tr, gulcinmihriye.muslu@medipol.edu.tr }

	\address[a]{Department of Mathematics, Nazarbayev University, Astana 010000, Kazakhstan}
	\address[b]{ Istanbul Technical University, Department of Mathematics, Maslak 34469, Istanbul,  Turkey  }
	\address[c]{Istanbul Medipol University,  School of Engineering and Natural Sciences, Beykoz 34810, Istanbul, Turkey}

	\begin{abstract}
		In this paper, we study the generalized Boussinesq equation as a model for the water wave problem with surface tension. Initially, we investigate the initial value problem within Sobolev spaces, deriving conditions under which solutions are either global or experience blow-up in time. Subsequently, we extend our analysis to Bessel potential and modulation spaces, determining the asymptotic behavior of solutions. We establish the non-existence of solitary waves for certain parameters using Pohozaev-type identities.
Additionally, we numerically generate solitary wave solutions of the generalized Boussinesq equation through the Petviashvili iteration method. To further examine the time evolution of solutions, we propose employing the Fourier pseudo-spectral numerical method. Our investigation extends to the gap interval, where neither global existence nor blow-up results have been theoretically established. We find that our numerical results effectively fill these gaps, supplementing the theoretical findings.
	\end{abstract}
	
	
	\date{\today}
	\journal{Journal}

	\maketitle
	
	\section{Introduction}
	In this paper, we study the generalized Boussinesq  (gBq) equation
	\begin{equation}\label{gBE}
		\begin{cases}
			u_{tt}={\big(u-\alpha u_{xx}+u_{tt}-\kappa u_{xxtt}+ f(u)\big)}_{xx},\\
			u(x,0)=u_0(x),\quad u_t(x,0)=u_1(x),\quad\alpha,\kappa\geq0	.
		\end{cases}	
	\end{equation}
	This equation was proposed by Schneider and Eugene \cite{se} to model the water wave problem with surface tension when $f(u) = u^2$. The solution $u(x,t)$ of \eqref{gBE} can be interpreted as the vertical velocity component on the top surface of an irrotational, incompressible fluid in a domain. Equation \eqref{gBE} can also be derived from the two-dimensional water wave problem. In a degenerate case (with Bond number $T=\frac{1}{3}$, a dimensionless parameter proportional to the surface tension), it was shown in \cite{se} that the long wave limit can be described approximately by two decoupled Kawahara equations.
It has been shown that the (improved) Boussinesq equation,
	\begin{equation}\label{imbq}
		u_{tt}={(u +u_{tt}+ u^2)}_{xx},
	\end{equation}
	is useful in the case of zero surface tension \cite{schneider}. Equation \eqref{imbq} is formally equivalent to the most well-known and classical Boussinesq equation
	\begin{equation}\label{bg-bouss}
		u_{tt}={(u \pm u_{xx}+ u^2)}_{xx},
	\end{equation}
	which is also valid (from a formal point of view) for unidirectional wave propagation in the water wave problem \cite{pego-weinstein}. Equation \eqref{bg-bouss} was first derived in \cite{Boussinesq} (rediscovered later by Keulegan and Patterson \cite{KP}) when investigating the bidirectional propagation of small amplitude and long wavelength capillary-gravity waves on the surface of shallow water.

Since the linear part of \eqref{bg-bouss} with the $+u_{xx}$ term contains the backward heat operator $\partial_t+\partial_x^2$, it is ill-posed. This issue does not occur with the $-u_{xx}$ term, but in this case, it cannot be proven as a physical model of water waves like the original Boussinesq equation.
To address this situation, and related to \eqref{gBE}, the following generalized Boussinesq equation was derived in \cite{cmv} (see also \cite{Amin-farah-wang} and references therein):
	\begin{equation}
		u_{tt}={(u \pm u_{xx }+u_{xxxx } +  u^2)}_{xx}.
	\end{equation}
	Rosenau \cite{rosenau} proposed the higher-order Boussinesq (HBq) equation
	\begin{equation}\label{hbq1}
		u_{tt}= (u+ \eta_1u_{tt}-\eta_2u_{xxtt}+(f(u))_{xx},
	\end{equation}
	where $\eta_1$ and $\eta_2$ are real positive constants to include higher-order dispersion effects. A more recent derivation of the HBq equation is provided in \cite{duruk2} to model the bi-directional propagation of longitudinal waves in an infinite, nonlocally elastic medium.
	
	Regarding the existence of local/global solutions of the initial value problem associated with \eqref{bg-bouss}, Bona and Sachs utilized Kato's theory to demonstrate the local well-posedness of smooth solutions. This result was further extended to $H^1$ in \cite{TsutsimiMatahashi}. For the nonlinearity $f(u)=-|u|^pu$, \cite{linares} established the existence of $L^2$- and $H^1$-solutions of \eqref{gBE}. Additionally, it was proven in \cite{linares} that these solutions, given small initial data, remain global in $H^1$ if $0 < p < 4$. Investigations conducted in \cite{liu} explored the existence of finite-time blow-up of solutions, decay of small solutions, and the strong instability of solitary waves with small propagation speeds for the Cauchy problem of the Boussinesq equation.

	In this paper, we study the initial value problem \eqref{gBE}. The well-posedness of \eqref{gBE} was shown in \cite{wangxu} using the contraction mapping principle.

Here, we improve upon the results of \cite{wangxu} by extending them to lower Sobolev indices (see Theorem \ref{theo-local-1}) in Section \ref{section-evolution}. Wang and Xu in \cite{wangxu} also derived conditions under which solutions of problem \eqref{gBE} either remain global or blow up, by applying the potential well approach. These results depend on the potential well depth $d$ and the construction of two stable and unstable submanifolds of $H^2$. The value of $d$ is variational and determined in terms of the stationary solution of \eqref{gBE}, which is the unique positive solution of
	\[
	0 =Q_0- \alpha  Q_0'' +\beta Q_0^{p+1},\qquad \beta<0.
	\]
	Our next objective in this paper is to investigate the time-decay behavior of solutions of \eqref{gBE}. Firstly, we utilize properties of the group associated with the linear part of \eqref{gBE} and extend our local well-posedness result to the Bessel potential spaces $H^{s,q}(\mathbb{R})$ and the modulation spaces $M^s_{p,q}$ in Sections \ref{beseel-space-section} and \ref{modul-sec}, respectively. In Section \ref{section-asymp}, we present two equivalent forms of equation \eqref{gBE} and again exploit decay properties of the linear part of \eqref{gBE} to obtain time decay estimates of solutions in $L^\infty$ and $H^s$ spaces.

	\noindent
	We also investigate the existence of nonstationary solitary waves of \eqref{gBE}. In contrast to the stationary case, we observe that the coefficient of $\beta$ should be positive to guarantee the existence of solitary waves. Interested readers can find further details in Section \ref{section-solitary}, where some non-existence results are discussed. As far as we know, the analytical solitary wave solution of \eqref{gBE} remains unknown. In Section \ref{section-solitary}, we employ the Petviashvili iteration method to numerically generate the solitary wave solution.
	
	\noindent
	A numerical method that combines a Fourier pseudo-spectral method in space with a Runge-Kutta method in time for the time evolution of solutions of the gBq equation is proposed in Section \ref{section-numericalmethod}. To the best of our knowledge, there have been no numerical studies conducted for the gBq equation, despite a considerable body of work in the literature aimed at solving equations \eqref{imbq}-\eqref{hbq1} (see \cite{borluk, bratsos1, canak, frutos, shokri, su, oruc} and references therein).

In Section \ref{section-numericalmethod}, we investigate the time evolution of solitary waves generated by Petviashvili's iteration method. Section \ref{section-evolution} addresses a gap interval where neither a global existence nor a blow-up result has been analytically established. Upon confirming the analytical results for global existence and blow-up numerically, as presented in Section \ref{section-evolution}, we then provide numerical insights into the behavior of solutions corresponding to the gap interval. We determine a threshold value, such that when the initial amplitude is below this threshold, the solution exists globally; however, if the initial amplitude exceeds this threshold, the solution blows up in a finite time.
	
	In the throughout of this paper,  we assume that $f\in C(\rr)$ such that $|f(u)|\leq C |u|^{p+1}$  and $C\geq0$, and $\|u\|_s$ stands by the norm of $u$ in the Soblev space $H^s(\rr)$.

	\setcounter{equation}{0}
	\section{Well-posedness}\label{section-evolution}
	In this section, we study the well-posedness of the initial value problem \eqref{gBE}. To do so, first, similar to the method in \cite{Amin-farah-wang}, by using a Fourier transform and the Duhamel principle, the solution of problem \eqref{gBE} can be written as follows:
	\begin{equation}\label{contract-op}
		u(x,t)= \mathbb{A}(u)=\partial_t S(t)u_0 (x)+S(t)u_1(x)
		+\int_0^tS(t-\tau)\lc_R^{-1}\partial_x^2f(u)\dd\tau,
	\end{equation}
	where  $\lc_R=I- \partial_x^2+\kappa\partial_x^4$ and
	\[
	(\partial_t S(t)f)^\wedge =\cos\left(\frac{t|\xi|\sqrt{1+\al\xi^2}}{\sqrt{1+\xi^2+\kappa\xi^4}}\right)\hat{f}(\xi),\quad
	(S(t)f)^\wedge =\sin\left(\frac{t|\xi|\sqrt{1+\al\xi^2}}{\sqrt{1+\xi^2+\kappa\xi^4}}\right)
	\frac{\sqrt{1+\xi^2+\kappa\xi^4}}{|\xi|\sqrt{1+\al\xi^2}}\hat{f}(\xi).
	\]
	\begin{remark}
		Notice that $\lc_R^{-1}f=K\ast f$, where
		\begin{equation}\label{E:kernel_formulas}
			K(x)=\kappa\pi\left\{\begin{array}{lll}
				\frac{1}{\lambda_2^2-\lambda_1^2}\left(\frac{1}{\lambda_1}\ee^{-\lambda_1|x|}-\frac{1}{\lambda_2}\ee^{-\lambda_2|x|}\right),&&4\kappa<1,
				\\ \\
				\frac{ \kappa^\frac34}{2}\left(1+\kappa^{-\frac14}|x|\right)\;
				\ee^{-\kappa^{-\frac14}|x|},&&4\kappa=1,
				\\ \\
				\frac{  e^{-\sigma|x|}}{2\sigma\omega(\sigma^2+\omega^2)}\left(\omega\cos(\omega x)+\sigma\sin(\omega|x|)\right),
				&&4\kappa>1,
			\end{array}
			\right.
		\end{equation}
		and
		\begin{equation}\label{E:lambda_sigma_omega}
			\begin{split}
				&\lambda_1 =
				\sqrt{\frac1{2\kappa}\left( 1-\sqrt{1-4\kappa}\right)}\qquad
				\lambda_2 =	\sqrt{\frac1{2\kappa}\left(1+\sqrt{1-4\kappa}\right)}\\
				&\sigma =\frac12\sqrt{  \frac2{\sqrt{\kappa}}+\frac1\kappa}\qquad
				\omega =\frac12\sqrt{ \frac2{\sqrt{\kappa}}-\frac1\kappa}.
			\end{split}
		\end{equation}
		It is clear that $K$ is a real-valued, even, continuous function in $\rr$. Moreover, $K$ is smooth at $\rr\setminus\{0\}$, and all its derivatives   decay exponentially. Furthermore, $K^{(n)}\in L^q(\rr)$ for any $1\leq q\leq\infty$ and $n\in\mathbb{N}$.
	\end{remark}
	
	By the arguments in \cite{wangxu}, we can obtain the following lemmas.
	\begin{lemma}\label{lem-1-lines}
		The following inequalities hold for all $\ff\in H^s(\rr)$ and $\psi\in H^{s-2}(\rr)$:
		\begin{enumerate}[(i)]
			\item  $\|\partial_tS(t)\ff\|_s\leq\|\ff\|_s$,
			\item  $\| S(t)\ff\|_s\leq2(1+t)\|\ff\|_s$,
			\item  $\|\partial_{tt} S(t)\ff\|_s\leq\sqrt{2(1+t)}\|\ff\|_s$,
			\item  $\|\lc_R^{-1} S(t)\partial_x^2\psi\|_s\leq\sqrt{2 }\|\psi\|_{s-2}$,
			\item  $\|\lc_R^{-1} \partial_t\partial_x^2S(t)\psi\|_s\leq \|\psi\|_{s-2}$.
		\end{enumerate}
	\end{lemma}
	
	\begin{lemma}\label{lem-2}
		Let $f\in C^m(\rr)$ be a function vanishing at zero, where $m\geq0$ is an integer. Then, for any $u,v\in H^s(\rr)\cap L^\infty(\rr)$ and $0\leq s\leq m$, the following statements hold:
		\begin{enumerate}[(i)]
			\item  $f(u)\in H^s(\rr)$ and $\|f(u)\|_s\leq C_1(\|u\|_{L^\infty(\rr)})\|u\|_s$,
			\item    $\|f(u)-f(v)\|_s\leq C_2(\|u\|_{L^\infty(\rr)},\|v\|_{L^\infty(\rr)})\|u-v\|_s$,
		\end{enumerate}
  where $C_1$ and $C_2$ are continuous nonnegative functions.
	\end{lemma}
	
	Now we 	define the function space $X(T)=C^1(0,T;H^s(\rr))$ equipped with the norm
	\[
	\|u\|_{X(T)}=\sup_{t\in[0,T]}(\|u(t)\|_s+\|u_t(t)\|_s).
	\]
	Using the Sobolev embedding and the arguments in \cite{wangxu} one can prove the following result.
	\begin{theorem}
		Let $f\in C^m$ with $m\geq s>1/2$. If $u_0,u_1\in \hs$, then there exist  $T>0$ and a unique solution $u\in C^1([0,T],\hs)$ of \eqref{gBE}.
	\end{theorem}
	
	In the case $p\in\mathbb{N}$, we can extend the above result to the lower Sobolev indices.
	
	\begin{theorem}\label{theo-local-1}
		Let $u_0,u_1\in H^s(\rr)$ and $p\in\mathbb{N}$. Assume that $s>\frac{p-1}{2(p+1)}$ if $p>1$ and $s\geq0$ if $p=1$. Then there exists $T>0$ and a unique solution $u\in C^1(0,T;H^s(\rr))$ such that $u(0)=u_0$ and $u_t(0)=u_1$. Moreover, if $T<+\infty$ is the maximal existence time of $u$, then
		\[
		\sup_{t\in[0,T]}(\|u(t)\|_s+\|u_t(t)\|_s)=+\infty.
		\]
		Furthermore, we have $u_{tt}\in C(0,T;H^s(\rr))$. In addition, if $(-\partial_x^2)^{-\frac12}u_1\in L^2(\rr)$, then $(-\partial_x^2)^{-\frac12}u_t\in C^1(0,T;L^2(\rr))$.
	\end{theorem}	
	\begin{proof}

		It is shown that the right-hand side of \eqref{contract-op} defined is a contraction mapping $\mathbb{A}$ of $C(0, T ;B)$  for some $T > 0$, where $B$ is a closed ball in $X(T)$. Indeed, since $S$ and $\partial_tS$ are both bounded on $H^s(\rr)$, then it is enough to study the nonlinear part of \eqref{contract-op} (see the proof of Theorem \ref{local-bessel} below).
		
		Note that when $p=1$, then
		$\la\xi\ra^s|\widehat{f(u)}|\leq (\la\xi\ra^s|\hat{u}|)\ast (\la\xi\ra^s|\hat{u}|)\leq \|u\|_s^2$, where $\la\xi\ra=1+|\xi|$.
		If $p>1$, then we get from the Young inequality that
		\begin{equation}\label{estimate-1}
			\begin{split}
				\la\xi\ra^{s_1}|\widehat{f(u)}|&\leq (\la\xi\ra^{s_1}|\hat{u}|)\ast\cdots\ast (\la\xi\ra^{s_1}|\hat{u}|) \\&\leq
				\left\|\la\xi\ra^{s_1}|\hat{u}|\right\|_{L^{(p+1)/p  }(\rr)}^{p+1}
				\leq \|u\|_{s }^{p+1}
				\left(\int_\rr\la\xi\ra^{\frac{2(p+1)(s_1-s)}{p-1}}\dd x\right)^{\frac{p-1}{2}}\\
				&\lesssim \|u\|_{s }^{p+1}
			\end{split}
		\end{equation}
		provided $s>s_1+\frac{p-1}{2(p+1)}$.
		We have from the above inequalities that
		\[
		\begin{split}
			\|\partial_x^2\lc_R^{-1}f(u)\|_s^2&=\int_\rr \frac{\xi^4\la\xi\ra^{2s}}{(1+\xi^2+\kappa\xi^4)^2}|\widehat{f(u)}|^2\dd\xi\\&
			\lesssim
			\int_\rr   \la\xi\ra^{2s-4} |\widehat{f(u)}|^2\dd\xi\\&
			\lesssim
			\|u\|_s^{2(p+1)} \int_\rr\la\xi\ra^{2s-4-2s_1}\dd x\\&\lesssim \|u\|_s^{2(p+1)},
		\end{split}
		\]
		for any $s>\frac{p-1}{2(p+1)}$ if $p>1$, and $s\geq0$ if $p=1$.
		
		Next, we have from \eqref{gBE} that
		
		\begin{equation}\label{double-t}
			\hat{u}_{tt}=-\frac{\xi^2(1+\al\xi^2)}{1+\xi^2+\kappa\xi^4}\hat{u}+
			\frac{\xi^2 }{1+\xi^2+\kappa\xi^4}\widehat{f(u)}.
		\end{equation}
		
		Then,
		\[
		\|u_{tt}\|_s
		=\|\hat{u}_{tt}\|_s\lesssim\|u\|_s+\|u\|_s^{p+1}\leq C(t)
		\]
		for any $t\in(0,T)$. On the other hand, we get from
		\[
		\|u_{tt}(t+\Delta t)-u_{tt}(t)\|_s=\|\hat{u}_{tt}(t+\Delta t)-\hat{u}_{tt}(t)\|_s\leq C(T)\|u (t+\Delta t)-u (t)\|_s\to0,\quad\mbox{as}\,\,\Delta t\to0,
		\]
		that $u_{tt}\in C(0,T;H^s(\rr))$.
		
		Now, we show that $(-\partial_x^2)^{-\frac12}u_t\in C^1(0,T;L^2(\rr))$ if $(-\partial_x^2)^{-\frac12}u_1\in L^2(\rr)$. Similar to above argument, we have from \eqref{double-t} and \eqref{estimate-1} that
		\[
		\left\|\frac{|\xi|(1+\al\xi^2)}{1+\xi^2+\kappa\xi^4}\hat{u}\right\|_{s+1}
		\lesssim\|u\|_s^2
		\]
		and
		\[
		\left\|\frac{|\xi|}{1+\xi^2+\kappa\xi^4}\widehat{f(u)}\right\|_{s+1}^2	
		\leq C(\|u\|_s)\|u\|_s^2.
		\]
		Analogous to the above estimate, we can show that
		\[
		\left\|(-\partial_x^2)^{-\frac12}u_{tt}(t+\Delta t)-(-\partial_x^2)^{-\frac12}u_{tt}(t)\right\|_{s+1}\to0\quad\mbox{as}\,\,\Delta t\to0.
		\]
		So that $(-\partial_x^2)^{-\frac12}u_{tt}\in C(0,T;H^{s+1}(\rr))$. Now, if $(-\partial_x^2)^{-\frac12}u_1\in L^2(\rr)$, then we derive from
		\[
		(-\partial_x^2)^{-\frac12}u_t=(-\partial_x^2)^{-\frac12}u_1+\int_0^t(-\partial_x^2)^{-\frac12}u_{\tau\tau}\dd\tau
		\]
		that $(-\partial_x^2)^{-\frac12}u_t\in C^1(0,T;L^2(\rr))$.
	\end{proof}
	
	In order to get more time-regularity of the local solution, we obtain the following result by using Lemma \ref{lem-1-lines}. The invariant $E$ also shows that the solutions is global in time.
	\begin{theorem}\label{global-sln}
		Let $s \geq 1$ and $u_0,u_1\in H^s(\rr)$ and $(-\partial_x^2)^{-\frac12}u_1\in H^{s+1}(\rr)$.  Then the solution $u(t)$, obtained from Theorem \ref{theo-local-1}, satisfies $u\in C^2(0,T;H^s(\rr))$,
		\[
		(-\partial_x^2)^{-\frac12}u_t\in C^1(0,T;H^{s+1}(\rr)),
		\]
		\[
		Q(u(t))=\int_\rr\left(u(-\partial_x^2)^{-\frac12}u_t+  u_xu_t+\kappa u_{xx}u_{xt}\right)\dd x=Q(u(0))
		\]
		and
		\[
		\begin{split}
			E(u(t))&=\frac12\left(\|(-\partial_x^2)^{-\frac12}u_t\|_\lt^2	+ \|u_t\|_\lt^2
			+ \kappa \| u_{xt}\|_\lt^2
			+ \|u\|_\lt^2+
			\al \|u_x\|_\lt^2\right)\\
			&\quad
			+ \int_\rr F(u)\dd x
			=E(u(0))
		\end{split}
		\]
		for all $t\in(0,T)$, where $F(u)=\int_0^uf(\tau)\dd\tau$. Moreover, if $F(u)\geq0$ for all $u\in\rr$, the solution $u(t)$ can be extended in time.
	\end{theorem}
	\begin{proof}
		Following the same lines of the proof of Lemma \ref{lem-1-lines}, we can obtain that
		\[
		\begin{split}
			&\|(-\partial_x^2)^{-\frac12}\partial_tS(t)\ff\|_s\leq\|(-\partial_x^2)^{-\frac12}\ff\|_s\\&
			\|(-\partial_x^2)^{-\frac12}\partial_{tt}S(t)\ff\|_s\leq\sqrt{2 }\|\ff\|_s,\\&
			\|(-\partial_x^2)^{-\frac12}\lc_R^{-1} \partial_t\partial_x^2S(t)\psi\|_{s+1}\leq \|\psi\|_{s-2}.
		\end{split}
		\]
		Therefore,   we have from \eqref{contract-op} and the proof of Theorem \ref{theo-local-1} that
		$(-\partial_x^2)^{-\frac12}u_t\in C^1(0,T;H^{s+1}(\rr))$ and $u_{tt}\in C^1(0,T;H^{s}(\rr))$. The conservation of the energy $E$ and the momentum  $Q$ easily follows from a simple
		regularization argument. This completes the proof of the first part of the Theorem.
		
		The global existence theorem is derived immediately from the conservation law $E$.
	\end{proof}
	\begin{remark}
		Theorem \ref{global-sln} shows that the local solution of \eqref{gBE} is globally well-posedness if $f(u)=au^{2p+1}$ with $a>0$.
	\end{remark}
	
	Our next aim is to obtain the conditions under which the local solution is global or blows up in time.
	
	In the rest of this section, we assume for simplicity that $\alpha=1$.

	Let $s\geq1$. Define the functionals
	\[
	J(u)=\frac12\|u\|_1^2+\int_\rr F(u)\dd x,\qquad I(u)=\|u\|_1^2+\int_\rr uf(u)\dd x,
	\]
	and the minimizing value $d=\inf_{u\in N}J(u)$, where
	\[
	N=\{u\in H^1\setminus\{0\},\; I(u)=0\} .
	\]
	The potential well theory provides some global results (see \cite{ryb}).
	\begin{theorem}\label{global-1}
		Let $s\geq1$, $a>0$ and $f(u)$ satisfies one of the following form
		\begin{enumerate}[(i)]
			\item $f(u)=- a|u|^{p}u$,  $p\neq 2k$, $k\in\mathbb{N}$, $[p]+1\geq s$,
			\item 	$f(u)=\pm au^{2k}$,   $k\in\mathbb{N}$,
			\item 	$f(u)=- au^{2k+1}$,    $k\in\mathbb{N}$,
			\item $f(u)=\pm a|u|^{p+1}$,   $p>0$,
			\item  $f(u)=- a|u|^{p}u$,   $p>0$.
		\end{enumerate}
		Moreover, assume that $u_0,u_1\in H^s$ and $(-\partial_{x}^2)^{-1/2}u_1\in\lt$ and $E(u(0))<d$ and $u_0\in W'$ then \eqref{gBE} admits a unique global solution $u\in C^2(0,\infty;H^1)$, with $(-\partial_{x}^2)^{-1/2}u_t\in C^1(0,\infty;\lt)$ and $u\in W$ for all $0\leq t<\infty$, where
		\[
		W'=\{u\in H^1,\; I(u)>0\}\cup\{0\},\qquad W=\{u\in H^1,\; I(u)>0,\; J(u)<d\}\cup\{0\}.
		\]
	\end{theorem}
	
	The same approach used for Theorem \ref{global-1} provides the following blow-up result (see \cite[Theorem 6.2]{ryb}).
	\begin{theorem}\label{blowup-1}
		Let $s\geq1$ and    $f(u)=\pm a|u|^{p+1}$, with  $p+1\neq 2k$, $k\in\mathbb{N}$, and $[p]+1\geq s$, or $f$ satisfy one of   conditions (ii)-(iv) of Theorem \ref{global-1}, with $[p]\geq2$.  Assume that $u_0,u_1\in H^s$, $(-\partial_x^2)^{-1/2}u_0,(-\partial_x^2)^{-1/2}u_1\in\lt$, $E(u(0))<d$ and $I(u_0)<0$.  Then the solution of \eqref{gBE} blows up in finite time.
	\end{theorem}
	Following Levine's method (see \cite{levine}), one can show the following blow-up result (see \cite[Theorem 3.4]{wangxu}). The nonlinearites $f(u)=-au^{2p+1}$ with $a>0$ or $f(u)=au^{2p}$ with $a\neq0$ and $p\in\mathbb{N}$ satisfy the following theorem.
	\begin{theorem}\label{blowup-2}
		Let $f\in C^1(\rr)$ and $u_0,u_1\in H^1$ such that $F(u_0)\in L^1$,   $(-\partial_x^2)^{-1/2}u_0\in H^2$ and $(-\partial_x^2)^{-1/2}u_1\in L^2$. Suppose that there exists  $\varepsilon>0$ such that
		\begin{equation}\label{inequality-nonlin}
			uf(u)\leq 2(2\varepsilon+1)F(u)+2\varepsilon u^2,\qquad  \forall u\in\rr,	
		\end{equation}
		and one of the following two assumptions holds:
		\begin{enumerate}[(1)]
			\item $E(u(0))<0$,
			\item $E(u(0))\geq0$ and
			\begin{align*}
				\left\langle(-\partial_x^2)^{-1/2}u_0,  (-\partial_x^2)^{-1/2} u_1\right\rangle&
				+\kappa \left\langle (u_0)_x,(u_1)_x\right\rangle
				+ \left\langle u_0,u_1\right\rangle\\& >
				\sqrt
				{E(u(0))\left(\|(-\partial_x^2)^{-1/2}u_0\|_\lt^2+\kappa  \|(u_0)_x\|_\lt^2+\|u_0\|_\lt^2
					\right)} .
			\end{align*}
		\end{enumerate}
		Then the solution of \eqref{gBE} blows up in finite time.
	\end{theorem}
	\begin{proof}
		The proof is based on applying the well-know inequality $\phi''\phi\geq(\varepsilon+1)(\phi')^2$, where
		\[
		\phi(t)=\|(-\partial_x^2)^{-1/2}u(t)\|_{L^2}^2+\kappa \| u_x(t)\|_{L^2}^2+\|u(t)\|_{L^2}^2+\delta(t+t_0)^2
		\]
		for suitable numbers $t_0$ and $\delta$ depending on $E(u(0))$. So we omit the details.
	\end{proof}
	\begin{remark}
		Inequality \eqref{inequality-nonlin} can be replaced with
		\begin{equation}
			uf(u)\leq2(1+2\varepsilon)F(u)+\frac{2}{C_\ast^2}\left(\int_\rr F(u)\dd x\right)^{\frac{2}{p+2}},
		\end{equation}
		where $C_\ast=\|\ff\|_1^{-\frac{p}{p+2}}$ and $\ff(x)=\left(\frac{p+2}{2}\right)^{\frac1p}\sech^{\frac2p}(\frac{px}{2} )$.
	\end{remark}
	
	In the case $f(u)=-a|u|^pu$ with $a>0$, we define the stable (potential well) and unstable submanifolds, by
	\[
	W^-=\left\{u\in H^1(\rr),\;\|u\|_1^2<\frac{2(p+2)}{p}e_0\right\},
	\quad
	W^+=\left\{u\in H^1(\rr),\;\|u\|_1^2>\frac{2(p+2)}{p}e_0\right\},
	\]
	where
	\[
	e_0=\frac{p}{2(p+2)}a^{-\frac{2}{p}}C_\ast^{-\frac{2(p+2)}{p}}.
	\]
	It can be shown similar to the proof of Lemma 4.1 in \cite{wangxu} that if $E(u(0))<e_0$, $u_0,v_1\in H^1$ and $(-\partial_x^2)^{-1/2}u_1\in H^2$, then
	the submanifolds $W^\pm$ are invariant under the flow of \eqref{gBE}. These sets enable us to find the following global/blow-up result by following the proof of  \cite[Theorem 4.6]{wangxu} with slight modifications.
	\begin{theorem}
		Let $u_0,v_1\in H^s$ and $(-\partial_x^2)^{-1/2}u_1\in H^{s+1}$ with $1\leq s\leq [p]+3$ such that $E(u(0))\leq e_0$. If $\|u_0\|_1^2\leq\frac{2(p+2)e_0}{p}$, then the solution $u(t)$ of \eqref{gBE} exists globally in $C^2([0,\infty);H^s)$.  If $\|u_0\|_1^2>\frac{2(p+2)e_0}{p}$, and
		\[
		\left\langle  (-\partial_x^2)^{-1/2}u_0, (-\partial_x^2)^{-1/2}u_1 \right\rangle+
		\kappa\left\langle   (u_0)_x,  (u_1)_x \right\rangle+
		\left\langle   u_0,  u_1 \right\rangle\geq0
		\]
		when $E(u(0))=e_0$, then the solution $u(t)$ of \eqref{gBE} blows up in finite time.
	\end{theorem}

	\section{Well-posedness in the Bessel potential spaces}\label{beseel-space-section}
	In this section, we extend the results of Section \ref{section-evolution} to the inhomogeneous Bessel potential spaces. These spaces are denoted by $H^{s,q}$ as the completion of Schwartz class with respect to the norm $\|u\|_{s,q}=\|u\|_{H^{s,q}(\rr)}=\|\Lambda^s u\|_{L^q(\rr)}$, where $\Lambda^s$ is the Fourier multiplier with the symbol $\scal{\xi}=(1+|\xi|^2)^{1/2}$.
	We use the following version of the Mihlin multiplier theorem.
	\begin{lemma}\label{berstein}\cite{ddy}
		Let $L\in\mathbb{N}$ and $L>n/2$.
		Assume that $D=(-\Delta)^{\frac12}$ and $u$ is a measurable function such that $u\in C^L(\rr^n\setminus\{0\})$.
		\begin{enumerate}[(I)]
			\item If   there exists a constant $\sigma_1>0$ such that for all $|\gamma|\leq L$
			\[
			|D^\gamma u(\xi)|\leq A_\gamma(1+|\xi|)^{-\sigma_1-|\gamma|},\quad\xi\in\rr^n,
			\]
			then $u$ is a $L^q$-multiplier for any $1\leq q<\infty$. If $\sigma_1=0$, then $u$ is a $L^q$-multiplier for any $1< q<\infty$.
			\item If    there exist constants $\sigma_2,\sigma_3>0$ such that for all $|\gamma|\leq L$
			\[
			|D^\gamma u(\xi)|\leq A_\gamma\min(|\xi|^{\sigma_2-|\gamma|},|\xi|^{-\sigma_3-|\gamma|}),\quad\xi\in\rr^n\setminus\{0\},
			\]
			then $u$ is a $L^q$-multiplier for any $1\leq q<\infty$.
		\end{enumerate}
	\end{lemma}
	\begin{lemma}\label{group-est}
		Let $\al,\kappa>0$, $1\leq q<\infty$ and $s\in\rr$. Then,
		\[
		\|S(t)u\|_{s,q}\lesssim\la t\ra^{2|\frac12-\frac1q|} \| u\|_{s,q},
		\qquad
		\|\partial_tS(t)u\|_{s,q}\lesssim\la t\ra^{2|\frac12-\frac1q|} \| u\|_{s,q}
		\]
		The above inequalities hold also if $\al=\kappa=0$.
	\end{lemma}
	\begin{proof}
		We prove the first inequality. The second one can similarly be proved.
		Define $$m(\xi)=\cos\left(\frac{t|\xi|\sqrt{1+\al\xi^2}}{\sqrt{1+\xi^2+\kappa\xi^4}}\right)-\cos(t).$$
		It is straightforward to see that $m$ satisfies Lemma \ref{berstein}, so $m$ is a $ L^1$ multiplier. Moreover,
		\[
		\|(m\hat{u})^\vee\|_{L^1(\rr)}\lesssim \|m\|_\lt^{1/2}\|m'\|_\lt^{1/2}\|  u\|_{L^1(\rr)}.
		\]
		Some computations show that
		\[
		|m(\xi)|\lesssim|t|\la\xi\ra^{-3},\qquad |m'(\xi)|\lesssim|t|\la\xi\ra^{-3},
		\]
		whence we get for all $t$ that
		\[
		\|(m\hat{u})^\vee\|_{L^1(\rr)}\lesssim|t|\|u\|_{L^1(\rr)}.
		\]
		Hence, we derive
		\[
		\|S(t)u\|_{L^1(\rr)}\lesssim\la t\ra\|u\|_{L^1(\rr)}.
		\]
		Finally, the Plancherel identity, an interpolation, and a duality argument show that
		\[
		\|S(t)u\|_{L^q(\rr)}\lesssim\la t\ra^{2|\frac12-\frac1q|} \| u\|_{L^q(\rr)}.
		\]
		The proof is completed by the fact that the operator $\Lambda$ commutes with $S(t)$. In the case $\al=\kappa=0$, one can observe simply that the above arguments still are valid for $S(t)$ and $\partial_t S(t)$.
	\end{proof}
	
	\begin{lemma}\label{multilinear}
		If $f\in C^1(\rr)$, $u\in H^{s,q}(\rr)$ and $s\geq\max\{0,\frac{1}{q}-\frac{1}{p+1}\}$, then
		\[
		\|\lc_R^{-1}\partial_x^{2}f(u)\|_{s,q}\lesssim\|u\|_{s,q}^{p+1}.
		\]	
	\end{lemma}
	
	\begin{proof}
		First, we assume that $q\geq p+1$. By rewriting  $\lc_R^{-1}\partial_x^2f(u)$ by $\lc_R^{-1}\partial_x^2\Lambda^{\varsigma}(\Lambda^{-\varsigma}f(u))$, we notice for  $\varsigma=\frac pq$ that $\lc_R^{-1}\partial_x^2\Lambda^{\varsigma}$ is $L^q$-multiplier, with $1\leq q<\infty$  from Lemma \ref{berstein}. Hence, by applying the Hardy-Littlewood-Sobolev inequality and the fractional Chain rule (see e.g. \cite{linaresponce}), we obtain that
		\[
		\begin{split}
			\|\lc_R^{-1}\partial_x^2f(u)\|_{s,q}&=
			\|\lc_R^{-1}\partial_x^2\Lambda^{\varsigma}(\Lambda^{-\varsigma}f(u))\|_{s,q}\\&
			\lesssim
			\|\Lambda^{s-\varsigma}f(u)\|_{L^q(\rr)}\\&
			\lesssim
			\|\Lambda^{s}f(u)\|_{L^\frac{p}{q}(\rr)}\\
			&\lesssim
			\| f'(u)\|_{L^\frac{p}{q}(\rr)}		\| u\|_{s,q}	\\
			&\lesssim
			\| u\|_{L^q(\rr)}^p		\| u\|_{s,q}\\
			&\lesssim
			\|u\|_{s,q}^{p+1},
		\end{split}	
		\]
		where in the last inequality we have used the inequality
		$\| u\|_{L^q(\rr)}  \lesssim	\|u\|_{s,q}$.
		
		Next, we assume that $1\leq q<p+1$ and $s\leq\frac{1}{q}$. In this case, we observe for any $\varsigma<2$, similar to previous case, that $\lc_R^{-1}\partial_x^2\Lambda^{\varsigma}$ is   $L^q$-multiplier, with $1\leq q<\infty$  from Lemma \ref{berstein}. Then, by applying the Hardy-Littlewood-Sobolev inequality  and the Sobolev embedding we have
		\[
		\begin{split}
			\|\lc_R^{-1}\partial_x^2\Lambda^{\varsigma}(\Lambda^{-\varsigma}f(u))\|_{s,q}
			&\lesssim
			\| \Lambda^{s-\varsigma}f(u) \|_{L^q(\rr)}\\&
			\lesssim
			\| f(u) \|_{L^1(\rr)}\\&
			\lesssim \|u\|_{L^q(\rr)}^{p+1}\\&
			\lesssim
			\|u\|_{s,q}^{p+1}.
		\end{split}
		\]
		
		Finally assume that $s>\frac1{q}$, then $H^{s,q}(\rr)$ is an algebra. Lemma \ref{berstein} shows that  $\lc_R^{-1}\partial_x^2$ is a $L^q$-multiplier   with $1\leq q<\infty$. Then
		\[
		\|\lc_R^{-1}\partial_x^2 f(u) \|_{s,q}
		\lesssim
		\| f(u))\|_{s,q}\lesssim \|u\|_{s,q}^{p+1}.
		\]
		
	\end{proof}

	\begin{theorem}\label{local-bessel}
		Let $f\in C^1(\rr)$, $u_0,u_1\in H^{s,q}(\rr)$ with $1\leq q<\infty$ and $s\geq\max(0,\frac1q-\frac1{p+1})$, then  there exist  $T>0$ and  a unique solution $u\in C^2(0,T;H^{s,q}(\rr))$ of \eqref{gBE}. Moreover, if $f$ is a polynomial, then the flow map $(u_0,u_1)\mapsto u(t)$ is real analytic from $H^{s,q}(\rr)\times H^{s,q}(\rr)$ to $C^2(0,T;H^{s,q}(\rr))$.
	\end{theorem}
	\begin{proof}
		The proof comes from the classical  Banach fixed point theorem by defining the ball
		\[
		B=\left\{u\in X_T,\; \|u\|_{X_T}\leq C\la T\ra^{2|\frac12-\frac1q|}(\|u_0\|_{s,q}+\|u_1\|_{s,q})\right\},
		\]
		where $C$ is the implicit constant of Lemma \ref{group-est}, and $X_T$ is the space of bounded
		functions on $[0,T]$ with values in $H^{s,q}(\rr)$ equipped with norm
		\[
		\|u\|_{X_T}=\sup_{t\in[0,T]}\|u(t)\|_{s,q}.
		\]
		Given $u,v\in B$, Lemmas \ref{group-est} and \ref{multilinear} imply that the right hand side of \eqref{contract-op} is estimated by
		\[
		\begin{split}
			\mathbb{A}(u)&\leq	C\la T\ra^{2|\frac12-\frac1q|}(\|u_0\|_{s,q}+\|u_1\|_{s,q})
			+CT\la T\ra^{2(p+2)|\frac12-\frac1q|}(\|u_0\|_{s,q}+\|u_1\|_{s,q})^{p+1}\\
			&=
			C\la T\ra^{2|\frac12-\frac1q|}(\|u_0\|_{s,q}+\|u_1\|_{s,q})
			\left(1+C'T\la T\ra^{2(p+1)|\frac12-\frac1q|}(\|u_0\|_{s,q}+\|u_1\|_{s,q})^p\right).
		\end{split}
		\]
		Moreover,
		\[
		\|\mathbb{A}(u) -\mathbb{A}(v)\|_{X_T}\leq C''
		T\la T\ra^{2p|\frac12-\frac1q|}(\|u_0\|_{s,q}+\|u_1\|_{s,q})^p\|u-v\|_{X_T}.
		\]	
		Taking $T$ sufficiently small, we see that $\mathbb{A}$ is a contraction, and its unique solution is continuous from $[0,T]$ to $H^{s,q}(\rr)$.
		The analyticity  of the flow map $(u_0,u_1)\mapsto u(t)$ follows from a standard
		argument by the implicit function theorem and the smoothness property of the flow map, deduced from Lemma \ref{multilinear}.

		To see $u_{tt}\in C(0,T;H^{s,q}(\rr))$, we first observe that $m_1=\frac{\al}{\kappa}-\frac{|\xi^2|(1+\al\xi^2)}{1+\xi^2+\kappa\xi^4} $   is $L^q$-multiplier for any $1\leq q<\infty$, so we get from     \eqref{double-t} and Lemma \ref{multilinear} that
		\[
		\|u_{tt}\|_{s,q}\lesssim\|u\|_{s,q}+\|u\|_{s,q}^{p+1}\leq C(t)
		\]
		for any $t\in(0,T)$. The continuity of $u_{tt}$ is shown similar to the proof of Theorem \ref{theo-local-1}. This implies that $(-\partial_x^2)^{-\frac12}u_{tt}\in C(0,T;H^{s+1,q}(\rr))$.
		Indeed, we have from \eqref{double-t} that
		\[
		\frac{1}{|\xi|}\hat{u}_{tt}=-\frac{|\xi|(1+\al\xi^2)}{1+\xi^2+\kappa\xi^4}\hat{u}+
		\frac{|\xi| }{1+\xi^2+\kappa\xi^4}\widehat{f(u)}.
		\]
		Hence we repeat the above argument by using Lemma \ref{multilinear} and Lemma \ref{berstein} to get
		\[
		\|(-\partial_x^2)^{-\frac12}u_{tt}\|_{s,q}
		\lesssim\|u\|_{s+1,q}+\|u\|_{s,q}^{p+1}.
		\]
	\end{proof}

	\section{Asymptotic behavior}\label{section-asymp}
	To obtain the asymptotic behavior of the solutions of \eqref{gBE}, a series of useful lemmas are laid out. The first one is the well-known Van der Corput lemma \cite{stein} as follows.
	\begin{lemma}\label{van-der-corput}
		Let $h$ be either convex or concave on $[a,b]$ with $-\infty\leq a<b\leq+\infty$. If $F$ is a continuously differentiable function on $[a,b]$, then
		\[
		\left|\int_a^bF(\xi)\;\ee^{\ii h(\xi)}\dd\xi\right|\leq 4\left(\min_{\xi\in[a,b]}|h''(\xi)|\right)^{-1/2}\left(|F(b)|+\|F'\|_{L^1[a,b]}\right),
		\]
		if $h''\neq0$ in $[a,b]$.
	\end{lemma}
	\begin{lemma}\label{oscillatory-integral}
		Let  $\eps>0$ be sufficiently small and $N>0$ be sufficiently large. Then there exists $C>0$ such that
		\[
		\sup_{\lam\in\rr}\left|\int_{|\xi|\leq N}\ee^{\ii t h(\xi,\lam)}\dd\xi\right|\leq C(t^{-1/2}\max\{N^{\ell_0/2},\eps^{-\ell/2}\}+ \eps),
		\]
		for all $t>0$, where $\ell=7$ and $\ell_0=6$ if $\kappa=\al$ or $\al=0$, while $\ell=9$ and $\ell_0=4$ if $\kappa\neq\al\neq0$, and $h(\xi,\lam)=\frac{\xi\sqrt{\al\xi^2+1}}{\sqrt{\kappa\xi^4+  \xi^2+1}}+\lam\xi$.
	\end{lemma}
	\begin{proof}
Without loss of generality, we can assume that $\lambda=0$ and denote for simplicity $h(\xi)=h(\xi,\lambda)$.		Let $\xi_0=0$ and $\pm\xi_j$, for $j=0,\cdots,k$ with $k=0,\cdots,4$, be the roots of $h''$, where $\xi_j>0$, for $j=1,2,3,4$.
		Indeed we have
		\[
		h''(\xi):=\frac{\partial^2 h}{\partial\xi^2}=\frac{\xi g(\xi)}
		{(\al\xi  ^2 + 1)^ {3/2}  (\kappa \xi^4 + \xi^2 + 1)^ {5/2}  },
		\]
		where
		\[
		g(\xi)=
		3 \al  (\kappa - \al)\kappa \xi^8 + 2 \kappa (\kappa - \al (5 \al + 1)) \xi^6 -  (\al (\al + 18 \kappa - 1) + \kappa)\xi^4 + 2((\al - 1) \al - 5 \kappa) \xi^2  + 3 (\al - 1)
		\]
		Define $\Gamma=\bigcup_{j=0}^k\Gamma_j$, where
		\[
		\Gamma_j=\{\xi\in\rr;\;\xi_j+\eps<|\xi|<\xi_{j+1}+\eps\},\qquad j=0,\cdots,k-1,
		\]
		and \[
		\Gamma_k=\{\xi\in\rr;\;\xi_k+\eps<|\xi|<N\}.
		\]
		Hence,
		\[
		\min_{\xi\in\Gamma_j}|h''(\xi)|=\min\{|h''(\xi_j+\eps)|,|h''(\xi_{j+1}+\eps)|\}\gtrsim\eps^\ell,\qquad j=0,\cdots,k-1,
		\]
		and
		\[
		\min_{\xi\in\Gamma_k}|h''(\xi)|=\min\{|h''(\xi_k+\eps)|,|h''(N)|\}\gtrsim\min\{\eps^\ell,N^{-\ell_0}\},
		\]
		where
		$\ell=7$ and $\ell_0=6$ if $\kappa=\al$ or $\al=0$, and $\ell=9$ and $\ell_0=4$ if $\kappa\neq\al$ and $\al\neq0$. 	Using Lemma \ref{van-der-corput}, we obtain
		\[
		\left|\int_\Gamma\ee^{\ii th(\xi)}\dd\xi\right|\leq Ct^{-1/2} \max\{\eps^{-\ell/2},N^{\ell_0/2}\}.
		\]
		On the other hand, we have
		\[
		\left|\int_{{|\xi\pm\xi_j|\leq \eps}}\ee^{\ii th(\xi)}\dd\xi\right|\leq C\eps,\qquad j=0,\cdots,k.
		\]
		It follows that
		\[
		\sup_{\lam\in\rr}\left|\int_{|\xi|\leq N}\ee^{\ii t h(\xi,\lam)}\dd\xi\right|\leq C(t^{-1/2}\max\{N^{\ell_0/2},\eps^{-\ell/2}\}+ \eps).
		\]
	\end{proof}

  To get the asymptotic behavior of the solutions of \eqref{gBE}, one can see that  \eqref{gBE} can be also written as the following system of equations
	\begin{equation}\label{system}
		\left\{\begin{array}{lll}
			u_t = v_x,\\ \\
			v_t = \lc_R^{-1}{\left(\lc_D(u) - f(u)\right)}_x,
		\end{array}\right.
	\end{equation}
	where  $\lc_R=I- \partial_x^2+\kappa\partial_x^4$ and
	$
	\mathcal{L}_D= I-\al\partial_x^2.
	$

	The following lemma gives a time estimate of the solutions of the linearized problem. The space $H^s(\rr)\times H^{s+1}(\rr)$ is denoted by $\x_s$.
	
	\begin{lemma}\label{est-lin-oper}
		Let $S(t)$ be the $C_0$ group of unitary operators for the linearized  problem of \eqref{system}
		\[
		\vu_t-
		\left(\begin{array}{lr}
			0&I\\
			\lc_R^{-1}\lc_D&0
		\end{array}\right)
		\vu_x=0,
		\]
		with $\vu(0)=\vdu=(u_0,v_0)$. If $u_0\in L^1(\rr)$, $\vdu\in\x_s$, for $s>1/2$,  and $v_0\in H^{1,1}(\rr)$, then
		\begin{equation}\label{linear-est-0}
			\|S(t)\vdu\|_{L^\infty(\rr)\times L^\infty(\rr)}
			\leq C
			(1+t)^{-\mu}\left(\|\vdu\|_{\x_s}+\|u_0\|_{L^1(\rr)}+\|v_0\|_{H^{1,1}(\rr)}\right),
		\end{equation}
		where $C>0$ is a constant depending only on s,
		\begin{equation}\label{exponent}
			\mu=\left\{\begin{array}{lrlr}
				\frac{1}{2s+\ell-1}&s\geq\frac{3}{2},\\ \\
				\frac{2s-1}{2(2s+\ell-1)}&s\leq\frac{3}{2},\\
			\end{array}\right.
		\end{equation}
		and $\ell$ is the same as in Lemma \ref{oscillatory-integral}.
	\end{lemma}
	\begin{proof}
		Since
		\[
		\vu(t)=S(t)\vdu(x)=\int_\rr\ee^{\ii x\xi}
		\left(\begin{array}{lr}
			\cos\left(\frac{t\xi\varrho(\xi)}{\vtt(\xi)}\right)&
			\frac{\ii\vtt(\xi)}{\varrho(\xi)}\sin\left(\frac{t\xi\varrho(\xi)}{\vtt(\xi)}\right)\\ \\
			\frac{\ii\varrho(\xi)}{\vtt(\xi)}\sin\left(\frac{t\xi\varrho(\xi)}{\vtt(\xi)}\right)
			&\cos\left(\frac{t\xi\varrho(\xi)}{\vtt(\xi)}\right)
		\end{array}\right)\what{\vdu}(\xi)\dd\xi,
		\]
		where $\vtt(\xi)=\sqrt{1+ \xi^2+\kappa\xi^4}$ and $\varrho(\xi)=\sqrt{1+\al\xi^2}$. It is deduced from Fubini's theorem that
		\[
		\begin{split}
			|\vu(t)|&=|S(t)\vdu(x)|
			\lesssim
			\sum\left|\int_\rr\left(\what{u_0}\pm\frac{\vtt(\xi)}{\varrho(\xi)}\widehat{{v_0}}\right)\ee^{\ii th(\xi,\pm x/t)}\dd\xi\right|
			+
			\sum\left|\int_\rr\left(\what{v_0}\pm\frac{\varrho(\xi)}{\vtt(\xi)}\widehat{{u_0}}\right)\ee^{\ii th(\xi,\pm x/t)}\dd\xi\right|
			\\
			&
			\lesssim
			\sum\left|\int_\rr\left(\what{u_0}\pm \widehat{\Lambda^1{v_0}}\right)\ee^{\ii th(\xi,\pm x/t)}\dd\xi\right|
			+
			\sum\left|\int_\rr\left(\what{v_0}\pm \widehat{\Lambda^{-1}{u_0}}\right)\ee^{\ii th(\xi,\pm x/t)}\dd\xi\right|,
		\end{split}\]
		where the sums are over all two sign combinations. It suffices to estimate the first term of the right-hand side of the last inequality. The second term is estimates analogously.
		
		By using the Cauchy-Schwarz inequality, it yields from  Lemma  \ref{oscillatory-integral}  that
		\begin{equation}\begin{split}
				&\left|\int_\rr(\what{u}_0+\sqrt{1+\xi^2}\what{v}_0)\ee^{\ii th(\xi,\pm x/t)}\dd\xi\right|\\
				&\quad\lesssim
				\left|\int_{|\xi|>N}(\what{u}_0+\sqrt{1+\xi^2}\what{v}_0)\dd\xi\right|
				\quad+
				\left|\int_{|\xi|\leq N}(\what{u}_0+\sqrt{1+\xi^2}\what{v}_0)\ee^{\ii th(\xi,\pm x/t)}\dd\xi\right|\\
				&\quad\lesssim\|\vdu\|_{\x_s}N^{-(s-1/2)}+\left(\|u_0\|_{L^1(\rr)}+\|v_0\|_{H^{1,1}(\rr)}\right)\left(\eps
				+t^{-1/2}\max\{\eps^{-\ell/2},N^{\kappa/2}\}\right).
		\end{split}\end{equation}
		For $t>1$, choosing $\eps=N^{-1}=t^{-\th}$, $\th>0$, we obtain that
		\[\begin{split}
			&\left|\int_\rr(\what{u}_0+\sqrt{1+\xi^2}\what{v}_0)\ee^{\ii th(\xi,\pm x/t)}\dd\xi\right|\\
			&\qquad\leq C\left(\|\vdu\|_{\x_s}+\|u_0\|_{L^1(\rr)}+\|v_0\|_{H^{1,1}(\rr)}\right)\left(t^{-\th(s-1/2)}+t^{-\th}+t^{-1/2(1-\ell\th)}\right).
		\end{split}\]
		Hence, for $\th=1/(2s+\ell-1)$ and some $C>0$, it is concluded that
		\begin{equation}\label{linear-est-1}
			\begin{split}
				|\vu(t)|&
				\leq C
				\left(\|\vdu\|_{\x_s}+\|u_0\|_{L^1(\rr)}+\|v_0\|_{H^{1,1}(\rr)}\right)t^{-\mu},
			\end{split}
		\end{equation}
		for $t>1$. On the other hand,  we have for $|t|\leq 1$ that
		\begin{equation}\label{linear-est-2}
			|\vu(t)|\leq C\|\vdu\|_{\x_s},
		\end{equation}
		for $s>1/2$. Combining \eqref{linear-est-1} for $|t|>1$ with \eqref{linear-est-2} for $|t|\leq1$, the proof of the estimate \eqref{linear-est-0} is complete.
	\end{proof}

	Following \cite{liu}, now for $\delta>0  $ if we define  the metric space
	\[
	X=\{\vec u\in C(0,\infty,\x_s)\;\|\vec u\|_X\leq\delta C\}
	\]
	with the norm
	\[
	\|\vec u\|_X=\sup_t((1+t)^\mu\|\vec u(t)\|_{L^\infty(\rr)\times L^\infty(\rr)}+\|\vec u(t)\|_{\x_s}),
	\]
	then, by choosing $C>0$ appropriately in terms of the constant of Lemma \ref{est-lin-oper}, we can obtain the following estimate. We omit the details.
	\begin{theorem}\label{global-existence}
		Let  $m\geq s > 1/2$  with $m$ an integer, and $f\in C^m(\rr)$. Suppose $p\mu>1$, where   $\mu$ is defined in \eqref{exponent}. Then there exists $\delta>0$ such that for any $\vdu=(u_0,v_0)\in\x_s$, $u_0\in L^1(\rr)$ and $v_0\in H^{1,1}(\rr)$, satisfying
		\[
		\|\vdu\|_{\x_s}+\|u_0\|_{L^1(\rr)}+\|v_0\|_{H^{1,1}(\rr)}<\delta,
		\]
		there exists a unique solution $\vu=(u,v)\in C^1([0,+\infty);\x_s)$ of \eqref{system} with $\vu(0)=\vdu$ such that
		\[
		\sup_{0\leq t<+\infty}\left((1+t)^\mu\|u(t)\|_{L^\infty(\rr)}+\|\vu(t)\|_{\x_s}\right)\leq C\delta,
		\]
		where the constant $C$ only depends on $f$, $\|\vdu\|_{\x_s}$, $\|u_0\|_{L^1(\rr)}$ and $\|v_0\|_{H^{1,1}(\rr)}$.
		Moreover, $E$ and $Q$   are conserved for all $t\geq0$.
	\end{theorem}

Now our aim is to investigate the time asymptotic behavior of solutions of \eqref{gBE} in the Sobolev spaces. Recalling Theorem \ref{global-sln}, our main result reads as follows.

\begin{theorem}
	Assume that $F(u)\geq0$ for all $u\in\rr$.	 Let $s>1$  and $u_0,u_1\in H^s(\rr)$,
	then the solution $u(t)$ of \eqref{gBE} satisfies
	\[
	\|u(T)\|_s\leq CT^{\frac{2}{2-\varrho_s}}
	\]
	for any $T>0$, where $C=C(s,\|u_0\|_s,\|u_1\|_s)>0$ and
	\[
	\varrho_s=\begin{cases}
		0,&s\in(1,\frac32),\\
		\frac{2(s-1)}{s-\frac12}^+,&s\in(\frac32,2),\\
		s-1,&s\geq2.
	\end{cases}
	\]
\end{theorem}
\begin{proof}
	We introduce $w=u+\ii B^{-1}u_t$, where $B=\sqrt{\frac{-\partial_x^2\lc_D}{\lc_R}}$. Then, \eqref{gBE} is transformed into
	\begin{equation}\label{trans=eq}
		\begin{cases}
			\ii w_t-Bw-  Mf(\Re(w))=0,\\
			w(0)=w_0=u_0+B^{-1}u_1,
		\end{cases}
	\end{equation}
	where
	\[
	M=\sqrt{\frac{-\partial_x^2}{\lc_R\lc_D}}.
	\]
	It is clear that $u=\Re(w)$ and $B^{-1}u_t=\Im(w)$. Moreover, $M$ can be considered as a Fourier multiplier with the symbol 
 \[
\hat{M}=\sqrt{\frac{\xi^2}{(1+\xi^2+\kappa\xi^4)(1+\al\xi^2)}}.
 \]
 
 By the Duhamel principle, \eqref{trans=eq} can be written as following integral equation,
	
	\begin{equation}\label{integral-form-2}
		w(t)=\ee^{-\ii tB}w_0-\ii \int_0^t\ee^{-\ii(t-\tau)B}Mf(\Re(w))\dd\tau.
	\end{equation}

	For simplicity, we can assume that $f(u)=|u|^pu$.
	We observe that this transformation gives a correspondence between the solutions $(u,(-\partial_x^2)^{-1/2}u_t)\in H^s\times H^{s+1}$ of \eqref{gBE} and $(w,\bar{w})\in H^s\times H^s$ of \eqref{trans=eq}.
	
	Let $D$ be the Fourier multiplier with the symbol $|\xi|$. By multiplying \eqref{trans=eq} by $D^{2s}\bar{w}$ and taking the image part, we get
	
	\begin{equation}\label{mult-ds}
		\frac{\dd }{\dd t}\|D^sw\|_{\lt}+\Im\la D^{2s}\bar{w}, M|\Re(w)|^p\Re(w)\ra=0.
	\end{equation}
	
	Let $s-1=m+\nu$, where $m\geq0$ and $\nu\in[0,1)$. Then,  we obtain by the Plancherel identity and interpolation  for $g=|\Re(w)|$ that
	\[
	\begin{split}
		\Im\la D^{2s}\bar{w}, M|\Re(w)|^p\Re(w)\ra
		&=
		\Im\la D^{s-1}\bar{w}, D^{s-1}|\Re(w)|^p\Re(w)\ra\\&
		\leq
		\Im\la D^{s-1}\bar{w}, D^{s-1}g^{p+1}\ra\\&
		\lesssim
		\|w\|_{s-1}\|D^{s-1}g^{p+1}\|_\lt\\&
		\lesssim
		\|w\|_1^{\frac{1}{s-1}}\|w\|_s^{\max\{0,\frac{s-2}{s-1}\}}\|D^{s-1}g^{p+1}\|_\lt.
	\end{split}
	\]
	By the expansion
	\[
	D^{s-1}g^{p+1}\cong D^\nu g^p\partial_x^m g+D^\nu\sum_{\substack{j_1+\cdots+j_{p+1}=m,\\ 0\leq j_1,\cdots, j_{p+1}\leq m-1}}c_{j_1,\cdots,j_{p+1}}\partial_{x}^{j_1}g\cdots\partial_{x}^{j_{p+1}}g
	=:I_1+I_2,
	\]
	we write
	\[
	I_1=g^pD^\nu\partial_{x}^m g+[D^\nu,g^p]\partial_{x}^m g=:I_3+I_4,
	\]
	where $[\cdot,\cdot]$ is the commutator. Hence, an interpolation reveals again that
	\[
	\|I_3\|_\lt\lesssim \|g\|_1^{p+\frac{1}{s-1}}\|g\|_s^{\max\{0,\frac{s-2}{s-1}\}}.
	\]
	Noting that $\|g\|_s\leq\|w\|_s$ for any $s\geq0$, we obtain that
	\[
	\|I_3\|_\lt\lesssim \|w\|_1^{p+\frac{1}{s-1}}\|w\|_s^{\max\{0,\frac{s-2}{s-1}\}}.
	\]
	To estimate $I_4$, we use the commutator estimate in Theorem A.12 of \cite{kpv} to obtain that
	\[
	\|I_4\|_\lt\lesssim \|\partial_{x}^mg\|_\lt\|D^\nu g^p\|_{L^\infty(\rr)}.
	\]
	By using the following pointwise inequality for the fractional
	derivatives \cite{coma}
	\[
	D^\nu\phi(g)\leq \phi'(g)D^\nu g,\qquad\text{any convex function }\phi,
	\]
	we deduce from some interpolations and the Sobolev inequality that
	\[
	\begin{split}
		\|I_4\|_\lt&\lesssim \|\partial_{x}^mg\|_\lt
		\|g\|_1^{p-1}\|D^\nu g\|_{L^\infty(\rr)}\\&
		\lesssim \|g\|_m\|g\|_1^{p-1}\|D^\nu g\|_{L^\infty(\rr)}\\&
		\lesssim\|g\|_1^{\frac{1+\al}{s-1}}\|g\|_s^{\frac{s-\al}{s-1}} \|g\|_1^{p-1} \|g\|_{\nu+\frac12 ^+}\\&\leq C({\|w\|_1})\|w\|_s^{\max\left\{0,\frac{\nu-\frac12^+}{s-1}\right\}}
	\end{split}
	\]
	where $C(\cdot)$ is an increasing function. On the other hand, we have by several interpolations that
	\[
	\begin{split}
		\|I_2\|_\lt&\lesssim
		\sum_{\substack{j_1+\cdots+j_{p+1}=m,\\ 0\leq j_1,\cdots, j_{p+1}\leq m-1}}c_{j_1,\cdots,j_{p+1}}
		\left\|D^\nu\left(\prod_{k=1}^{p+1}\partial_x^{j_k}g\right)\right\|_\lt\\&
		\lesssim\sum_{\substack{j_1+\cdots+j_{p+1}=m,\\ 0\leq j_1,\cdots, j_{p+1}\leq m-1}}c_{j_1,\cdots,j_{p+1}}\sum_{k=1}^{p+1}
		\|\partial_x^{j_k}g\|_{H^{\nu,2(p+1)}}\prod_{i\neq k}\|\partial_x^{j_i}g\|_{2(p+1)}\\
		&
		\lesssim \sum_{\substack{j_1+\cdots+j_{p+1}=m,\\ 0\leq j_1,\cdots, j_{p+1}\leq m-1}}
		\sum_{k=1}^{p+1}\|g\|_{H^{\nu+j_k,2(p+1)}}\prod_{i\neq k}\|g\|_{j_i+\frac{p}{2(p+1)}}\\
		&
		\lesssim \sum_{\substack{j_1+\cdots+j_{p+1}=m,\\ 0\leq j_1,\cdots, j_{p+1}\leq m-1}}
		\sum_{k=1}^{p+1}\|g\|_{\nu+j_k+\frac{p}{2(p+1)}}\prod_{i\neq k}\|g\|_{j_i+\frac{p}{2(p+1)}}.
	\end{split}
	\]
	Now by using the estimates
	\[
	\|g\|_{j_i+\frac{p}{2(p+1)}}\lesssim\|g\|_1^{1-\theta_{j_i}}
	\|g\|_s^{\theta_{j_i}}
	\]
	and
	\[
	\|g\|_{\nu+j_k+\frac{p}{2(p+1)}}\lesssim\|g\|_1^{1-\theta_{\theta_{j_k}}}\|g\|_s^
	{\theta_{j_k}},
	\]
	where
	\[
	\theta_{j_i}=\frac{j_i+\frac{p}{2(p+1)}-1}{s-1},\qquad\theta_{j_k}=\frac{\nu+j_k+\frac{p}{2(p+1)}-1}{s-1},
	\]
	we conclude that
	\[
	\begin{split}
		\|I_4\|_\lt
		&\lesssim
		\sum_{\substack{j_1+\cdots+j_{p+1}=m,\\ 0\leq j_1,\cdots, j_{p+1}\leq m-1}}c_{j_1,\cdots,j_{p+1}}
		C(\|g\|_1)\|g\|_s^\theta\\
		&\leq C(\|w\|_1)\|w\|_s^{\max\left\{0,\frac{2s-p}{2(s-1)}\right\}},
	\end{split}
	\]
	where
	\[
	\theta=\frac{\nu+j_k-1+\frac{p}{2(p+1)}-\frac{(p+1)^2-1}{2(p+1)}+\sum_{i\neq k}j_i}{s-1}
	\]
	and $C(\cdot)$ is an increasing function. After calculating the maximum value of the above exponents, we obtain that
	
	\begin{equation}\label{nonl-est}
		\Im\la D^{2s}\bar{w}, M|\Re(w)|^p\Re(w)\ra
		\leq C(\|w\|_1)\|w\|_s^{\varrho_s}.
	\end{equation}
	Inserting \eqref{nonl-est} into \eqref{mult-ds} and integrating over $[T,T+1]$, we obtain that
	\[
	\|D^s w(T+1)\|_\lt^2\leq\|D^sw(T)\|_\lt^2+C(\|w(T)\|_1)\|w(T)\|_s^{\varrho_s}.
	\]
	Since $\|u(T)\|_1$ is bounded independent of $T$, then
	\[
	\|w(T+1)\|_s^2\leq \|w(T)\|_s^2+C_{\|w_0\|_s}\|w(T)\|_s^{\varrho_s}
	\]
	for all $T\geq0$. Hence, we get
	\[
	\|w(T)\|_s\lesssim\la T\ra^{\frac{2}{2-\varrho_s}}.
	\]
	The proof is complete by recalling $u=\Re(w)$.
\end{proof}

\section{Modulation spaces}\label{modul-sec}
In this section, we study the well-posedness in   the Modulation spaces $M_{p,q}^s$ that is larger than the Bessel potential spaces. 

Before presenting our main findings, we revisit some preliminary definitions and notations pertinent to the modulation spaces $M_{p,q}^s(\mathbb{R}^n)$, as outlined in \cite{banq,Kato,BaoxHudz}. We denote by $\mathscr{S}$ the Schwartz space on $\mathbb{R}^n$ and $\mathscr{S}^\prime$ its dual space.  Let $Q_0=\{\xi\in\mathbb{R}^n,\;\xi_i\in [1/2,1/2),\ i=1,...,n\}$ and $Q_k=k+Q_0,$ $k\in \mathbb{Z}^n.$ Thus, $\{Q_k\}_{k\in \mathbb{Z}^n}$ constitutes a decomposition of $\mathbb{R}^n,$ that is, $Q_i\cap Q_j=\emptyset$ and $\bigcup_{k\in \mathbb{Z}^n}Q_k=\mathbb{R}^n.$ Let $\rho:\mathbb{R}^n\rightarrow [0,1]$ be a smooth function satisfying $\rho(\xi)=1$ for $\vert \xi\vert\leq \frac{\sqrt{n}}{2}$ and $\rho(\xi)=0$ for $\vert \xi\vert\geq \sqrt{n}.$ Set $\rho_k(\xi)=\rho(\xi-k),$ $k\in\mathbb{Z}^n,$ a translation of $\rho.$ It is clear to see that $\rho_k(\xi)=1$ in $Q_k,$ and thus, $\sum_{k\in\mathbb{Z}^n}\rho_k(\xi)\geq 1$ for all $\xi\in \mathbb{R}^n.$ Let
$$\sigma_k(\xi)=\rho_k(\xi)\left( \sum_{l\in\mathbb{Z}^n}\rho_l(\xi)\right)^{-1},\ k\in\mathbb{Z}^n.$$
The sequence $\{\sigma_k(\xi)\}_{k\in\mathbb{Z}^n}$ possesses the following properties:
\begin{eqnarray*}
&\vert \sigma_k(\xi)\vert\geq C,\ \text{for all}\ \xi\in Q_k, \qquad
\mbox{supp}(\sigma_k)\subset\{\xi: \vert \xi-k\vert_\infty\leq \sqrt{n}\},&\\
&\sum\limits_{k\in\mathbb{Z}^n}\sigma_k(\xi)=1, \qquad
 \vert D^\alpha\sigma_k(\xi)\vert\leq C_{m}\ \ \text{for all}\ \xi\in \mathbb{R}^n,\ \vert\alpha\vert\leq m.&
\end{eqnarray*}
We consider the frequency-uniform decomposition operators $\square_k:=\mathcal{F}^{-1}\sigma_k\mathcal{F}=\mathcal{F}^{-1}[\sigma_k\cdot\mathcal{F}],$ $k\in\mathbb{Z}^n,$ where $\mathcal{F}$ is the Fourier transform. The modulations spaces $M^s_{p,q}=M^s_{p,q}(\mathbb{R}^n)$ are defined for $s\in\mathbb{R}$ and $1\leq p,q\leq\infty,$ as  
\begin{eqnarray*}
M_{p,q}^s:=\left\{ f\in \mathscr{S}^{\prime}(\mathbb{R}^n):\ \Vert f\Vert_{M^s_{p,q}}<\infty\right\},
\end{eqnarray*}
where
\[
\Vert f\Vert_{M^s_{p,q}}=
\begin{cases}
 \norm{\scal{k}^{s}\Vert \square_k f\Vert_{L^p}}_{l^q_k},\ \ &\ 1\leq q<\infty,\\
 \\
\sup\limits_{k\in\mathbb{Z}^n}\scal{k}^{s}\Vert \square_kf\Vert_{L^p}, \ \ &\ q=\infty.
\end{cases}
\]
We will for simplicity write $M^0_{p,q}(\mathbb{R})=M_{p,q}(\mathbb{R}).$

The following properties of these space are notable:
\begin{enumerate}[(i)]
 \item If $\Omega\Subset\mathbb{R}^n,$ then $\mathscr{S}^\Omega=\{f:f\in \mathscr{S}\ \mbox{and}\ \mbox{supp}(\widehat{f})\subset \Omega\}$ is dense in $M^s_{p,q},$ $s\in\mathbb{R},$ $0<p,q<\infty.$
\item  $M^{s_1}_{p_1,q_1}\subset M^{s_2}_{p_2,q_2},\ \mbox{if}\ s_1\geq s_2,\ 0<p_1\leq p_2,\ 0<q_1\leq q_2.$
\item  $M^{s_1}_{p,q_1}\subset M^{s_2}_{p,q_2},\ \mbox{if}\ q_1>q_2,\ s_1>s_2,\ s_1-s_2>n/q_2-n/q_1.$
\item   $M_{p,1}\subset L^\infty\cap L^p,$  for $1<p\leq \infty.$ 
\item   $B^{s+n/q}_{p,q}\subset M^s_{p,q},$ for $0<p,q\leq \infty$ and $s\in\mathbb{R}.$ 
\item   $B^{s_1}_{p,q}\subset M^{s_2}_{p,q}$ if and only if $s_1\geq s_2+n\nu_1(p,q).$
\item    if  $s_1> s_2+n\nu_1(p,q),$, then
\[
H^{s_1,p}\subset M^{s_2}_{p,q},
\]
where 
\[
\nu_1(p,q)= 
\begin{cases} 
0,&\ \  \ (\frac{1}{p},\frac{1}{q})\in \left\{(\frac{1}{p},\frac{1}{q})\in [0,\infty)^2\ \ :\ \frac{1}{q}\leq \frac{1}{p}\ \mbox{and}\ \frac{1}{q}\leq 1-\frac{1}{p}\right\},\\
\\
\frac{1}{q}+\frac{1}{p}-1,&\ \  \ (\frac{1}{p},\frac{1}{q})\in \left\{(\frac{1}{p},\frac{1}{q})\in [0,\infty)^2\ \ :\ \frac{1}{p}\geq\frac{1}{2}\ \mbox{and}\ \frac{1}{p}\geq 1-\frac{1}{q}\right\}, \\
\\
-\frac{1}{p}+\frac{1}{q},&\ \  \ (\frac{1}{p},\frac{1}{q})\in \left\{(\frac{1}{p},\frac{1}{q})\in [0,\infty)^2\ \ :\ \frac{1}{p}\leq\frac{1}{2}\ \mbox{and}\ \frac{1}{q}\geq \frac{1}{p}\right\}.
\end{cases}
\]
\end{enumerate}

\begin{lemma}\label{Ib2}
Let $p $ be a natural number, $1\leq q \leq \infty, $   $0\leq  s <n/{\nu},$ and $1\leq \mu, \nu <\infty$ such that
\[ \frac{1}{\nu}-\frac{(p-1)s}{n}\leq \frac{p}{\mu}-p+1, \ \ \ \ \ 1\leq \nu\leq \mu.\]
Then,  
$  \|u^p\|_{M^{s}_{{q},\mu}}\lesssim  \|u\|^p_{M^{s}_{{pq},\nu}}.$
\end{lemma}
 \begin{lemma} \label{Ib1}
 Let $1\leq p,p_1,p_2\leq \infty,$ $1<\sigma,\sigma_1,\sigma_2<\infty.$   If $\frac{1}{p}=\frac{1}{p_1}+\frac{1}{p_2},$ $\frac{1}{\sigma}-\frac{1}{\sigma_1}-\frac{1}{\sigma_2}+1\leq \frac{s}{n}<\frac{1}{\sigma},$ then   
 \[
 \|u v\|_{M^{s}_{p,\sigma}}\lesssim  \|u\|_{M^{s}_{p_1,\sigma_1}}\|v\|_{M^{s}_{p_2,\sigma_2}} 
 \]
 for any $u\in M^s_{p_1,\sigma_1}(\mathbb{R}^n)$ and $v\in M^s_{p_2,\sigma_2}(\mathbb{R}^n).$
 \end{lemma}
\begin{lemma}\label{prod}  Let $m$ be a positive integer, $s\geq 0 $, $0<p_i\leq \infty,$ $1\leq q_i\leq \infty$ with $1\leq i\leq m$ such that $\frac{1}{p_1}+\cdots +\frac{1}{p_m}=\frac{1}{p_0},$ $\frac{1}{q_1}+\cdots +\frac{1}{q_m}=m-1+\frac{1}{q_0}$.  Then,  
\[ 
\left\| \prod_{i=1}^mu_i\right\|_{M^{s}_{{p_0},q_0}}\lesssim \prod_{i=1}^m \|u_i\|^p_{M^{s}_{{p_i},q_i}}.
\]
\end{lemma}

Since $h''(\xi)>0$ and $h''(\xi)\to0$ as $|\xi|\to\infty$, then we can modify the ideas of \cite{linares} to obtain more dispersive estimates (see Lemma \ref{group-est}).

 \begin{lemma}\label{LemmHsigmap}
Let $\mathfrak{S} = \mathrm{e}^{\mathrm{i} t B}$, $\sigma<-1$, $N>2$, $\epsilon>0$, $2\leq q\leq\infty$, and $1/q+1/q'=1$. Then,  
\[
\|\ss(t)f\|_{\sigma,q}\lesssim|t|^{-2(\frac1q-\frac12)\vartheta_\sigma}\|f\|_{L^{q'}}
\]
for all $f\in L^{q'}$ and $t\neq0$, where
 \begin{equation}\label{beta-sigma}
     \vartheta_\sigma=\begin{cases}
    -\frac{\ell_0}{\ell}\theta,&\sigma<-1-\frac{\ell_0}{\ell},\\
    \frac{\ell_0}{2}\theta-\frac12,& -1>\sigma>-1-\frac{\ell_0}{\ell},
\end{cases}
 \end{equation}
and $\theta=\frac{1}{\ell_0-2\sigma-2}$.
 \end{lemma}
\begin{proof}
    Let $K(t)=\Lambda^\sigma \ss(t)$. We first show that
   \begin{equation}\label{os-0}
     \|K(t)f\|_{L^\infty}\lesssim
 g(t,\epsilon,N,\sigma)\|f\|_{L^1}, 
   \end{equation}
   where
   \[
  g(t,\epsilon,N,\sigma)
  =\epsilon+
  (2-\sigma)|t|^{-\frac12}\epsilon^{-\frac\ell2}
  +
  |t|^{-\frac12}
       \max\sett{\epsilon^{-\frac\ell2}, N^{\frac{\ell_0}{2}}}\left(N^\sigma+(\epsilon+|\xi_k|)^\sigma\right)
       -\frac{N^{\sigma+1}}{1+\sigma}.
   \]
        Using the same notations and the proof of Lemma \ref{oscillatory-integral}, we have
   \begin{equation}\label{os-1}
       \abso{\int_{|\xi\pm\xi_j|\leq\epsilon}(1+|\xi|^2)^{\frac\sigma2}\ee^{\ii th}\hat{f}(\xi)\dd\xi}\lesssim\epsilon\|f\|_{L^1}.
   \end{equation}
   On the sets $\Gamma_j$ with $j=0,\cdots,k-1$, we obtain
   \begin{equation}\label{os-2}
       \abso{\int_{\Gamma_j}(1+|\xi|^2)^{\frac\sigma2}\ee^{\ii th}\hat{f}(\xi)\dd\xi}\lesssim|t|^{-\frac12}\epsilon^{-\frac\ell2}(2-\sigma)\|f\|_{L^1}.
   \end{equation}
   On the other hand, by using Lemma \ref{van-der-corput}, we get
     \begin{equation}\label{os-3}
       \abso{\int_{\Gamma_k}(1+|\xi|^2)^{\frac\sigma2}\ee^{\ii th}\hat{f}(\xi)\dd\xi}\lesssim|t|^{-\frac12}
       \max\sett{\epsilon^{-\frac\ell2}, N^{\frac{\ell_0}{2}}}\left(N^\sigma+(\epsilon+|\xi_k|)^\sigma\right)
       \|f\|_{L^1}.
   \end{equation}
   Finally, since $\sigma<-1$, then
  \begin{equation}\label{os-4}
       \abso{\int_{|\xi|\geq N}(1+|\xi|^2)^{\frac\sigma2}\ee^{\ii th}\hat{f}(\xi)\dd\xi}\lesssim\int_N^\infty(1+\xi^2)^{\frac\sigma2}\dd \xi
       \|f\|_{L^1}\lesssim-\frac{N^{\sigma+1}}{1+\sigma}\|f\|_{L^2}.
   \end{equation}
   Collecting \eqref{os-1}-\eqref{os-4}, we derive \eqref{os-0}.

   From \eqref{os-0}, taking $\epsilon^{-\frac\ell2}=N^{\frac{\ell_0}2}$, we obtain that
   \[
   \|K(t)f\|_{L^\infty}\lesssim
   N^{-\frac{\ell_0}\ell}+(2-\sigma)|t|^{-\frac12}N^{\frac{\ell_0}{2}}+
   |t|^{-\frac12}N^{\sigma+\frac{\ell_0}{2}}-\frac{N^{\sigma+1}}{\sigma+1}.
   \]
   When $|t|>1$, we set $N=|t|^\theta$, so that
   \[
      \|K(t)f\|_{L^\infty}
      \lesssim
|t|^{-\frac{\ell_0}{\ell}\theta}
+(2-\sigma)|t|^{\frac{\ell_0}{2}\theta-\frac12}+|t|^{(\sigma+1)\theta}.
   \]
   The above inequalities lead us to
   \[
   \|K(t)f\|_{L^\infty}
      \lesssim
      |t|^{\vartheta_\sigma}\|f\|_{L^1}.
   \]
   It is clear that $K(t):L^2\to L^2$ is a continuous operator, thus the Riesz-Thorin interpolation Theorem completes the proof.
\end{proof}

    The following estimates are also valid  for the group $S(t)$.
 
\begin{lemma}\label{GrupoMs} 
Let  $t \in \mathbb{R}$, $\sigma < -1$, $2 \leq p < \infty$, $\frac{1}{p} + \frac{1}{p'} = 1$, $0 < q < \infty$, and $\theta \in [0,1]$. Then we have
\begin{equation}
\| \mathfrak{S}(t) f \|_{M^{s}_{p,q}} \lesssim \scal{t}^{-2\theta(\frac1p-\frac12)\vartheta_{\sigma}} \| f \|_{M^{s-\sigma\theta}_{p',q}}.
\end{equation}
\end{lemma}

\begin{proof} 
From Lemma \ref{LemmHsigmap}, and using the fact that $\mathfrak{S}(t)$ and $\square_k$ commute, we obtain 
\begin{equation}\label{IneSqut}
\| \square_k \mathfrak{S}(t)f \|_{H^{\sigma}_p} \lesssim |t|^{-2\left(\frac 1p-\frac{1}{2}\right)\vartheta_{\sigma}} \| \square_k f \|_{L^{p'}}. 
\end{equation}
Using the Bernstein  multiplier estimate, we have
\begin{equation}\label{IneBerMul}
\norm{\square_k \Lambda^{\delta}g }_{L^p} \lesssim \scal{k}^{\delta} \| g \|_{L^{p}}.
\end{equation}
Then, we obtain from  (\ref{IneSqut}) and (\ref{IneBerMul}) that
\begin{equation}\label{SquSum1}
\| \square_k \mathfrak{S}(t) f \|_{L^p} \lesssim
\scal{k}^{-\sigma}\sum_{l }   \| \square_{k+l} \mathfrak{S}(t) f \|  _{H^{\sigma}_p} \lesssim \scal{k}^{-\sigma} |t|^{-2\left(\frac 1p-\frac12\right)\vartheta_{\sigma}} \sum_{l }  \| \square_{k+l} f \| _{L^{p'}}.
\end{equation}
Moreover,   we obtain from the H\"older and Young  inequalities that
\begin{align}\label{SquSum2}
\| \square_k \mathfrak{S}(t)f \|_{L^p} & \lesssim \| \sigma_k \mathrm{e}^{-\mathrm{i} tB}\widehat{Bf} \|_{L^{p'}} \lesssim \sum_{l }  \| \sigma_k \ee^{-\ii tB}(B\square_{k+l}f)^\wedge \|_{L^{p'}}\notag\\
& \lesssim \sum_{l }\|  (\square_{k+l}f )^\wedge\|_{L^p} \lesssim \sum_{l }\norm{ \square_{k+l}f }_{L^{p'}}.
\end{align}
We get  from   (\ref{SquSum1}) and (\ref{SquSum2}) and an interpolation argument that
\begin{equation}\label{SquTheta}
\|  \square_k \mathfrak{S}(t) f \|_{L^p} \lesssim  |t|^{-2\theta\left(\frac 1p-\frac12\right)\vartheta_{\sigma}} \scal{k}^{-\sigma\theta}\sum_{l } \| \square_{k+l} f \|_{L^{p'}},
\end{equation}
for any $\theta \in [0,1]$. Since $ \sigma<0$, from (\ref{SquSum2}) we have
\begin{align}\label{SquSumThe}
\| \square_k \mathfrak{S}(t) f \|_{L^p} & \lesssim  \scal{k}^{-\sigma\theta} \sum_{l } \|  \square_{k+l}f \|_{L^{p'}}. 
\end{align}
Combining (\ref{SquTheta}) and (\ref{SquSumThe}), we arrive at
\begin{equation}\label{FinIneST}
\| \square_k \mathfrak{S}(t) f \|_{L^p} \lesssim\scal{k}^{-\sigma\theta}\scal{t}^{2\theta\left(\frac 12-\frac{1}{p}\right)\vartheta_{\sigma}} \sum_{l } \|  \square_{k+l} f \|_{L^{p'}}.
\end{equation}
Finally, the proof is complete if we multiply  (\ref{FinIneST}) by $\scal{k} ^{s}$ and   take the $l^p$-norm of the result.
\end{proof}


\begin{lemma}\label{p20}
Let $s\in\mathbb{R}, \sigma<-1,$ $1\leq r \leq\infty,$ $2\leq p<\infty,$ $0<q<\infty,$ and $\theta\in [0,1].$   Then
\begin{equation}
\norm{\ss(t) u_0}_{L^r(I,M^s_{p,q})}\lesssim\norm{u}_{M^{s-\theta\sigma}_{p',q}},
\end{equation}
where $\mathbb{R}\Supset I\ni0$. 
\end{lemma}
\begin{proof}
	The proof follows from  Lemma \ref{GrupoMs} and the continuity of $t \mapsto \scal{t}^{2\theta (\frac 12-\frac 1p)\vartheta_{\sigma}}$   on $I,$. In fact, we get
	\begin{align*}
		\norm{\ss(t) u_0}_{L^r( I,M^s_{p,q})}& \lesssim   \norm {\scal{t}^{2\theta(\frac 12-\frac 1p)\vartheta_{\sigma}} \left\Vert u_0\right\Vert_{M^{s-\theta\sigma}_{p',q}}}  _{L_t^r}\\
		& =\Vert u_0 \Vert_{M^{s-\theta\sigma}_{p',q}} \left\Vert  \scal{t}^{2\theta(\frac 12-\frac 1p)\vartheta_{\sigma}} \right\Vert_{L_t^r(I)} \lesssim \Vert u_0 \Vert_{M^{s-\theta\sigma}_{p',q}},
	\end{align*}
	as desired.
\end{proof}

 \begin{lemma}\label{Otro_GrupoMs} 
Let $ s \in \mathbb{R},$  $0<q<\infty,$  $2\leq p <\infty,$ $\theta\in[0,1]$. Then, we have
\begin{equation}
\Vert \ss(t)M f \Vert_{M^{s}_{p,q}}\lesssim \langle t\rangle^{-2(\frac1p-\frac12)}\Vert f\Vert_{M^{s}_{p,q}}.
\end{equation}
\end{lemma}
\begin{proof}
	We have from Lemma  \ref{group-est},
	\begin{equation}\label{e1}
		\Vert \ss(t) Mf \Vert_{L^{p}}\lesssim \langle t\rangle^{-2(\frac1p-\frac12)}\Vert Mf\Vert_{L^{p}}.
	\end{equation}
	Thus,   we get from (\ref{e1})  that
	\begin{equation}\label{e2}
		\Vert \ss(t) Mf \Vert_{L^{p}}\lesssim \langle t\rangle^{-2(\frac1p-\frac12)}\Vert f\Vert_{L^{p}}.
	\end{equation}
	Since $\square_k\ss(t)M=\ss(t)M\square_k$, we have from (\ref{e2}) that
	\begin{equation}\label{e3z}
		\Vert \square_k\ss(t) Mf \Vert_{L^{p}}\lesssim \langle t\rangle^{-2(\frac1p-\frac12)}\Vert \square_k f\Vert_{L^{p}}.
	\end{equation}
	Multiplying (\ref{e3z}) by $\scal{k}^s$ and taking the $l^q-$norm in both sides of (\ref{e3z}), we obtain the
	desired result.
\end{proof}

\begin{lemma}\label{Isomor}
Let $1<p\leq \infty,$  $0<q\leq \infty,$  $s \in\mathbb{R}.$ Then,
\[
\Vert Mg\Vert_{M_{p,q}^{s}}\lesssim \Vert g\Vert_{M_{p,q}^{s-2}} 
\]
for all $g\in M_{p,q}^{s-2}.$
\end{lemma}
\begin{proof}
Inspired from \cite[Proposition A.1]{Kato}, we first choose a  smooth cutoff function $\chi$ such that $\chi(\xi)=1$ when $|\xi|\leq1$  and $\chi(\xi)=0$ when $|\xi|\geq2$.
We also define
\[\chi_k(\xi):=\chi\left(\frac{\xi-k}{C}\right).\]
Then $\chi_k\equiv1$ on ${\rm supp}(\sigma_k)$, where the constant $C>1$ is in the defintion of $\sigma_k$.  We have from  the Young  inequality and a  change of variable    that
\begin{align*}
\Big\Vert  [\sigma_k  (Mg)^\wedge]^\vee\Big\Vert_{L^p} 
&= \left\Vert \int_{\mathbb{R}} \ee^{\ii x\xi}\chi\left(\frac{\xi}{C}\right) \hat{M}(\xi-k) \dd\xi\right\Vert_{L^1} \Big\Vert \square_k g\Big\Vert_{L^p}.
\end{align*}
Now, since $|\hat{M}(\xi-k)|\lesssim \langle k\rangle^{-2}$ for all $|\xi|\leq C$ and  $k\in \mathbb{Z},$ we obtain
\begin{align*}
 \int_{\vert x \vert \leq 1}\left| \int_{\mathbb{R}} \ee^{\ii x\xi}\chi\left(\frac{\xi}{C}\right) \hat{M}(\xi-k) \dd\xi\right| \dd x \lesssim \langle k\rangle^{-2}.
 \end{align*}
On the other hand, we observe that
 \begin{align*}
 \int_{\vert x \vert \geq  1}\left| \int_{\mathbb{R}} \ee^{\ii x\xi}\chi\left(\frac{\xi}{C}\right) \hat{M}(\xi-k) \dd\xi\right| \dd x\lesssim  \langle k\rangle^{-2}.
 \end{align*}
 Thus,
\[
\Big\Vert [\sigma_k  (Mg)^\wedge]^\vee\Big\Vert_{L^p}\lesssim \langle k\rangle^{-2} \Vert \square_k g\Vert_{L^p}.
\]
By a natural modifications, we have in   the case $q=\infty$ that
 \[
 \Vert Mg\Vert_{M_{p,q}^{s}} \lesssim\left( \sum_{k\in\mathbb{Z}}\langle k\rangle^{(s-2)q}\Vert \square_k g\Vert^q_{L^p} \right)^{1/q} \lesssim  \Vert g\Vert_{M_{p,q}^{s-2}},
 \]
and the proof is complete.
\end{proof}

 \begin{lemma}\label{propuse}
Let $p\geq 1$ an integer, $p_0=p+2,$ $1\leq q <\infty,$ $1-\frac 1q \leq s<\frac 1q,$ $r=\frac{p(p+2)}{p+2+\theta  p \vartheta_{\sigma}},$    and $f(u)=u^{p+1}.$  Also assume that $0<\theta\leq -\frac{2}{\sigma}$ and $\sigma<-1$ such that $0<\frac{-\theta p\vartheta_{\sigma}}{p+2}<1.$ Then
\begin{equation}
\left\Vert \int_0^t \ss(t-\tau)Mf(\Re u(\tau))d\tau\right\Vert_{L^r(\mathbb{R},M^s_{p_0,q})}\lesssim \Vert u\Vert^{p+1}_{L^r(\mathbb{R},M^s_{p_0,q})}.
\end{equation}
\end{lemma}
\begin{proof}
	It can be seen from Lemmas \ref{GrupoMs} and   \ref{Isomor} that
\begin{align*}
\left\Vert \int_0^t \ss(t-\tau)Mf(\Re u(\tau))d\tau\right\Vert_{L^r(\mathbb{R},M^s_{p_0,q})}&
\lesssim\left\Vert  \int_0^t \left\Vert \ss(t-\tau)Mf(\Re u(\tau))\right\Vert_{M^s_{p_0,q}} d\tau\right\Vert_{L_t^r}\\
& \lesssim \left\Vert  \int_0^t \scal{t-\tau}^{2\theta(\frac 12-\frac 1{p_0})\vartheta_{\sigma}} \left\Vert Mf(\Re u(\tau))\right\Vert_{M^{s-\sigma\theta}_{p_0',q}} d\tau\right\Vert_{L_t^r}\\
& \lesssim \left\Vert  \int_0^t \scal{t-\tau}^{2\theta(\frac 12-\frac 1{p_0})\vartheta_{\sigma}} \left\Vert f(\Re u(\tau))\right\Vert_{M^{s-\sigma\theta-2}_{p_0',q}} d\tau\right\Vert_{L_t^r}.
\end{align*}
Since $\sigma<-1,$   we can choose $0<\theta\leq -\frac{2}{\sigma},$ and obtain from the embedding $M^s_{p_0',q}\subset M^{s-\sigma\theta-2}_{p_0',q}$ and Lemma \ref{Ib2} that
\begin{align}
\left\Vert \int_0^t \ss(t-\tau)Mf(\Re u(\tau))d\tau\right\Vert_{L^r(\mathbb{R},M^s_{{p_0},q})}&\lesssim \left\Vert  \int_0^t \scal{t-\tau}^{2\theta(\frac 12-\frac 1{p_0})\vartheta_{\sigma}} \left\Vert f(\Re u(\tau))\right\Vert_{M^{s}_{p_0',q}} d\tau\right\Vert_{L_t^r}\nonumber\\
&\lesssim \left\Vert  \int_0^t \scal{t-\tau}^{2\theta(\frac 12-\frac 1{p_0})\vartheta_{\sigma}} \Vert u(\tau)\Vert^{p+1}_{M^{s}_{{p_0},q}} d\tau\right\Vert_{L_t^r}.\label{z12}
\end{align}
Since
\[ 
0<1+2\theta\left( \frac 12-\frac 1{p_0}\right)\vartheta_{\sigma}<1 \ \ \ \  \text{and}\ \ \ \  \frac 1 r=\frac{p+1}{r}-\left(1+2\theta\left( \frac 12-\frac 1{p_0}\right)\vartheta_{\sigma}\right),
\]
an application of the Hardy-Littlewood-Sobolev  inequality in (\ref{z12}) completes the proof.
\end{proof} 
 \begin{lemma}\label{alter}
Let $p\geq 1$ an integer, $p_0=p+2,$ $s\geq 0,$ $r=\frac{p(p+2)}{p+2+\theta p \vartheta_{\sigma}},$  and $f(u)=u^{p+1}.$  Also assume that $0<\theta\leq -\frac{2}{\sigma}$ and $\sigma<-1$ such that $0<\frac{-\theta p\vartheta_{\sigma}}{p+2}<1.$ Then
\begin{equation}
\left\Vert \int_0^t \ss(t-\tau)Mf(\Re u(\tau))d\tau\right\Vert_{L^r(\mathbb{R},M^s_{{p_0},1})}\lesssim \Vert u\Vert^{p+1}_{L^r(\mathbb{R},M^s_{{p_0},1})}.
\end{equation}
\end{lemma}
\begin{proof}
	The proof of Lemma \ref{alter} is similar  to one of Lemma \ref{propuse} by   
	applying Lemma \ref{prod} instead of Lemma \ref{Ib2}
\end{proof}

\begin{theorem}\label{GlobalLamb0}
Let $p\geq 1$ be an integer, $p_0=p+2,$ $s\geq 0,$ and $0<\theta\leq -\frac{2}{\sigma}$, and $r=\frac{p(p+2)}{p+2+p\theta\vartheta_{\sigma}}$, where $\sigma<-1$. Assume $0<\frac{-p\theta\vartheta_{\sigma}}{p+2}<1$. Then, there exists $\epsilon>0$ such that if $\Vert \mathfrak{S}(t)u_0\Vert_{L^r(\mathbb{R};M^s_{p_0,1})}<\epsilon$, then equation (\ref{trans=eq}) is globally well-posed in $L^{r}(\mathbb{R}; M^s_{p_0,1})$.
\end{theorem}

The condition $q=1$ in Theorem \ref{GlobalLamb0} arises from utilizing the product estimate in modulation spaces. By employing a different product estimate provided in \cite{Iwabuchi}, we can extend the global existence result to values in $M^s_{p,q}$ for $1\leq q<2$, by   a regularity assumption on the index $s$ such that $1-\frac 1q\leq s<\frac 1q$. This extension is elaborated in the following theorem.

\begin{theorem}\label{GlobalLamb1}
Consider $p\geq 1$ an integer, $p_0=p+2$, $1\leq q <\infty$, and $1-\frac 1q\leq s<\frac 1q$. Also, let $0<\theta\leq -\frac{1}{\sigma}$ and $r=\frac{p(p+2)}{p+2+p\theta\vartheta_{\sigma}}$, where $\sigma<-1$, such that $0<\frac{-p\theta\vartheta_{\sigma}}{p+2}<1$. Then, there exists $\epsilon>0$ such that if $\Vert \mathfrak{S}(t)u_0\Vert_{L^r(\mathbb{R};M^s_{p_0,q})}<\epsilon$, equation (\ref{trans=eq}) in globally well-posed in $L^{r}(\mathbb{R}; M^s_{p_0,q})$.
\end{theorem}
\begin{proof}[Proof of Theorems \ref{GlobalLamb1} and \ref{GlobalLamb1}]
    
We will first delve into the proof of Theorem \ref{GlobalLamb1}, which relies on a fixed-point argument involving above obtained estimates. To do so, define $B_{2\epsilon}=\{u:\Vert u\Vert_{L^{r}(\mathbb{R}; M^s_{p_0,q})}\leq 2\epsilon\}$, where $\epsilon>0$. Define the map $\Gamma$ on the metric space $B_{2\epsilon}$ as the right-hand side of \eqref{integral-form-2}. To ensure $\Gamma: B_{2\epsilon}\rightarrow B_{2\epsilon}$ is a contraction, we select $\epsilon>0$ such that $\Gamma: B_{2\epsilon}\rightarrow B_{2\epsilon}$, as elaborated in Lemma \ref{propuse}. From Lemma \ref{propuse} and the smallness of $\Vert \ss(t)u_0\Vert_{L^r(\mathbb{R};M^s_{p_0,q})}$ we have for $u\in B_{2\epsilon}$ that 
\begin{equation}\label{e18}
\Vert \Gamma u \Vert_{L^{r}(\rr;M^s_{p_0,q})}\leq \Vert \ss(t)u_{0}(x)\Vert_{L^{r}(\mathbb{R};M^s_{p_0,q})}+C\Vert u\Vert^{p+1}_{L^{r}(\mathbb{R}; M^s_{p_0,q})}  \leq \epsilon(1+C2^{p+1}\epsilon^p) .
\end{equation}
Choosing $\epsilon>0$ such that $C2^{p+1}\epsilon^p<1$, ensures that the mapping $\Gamma$ is a contraction. By utilizing Lemmas \ref{GrupoMs} and  \ref{Isomor} (as demonstrated in the proof of Lemma \ref{propuse}), and considering Lemmas \ref{Ib1} and \ref{Ib2}, we deduce
\begin{align*}
 \Vert \Gamma u-\Gamma v\Vert_{L^{r}(\mathbb{R};M^s_{p_0,q})} 
 \lesssim \left\Vert  \int_0^t \scal{t-\tau}^{2\theta(\frac 12-\frac 1{p_0})\vartheta_{\sigma}} \Vert u-v\Vert_{M^s_{p_0,q}}\sum_{k=1}^{p+1}\Vert u\Vert^{p+1-k}_{M^s_{p_0,q}}\Vert v\Vert^{k-1}_{M^s_{p_0,q}}d\tau\right\Vert_{L_t^r}.
\end{align*}
Considering the assumption $0<\frac{-p\theta\vartheta_{\sigma}}{p+2}<1$ and $r=\frac{p(p_0+2)}{p+2+p\theta\vartheta_{\sigma}}$, it is revealed from the H\"older and Hardy-Littlewood-Sobolev  inequalities that
\begin{equation}\label{e19}
    \begin{split}
\Vert \Gamma u-\Gamma v\Vert_{L^{r}(\mathbb{R};M^s_{p_0,q})}&\lesssim \Vert u-v\Vert_{L^{r}(\mathbb{R};M^s_{p_0,q})}\paar{\Vert u\Vert^p_{L^{r}(\mathbb{R};M^s_{p_0,q})}+\Vert v\Vert^p_{L^{r}(\mathbb{R};M^s_{p_0,q})}} \\
&\leq 2^{p+1}\epsilon^p C\Vert u-v\Vert_{L^{r}(\mathbb{R};M^s_{p_0,q})}.
\end{split}
\end{equation}
Combining inequalities (\ref{e18}) and (\ref{e19}), we establish that the mapping $\Gamma: B_{2\epsilon}\rightarrow B_{2\epsilon}$ acts as a contraction, ensuring the existence of a unique fixed point. The demonstration of Theorem \ref{GlobalLamb0} mirrors the approach adopted for Theorem \ref{GlobalLamb1}, albeit incorporating Lemma \ref{alter} in lieu of Lemma \ref{propuse}.
\end{proof}


\section{Solitary wave solutions}\label{section-solitary}
In this section, we study the  non-existence of solitary waves of   \eqref{gBE} for some parameters.  Throughout this section, we assume that $f(u)=\beta u^{p+1}$ with $\beta\in\rr$.

To find the  localized solitary wave solutions of the  equation \eqref{gBE}, we use the ansatz \mbox{$u(x,t)=Q_c(\xi), ~~\xi=x-ct$}
with $\displaystyle{\lim_{|\xi| \rightarrow \infty} Q_c(\xi)=0}$ which yields to  the following ordinary differential equation
\begin{equation}
	c^2 Q_c'' =Q_c''-(\alpha-c^2) Q_c''''-\kappa c^2 Q_c''''''+\beta(Q_c^{p+1})''.
	\label{solitary}
\end{equation}
Here $'$ denotes the derivative with respect to $\xi$.
Integrating the equation \eqref{solitary} twice, we obtain
\begin{equation} \label{solitary2}
	\kappa c^2 Q_c''''+(\alpha-c^2) Q_c''+(c^2-1)Q_c=\beta Q_c^{p+1}.
\end{equation}

\noindent
In the case $\beta<0$ and $c=0$, it is well-known that the equation \eqref{solitary2} has the unique solitary wave
\[
Q_0(x)=\left( \frac{p+2}{-2\beta}\right)^\frac1p
\sech^\frac2p\left(\frac{p}{2}\sqrt{\frac1\al}x\right).
\]

\begin{theorem}\label{theorem-nonex}
	
	Equation \eqref{solitary2} with $\alpha>0$ and $\kappa>0$ does not admit any nontrivial solution \mbox{$Q_c \in  H^{4}(\mathbb{R})$}  if one of the following conditions holds.
	\begin{description}
		\item[i.] $c^2 \geq \max\{\alpha,1\}$, $\beta <0$ and $p$ is even.
		\item[ii.] $ c^2 \leq\min \{ \alpha, 1 \},~ \beta>0 $  and $p$ is even.
		\item[iii.]$\alpha \leq c^2 <1$    for all $p>0$.
	\end{description}
\end{theorem}
\textbf{Proof:} Let $Q_c$ be any nontrivial solution of the equation \eqref{solitary2}. Multiplying the equation \eqref{solitary2} by $Q_c$, integrating on $\mathbb{R}$ and  performing integration by parts,  we get

\begin{equation}
	\kappa c^2\int_{\mathbb{{R}}} (Q_c^{\prime\prime})^2 \dd x + (c^2-\alpha) \int_{\mathbb{R}} (Q_c^\prime)^2 \dd x  +  (c^2-1) \int_{\mathbb{{R}}} {Q_c^2 \dd x }=  \beta \int_{\mathbb{{R}}} Q_c^{p+2} \dd x. \label{energy1}
\end{equation}

\noindent
The term on the left side of this equation will be non-negative,  a contradiction, if the condition $(i)$ is satisfied.

\noindent
On the other hand, multiplying  the equation \eqref{solitary2} by $x Q_c'$ and integrating over $\mathbb{R}$ yields the Pohozaev type identity
\begin{equation}
	\frac{3\kappa c^2}{2}\int_{\mathbb{{R}}} (Q_c^{\prime\prime})^2 \dd x + \frac{c^2-\alpha}{2} \int_{\mathbb{R}} (Q_c^\prime)^2 \dd x -  \frac{c^2-1}{2} \int_{\mathbb{{R}}} {Q_c^2 \dd x }= -\frac{\beta}{p+2} \int_{\mathbb{{R}}} Q_c^{p+2} \dd x. \label{energy2}
\end{equation}


\noindent
Eliminating  $(Q_c^{\prime\prime})^2$ terms in the equations \eqref{energy1} and \eqref{energy2},  we get
\begin{equation}
	2(c^2-\alpha) \int_{\mathbb{R}} (Q_c^{\prime})^2 \dd x + 4(c^2-1) \int_{\mathbb{R}} Q_c^2 \dd x  =  \frac{p+4}{p+2}~\beta\int_{\mathbb{{R}}} Q_c^{p+2} \dd x.
\end{equation}
The term on the left side of this equation will be negative, a contradiction, when condition $(ii)$ is satisfied.

\noindent
Eliminating  $Q_c^{p+2}$ terms in the equations \eqref{energy1} and \eqref{energy2} leads to
\begin{equation}\label{poho-3}
	\kappa c^2\left(\frac{3p}{2}+4 \right)\int_{\mathbb{{R}}} (Q_c^{\prime\prime})^2 \dd x + (c^2-\alpha)\frac{p+4}{2} \int_{\mathbb{R}} (Q_c^\prime)^2 \dd x  = (c^2-1)\frac{p}{2}\int_{\mathbb{{R}}} Q_c^2 \dd x.
\end{equation}
The condition $\alpha \leq c^2 <1$ implies that the left-hand side is non-negative and the right-hand side is negative.

\begin{cor}
	It	is also possible to extend the previous result, independent of $p$.
	Indeed, since $Q_c\in H^2$, a bootstrap argument
	shows the solution $Q_c\in H^\infty$  and $Q_c\to0$ together with its derivatives as $|x|\to\infty$ (see \cite{Amin-lev}). Thus,
	to study the asymptotic behavior of the solution, it suffices to study the solutions of the linearized equation by ignoring the nonlinear term, that is
	\[
	\kappa c^2 g''''+(\alpha-c^2) g''+(c^2-1)g=0.
	\]
	The characteristic equation is
	\begin{equation}
		\kappa c^2 \lambda^4+(\alpha-c^2) \lambda^2+ c^2-1  =0, \label{characteristic}
	\end{equation}
	and the roots are
	\begin{equation}
		\lambda^2=\frac{\alpha-c^2\pm\sqrt{\Delta}}{-2\kappa c^2}, \nonumber
	\end{equation}
	where
	\begin{equation}\label{disc}
		\Delta= (\alpha-c^2)^2-4\kappa c^2(c^2-1)
	\end{equation}
	In the case $\alpha>c^2>1$ and
	\begin{equation}\label{c-codition-0}
		\kappa\leq \frac{	(\alpha-c^2)^2}{ 4c^2(c^2-1) }
	\end{equation}
	the above roots are pure imaginary. It turns out that the solution of the linearized equation, and whence the nonlinear equation does not vanish at infinity. This is a contradiction.
\end{cor}

\begin{remark}
	By using the proof of Theorem \ref{theorem-nonex}, one can extend the non-existence results. Actually, by using the inequality
	\[
	\|g'\|_{L^2}^2\leq
	\|g\|_{L^2} \|g''\|_{L^2}
	\]
	and the Cauchy-Schwarz inequality $ab\leq \epsilon a^2+b^2/(4\epsilon)$, we obtain from \eqref{poho-3} that
	\[
	(	\kappa c^2(3p+8)-\epsilon|c^2-\alpha|(p+4))  \|Q_c''\|_{L^2} ^2
	+\left(p (1-c^2 )-\frac{|c^2-\alpha|(p+4)}{4\epsilon}\right) \|Q_c\|_{L^2}^2 \leq0.
	\]
	If $c^2\neq\alpha$ and $c^2<1$, by choosing $\epsilon=\frac{\kappa c^2(3p+8)}{|c^2-\alpha|(p+4)}$ we have
	\[
	\left(p (1-c^2 )-\frac{(c^2-\alpha)^2(p+4)^2}{4\kappa c^2(3p+8)}\right) \|Q_c\|_{L^2}^2 \leq0.
	\]
	Thus, $Q_c\equiv0$ if
	\[
	\frac{\kappa(1-c^2 )c^2}{(c^2-\alpha)^2}>\frac{(p+4)^2}{4 p(3p+8)}.
	\]
\end{remark}

In the case $\beta>0$,   we can consider the variational problem

\begin{equation}\label{minim}
	\inf\{I_c(u), \;u\in H^2(\rr),\;\|u\|_{L^{p+2}(\rr)}=1\},
\end{equation}
where
\[
I_c(u)=\frac12\int_\rr\left(\kappa c^2u_{xx}^2+(c^2-\al)u_x^2+(c^2-1)u^2\right)\dd x,
\]
and obtain the existence of solitary waves of \eqref{solitary2} by using the arguments in \cite{Amin-lev} and  applying the concentration-compactness principle. More details can be found in \cite[Theorem 2.1]{Amin-lev}.
\begin{theorem}
	Let $p>1$, $\kappa,\beta>0$, $c^2>1$ and the condition
	\begin{equation}\label{discriminant}
		c^4(1-4\kappa)+\alpha^2+2c^2(2\kappa-\alpha)<0
	\end{equation}
	holds. If
	$\{u_n\}$ is a minimizing sequence of \eqref{minim}, then there is a subsequence $u_{n_k}\subset H^2(\rr)$, scalars $y_k\in\rr$ and $u\in H^2(\rr)$ such that $u_{n_k}(\cdot+y_k)\to u$ in $H^2(\rr)$ and $u$ is a minimizer of \eqref{minim}.
\end{theorem}

\subsection{The Petviashvili  iteration method}

The Petviashvili iteration method was first introduced by V.I. Petviashvili for the Kadomtsev-Petviashvili equation in \cite{petviashvili} to construct a solitary wave solution. We use the Petviashvili  iteration method \cite{pelinovsky, petviashvili, yang} to construct the solitary wave solution of \eqref{gBE} cannot be determined analytically.
If we use the Fourier transform,
\begin{equation}
	u(x)=\frac{1}{2\pi} \int_{-\infty}^{\infty} \widehat{u} (k)e^{ik x} \dd k, \hspace*{30pt} \widehat{u}(k)=\int_{-\infty}^{\infty} u(x) e^{-ik x} \dd x,
\end{equation}
the equation \eqref{solitary2} is rewritten
\begin{equation}
	\left[\kappa c^2k^4+(c^2-\alpha)k^2+c^2-1 \right]\widehat{Q_c}(k)= \beta~ \widehat{Q_c^{p+1}}(k). \label{iteras2}
\end{equation}
From now on, we write  $Q$ shortly instead of $Q_c$. A simple iterative algorithm for $\widehat{Q}(k)$ of the  equation \eqref{iteras2} can be proposed in the form
\begin{equation}
	\widehat{Q}_{n+1}(k)= \frac{\beta~ \widehat{Q_n^{p+1}}(k)}{\kappa c^2k^4+(c^2-\alpha)k^2+c^2-1}. \label{algor}
\end{equation}
where $\widehat{Q}_n (k)$ is the Fourier transform of $Q_n (x)$ which is the   $n^{th}$ iteration of the numerical solution.
Although there exists a fixed point $\widehat{Q}_n (k)$ of the equation \eqref{iteras2}, the algorithm \eqref{algor} diverges. To ensure the convergence, we
add a stabilizing factor $M_n$ given in \cite{petviashvili}. The new algorithm for \eqref{gBE} is given by
\begin{equation}
	\widehat{Q}_{n+1}(k)= M_n^\gamma \frac{\beta~ \widehat{Q_n^{p+1}}(k)}{\kappa c^2k^4+(c^2-\alpha)k^2+c^2-1}, \label{newalgor}
\end{equation}
where the stabilizing factor is
\begin{equation}
	M_n=\frac{\int_{-\infty}^{\infty} ~ \left[\kappa c^2k^4+(c^2-\alpha)k^2+c^2-1 \right] [\widehat{Q}_n(k)]^2 \dd k }{\int_{-\infty}^{\infty} {\beta}~ \widehat{Q^{p+1}_n}(k)\widehat{Q}_n(k) \dd k  } \label{stab}
\end{equation}
and $\gamma$ is a free parameter. We  construct the solitary wave solutions for \eqref{gBE} if the  following condition
\begin{equation}
	\kappa c^2k^4+(c^2-\alpha)k^2+c^2-1 \neq 0  \label{condc}
\end{equation}
is satisfied for all $k\in \mathbb{R}$. The iterative process is  controlled by the  error,
\begin{equation}
	Error(n)=\|Q_n-Q_{n-1}\|,~~~~n=0,1,\cdots \nonumber
\end{equation}
between two consecutive iterations defined with the  number of iterations,  the stabilization factor error
\begin{equation}
	|1-M_n|, ~~~~n=0,1,\cdots \nonumber
\end{equation}
and the residual error
\begin{equation}
	{RES(n)}= \|{\cal R} Q_n\|_\infty, ~~~~n=0,1,\cdots \nonumber
\end{equation}
where
\begin{equation}
	{\cal R}Q= \kappa c^2 Q_c^{\prime\prime\prime\prime}+(\alpha-c^2) Q_c^{\prime\prime}+(c^2-1)Q_c-\beta Q_c^{p+1}
\end{equation}

\subsubsection{Accuracy test}
In the case of $\alpha=0$, equation \eqref{gBE} reduces the  HBq equation  \eqref{hbq1}.  The solitary wave solution of the HBq equation
with $f(u)=u^p$
is  given by
\begin{eqnarray} \label{soliterhbq}
	&&\hspace{80pt}u(x,t)=A\left\{ {\mbox{sech}}^4\left( B(x-ct-x_0)\right) \right\}^{\frac{1}{p-1}},\hspace{-8pt}
	\\  && \hspace{-40pt} A=\left[\frac{\eta_1^2 c^2(p+1)~(p+3)~(3p+1)}{2\eta_2~(p^2+2p+5)^2}\right]^{\frac{1}{p-1}},
	\hspace{10pt} B=\left[\frac{\eta_1(p-1)^2}{4\eta_2(p^2+2p+5)}\right]^{\frac{1}{2}},
	\\ &&\hspace{50pt} c^2=\left\{{1-\left[\frac{4\eta_1^2~(p+1)~^2}{\eta_2~(p^2+2p+5)^2}\right]}\right\}^{-1}  , \label{soliterhbq3}
\end{eqnarray}
where $A$ is the amplitude and $B$ is the inverse width of the solitary wave in \cite{oruc}.
Here $c$ represents the velocity of the solitary wave centered at $x_0$ with $c^2>1$.

\begin{figure}[!htbp]
	\begin{minipage}[t]{0.45\linewidth}
		\centering
		\hspace*{-20pt}
		\includegraphics[height=5.5cm,width=7.5cm]{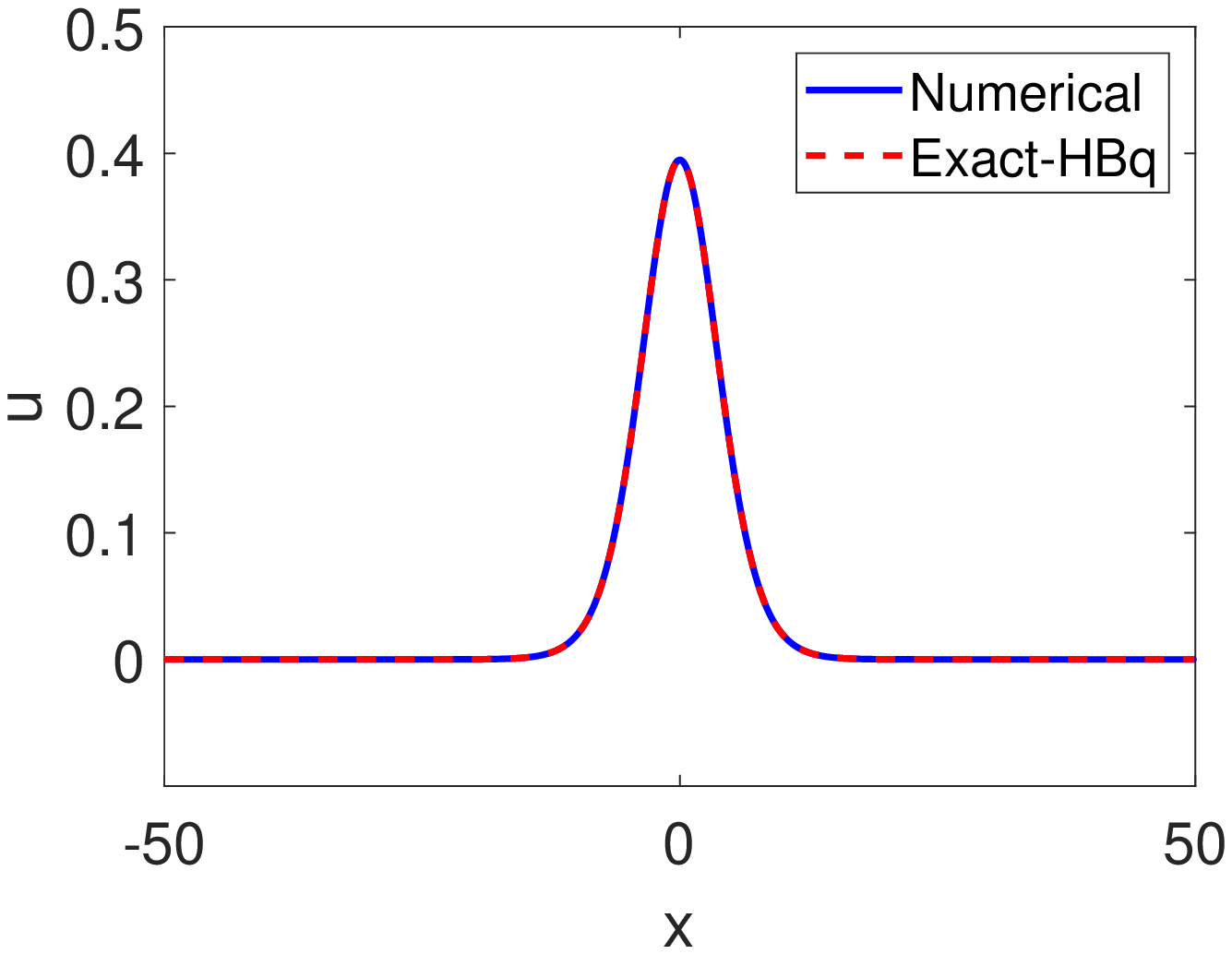}
	\end{minipage}%
	\hspace{3mm}
	\begin{minipage}[t]{0.45\linewidth}
		\centering
		\includegraphics[height=5.5cm,width=7.5cm]{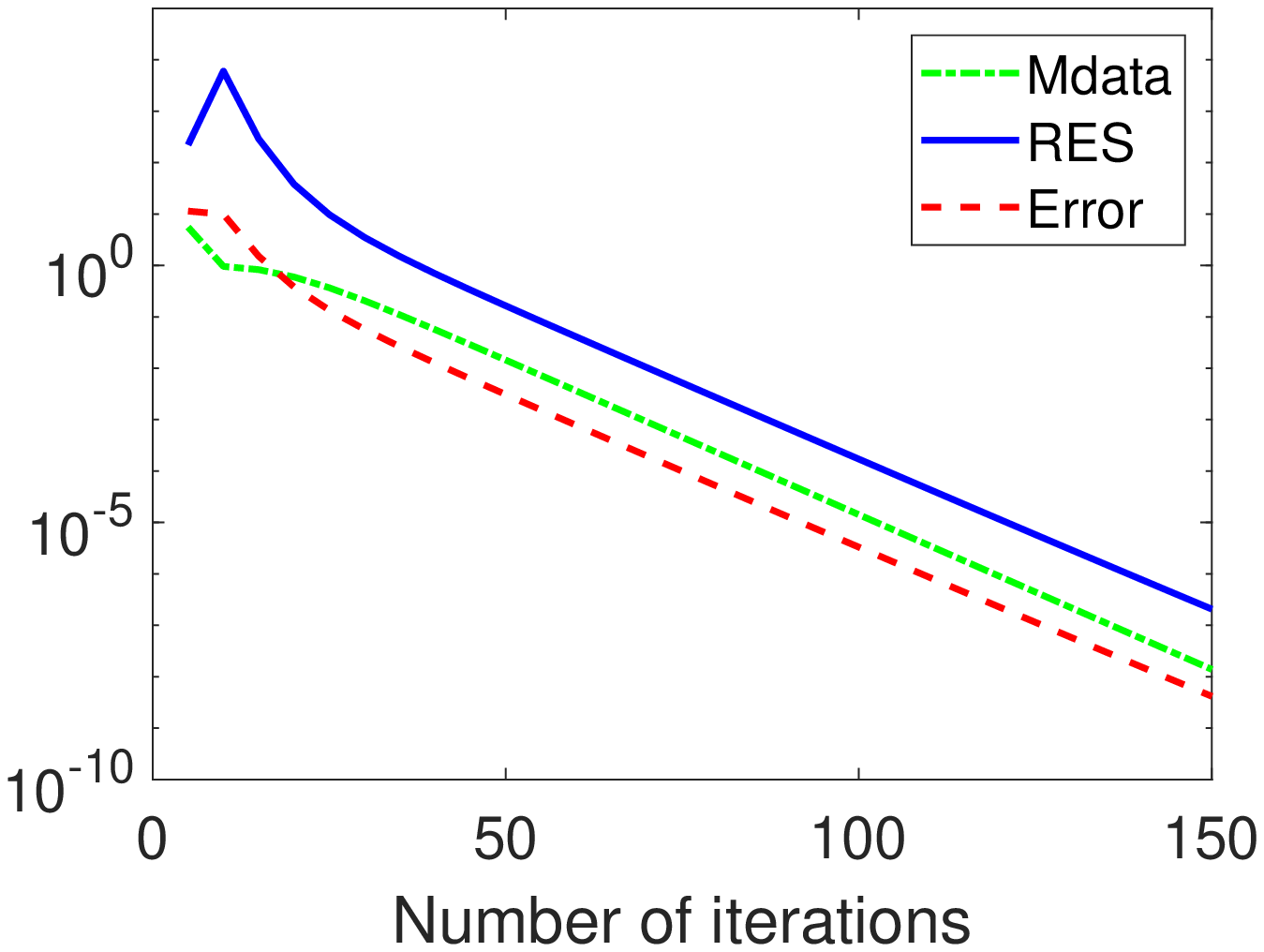}
	\end{minipage}
	
	\caption{The solitary wave solution constructed by Petviashvili method for \eqref{gBE} and   the exact solitary wave solution of the HBq equation for the speed $c = 1.1272$ and   the variation of the $Error(n)$, $|1-M_n|$ and $RES$ with the number of iterations in semi-log scale.}\label{figure1}
\end{figure}
\noindent
In order to test our scheme, we compare the solitary wave profile generated by Petviashvili's method for \eqref{gBE} ($\alpha=0, \kappa=1, p=1, \beta=1$)  with the exact solitary wave solution of the HBq equation given in \eqref{soliterhbq}-\eqref{soliterhbq3} ($\eta_1=\eta_2=1, p=2)$. The space interval is $-50 \le x \le 50$ and we choose the number of spatial grid points $N=1024$.
Choosing $\eta_1=\eta_2=1, p=2$ in the equation \eqref{soliterhbq3}, we compute the wave speed as $c= 1.1272$. Then, taking the same wave speed, we
generate the solitary wave profile by the Petviashvili's method for \eqref{gBE} ($\alpha=0, \kappa=1, p=1, \beta=1$).
In the left panel of  Figure \ref{figure1}, we present the solitary wave solution constructed by Petviashvili's method for \eqref{gBE} ($\alpha=0, \kappa=1, p=1, \beta=1$)  with the exact solitary wave solution of the HBq equation.
In the right panel of Figure \ref{figure1}, we show the variation of three different errors
with the number of iterations in semi-log scale. As seen from the figure, the solitary wave profile generated by Petviashvili's method coincides with the exact solitary wave solution of the HBq equation which is a special case of the gBq equation. We also compute the $L_{\infty}$-error norm as  $5.69\times 10^{-12}$.

\subsubsection{Numerical generation of solitary waves for the gBq equation}
To the best of the authors' knowledge, there is no exact solitary wave solution of the gBq equation where the parameter $\alpha$ is different from zero. It is natural to ask how the higher-order effects of dispersion and nonlinearity affect the solitary wave solutions.
We now construct the solitary wave solution for some values of $\alpha, \kappa$, and $\beta$.  The space interval is $-50 \le x \le 50$ choosing the number of spatial grid points $N=1024$. We consider the quadratic nonlinearity $(p=1, f(u)=\beta u^{p+1})$.

\begin{figure}[!htbp] \label{figure2}
	\begin{minipage}[t]{0.45\linewidth}
		\centering
		\hspace*{-20pt}
		\includegraphics[height=5.5cm,width=7.5cm]{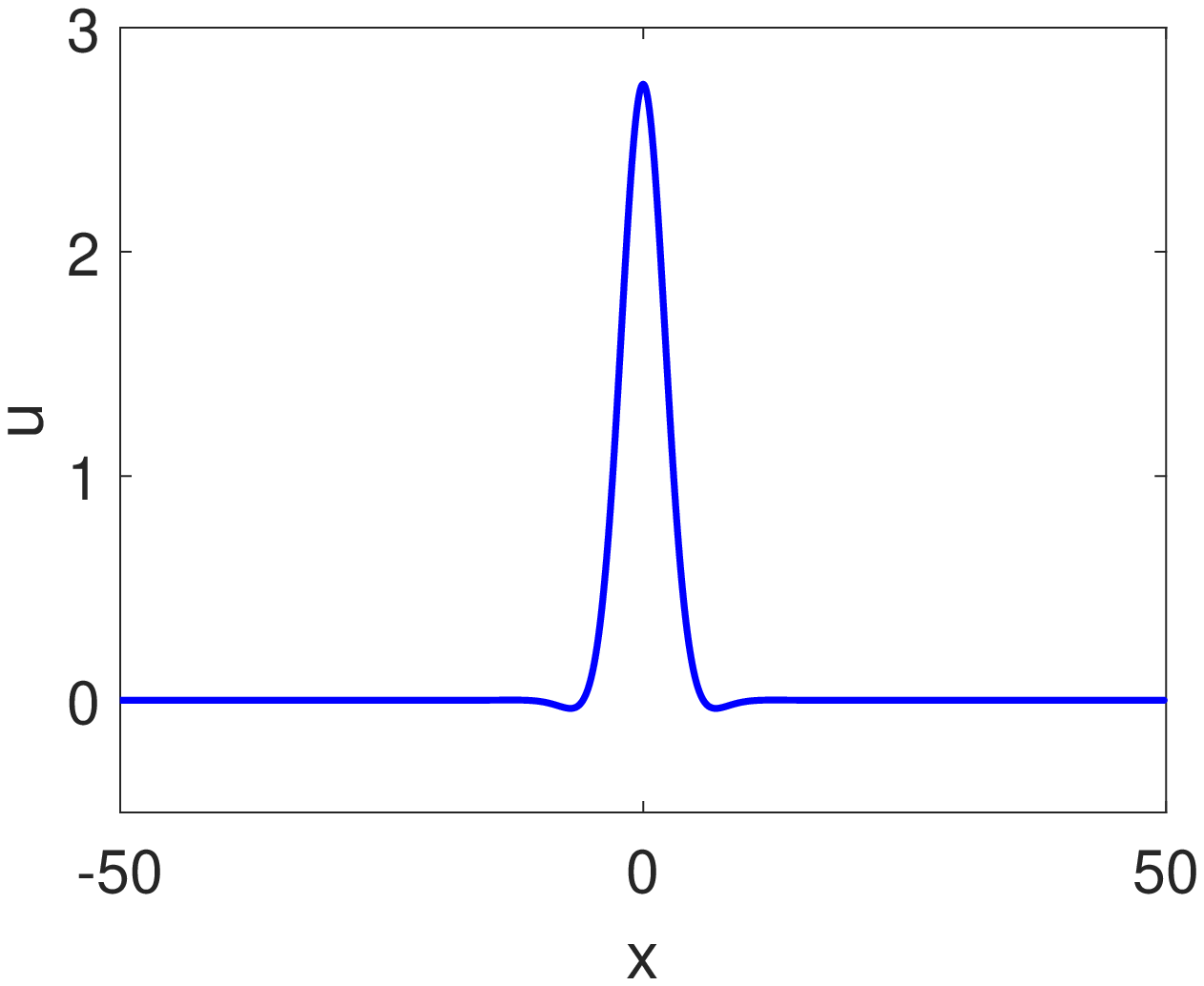}
	\end{minipage}%
	\hspace{3mm}
	\begin{minipage}[t]{0.45\linewidth}
		\centering
		\includegraphics[height=5.5cm,width=7.5cm]{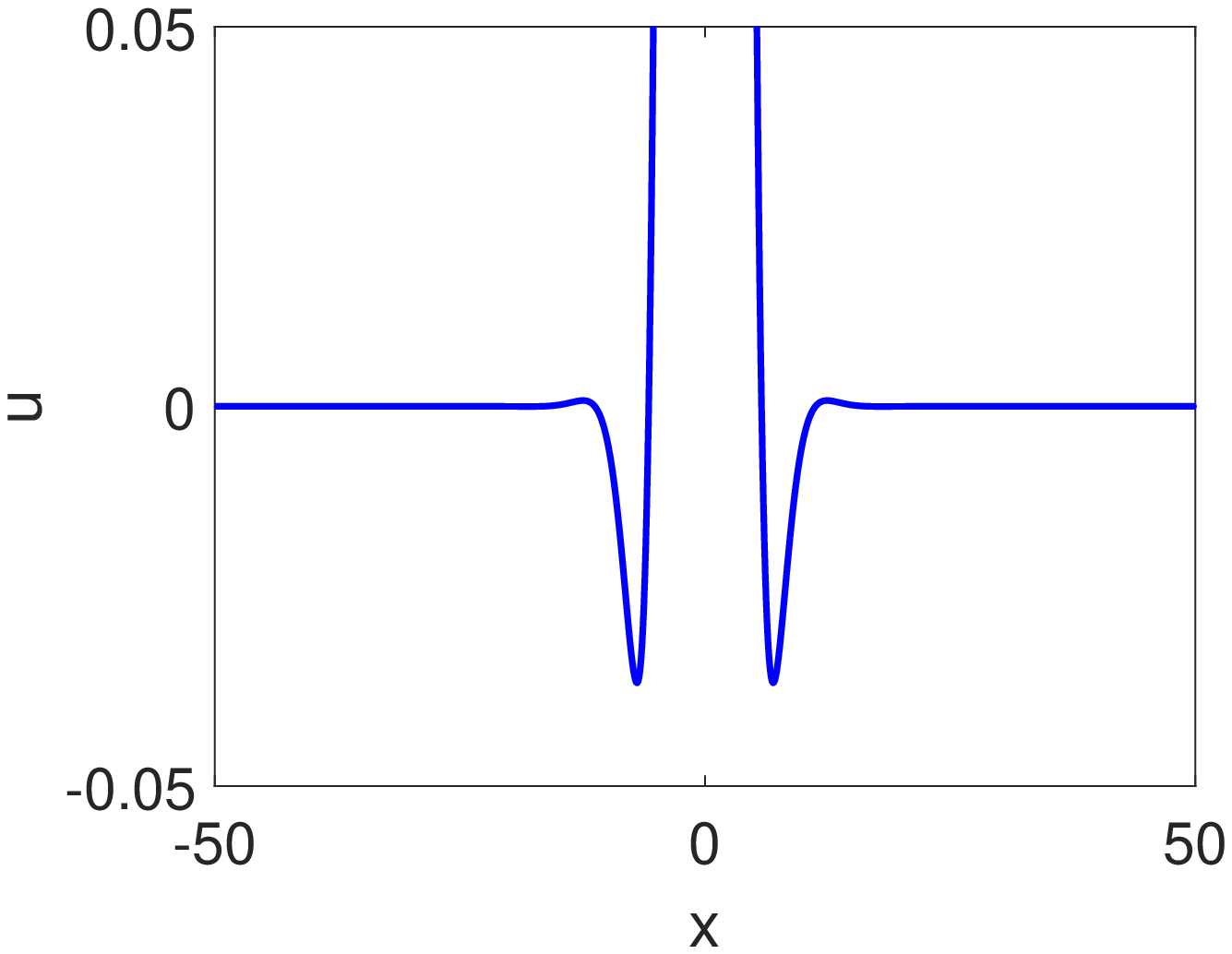}
	\end{minipage}
	\begin{minipage}[t]{0.45\linewidth}
		\centering
		\hspace*{120pt}
		\includegraphics[height=5.5cm,width=7.5cm]{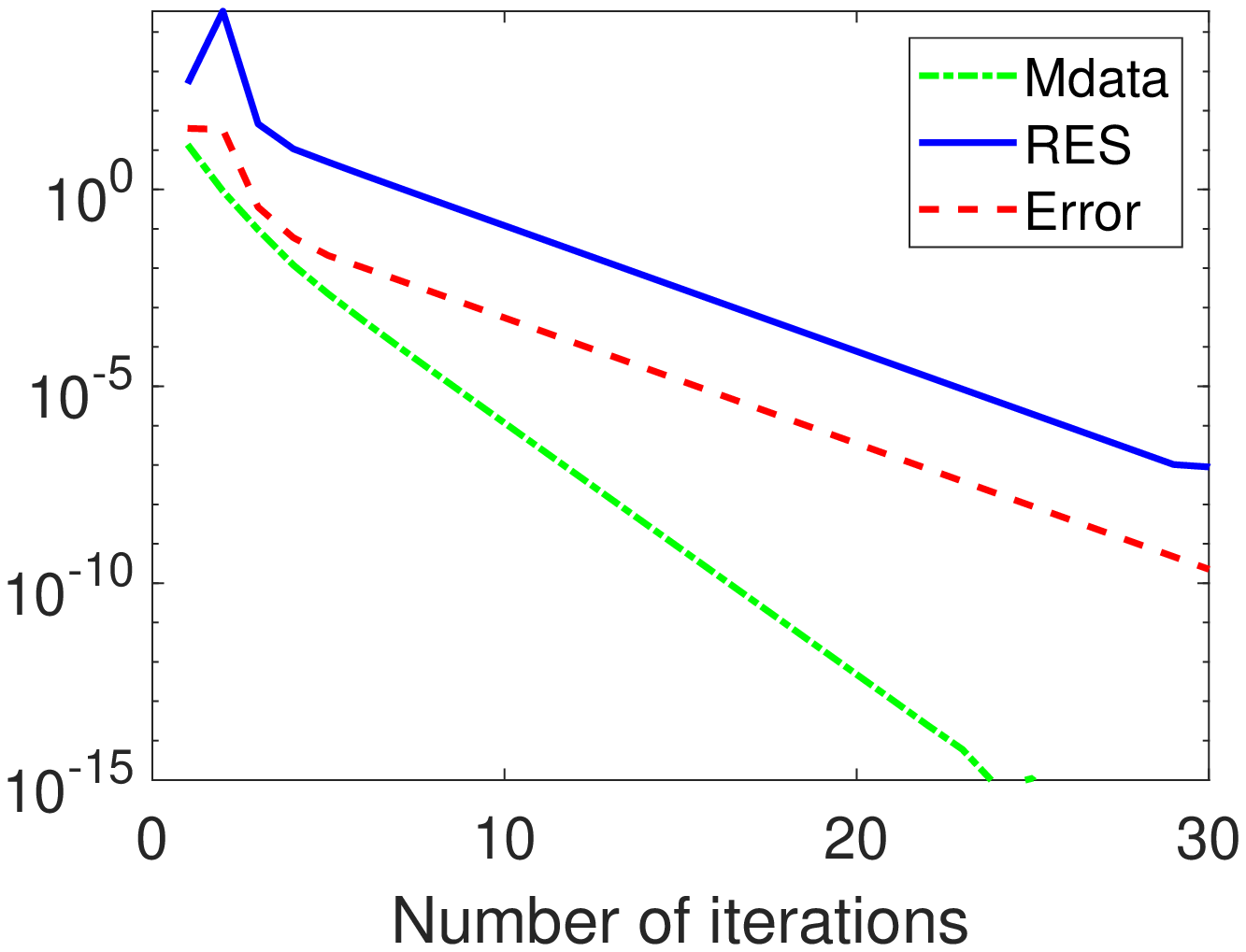}
	\end{minipage}
	
	\caption{The solitary wave solution generated by the  Petviashvili's method for \eqref{gBE}  with the parameters $\alpha=2$, $c=\sqrt{3}, \kappa=\beta=1$ (top left),  the close-up look at the profile (top right) and   the variation of the $Error(n)$, $|1-M_n|$ and $RES$ with the number of iterations in semi-log scale (bottom).} 
\end{figure}

\begin{figure}[!htbp] 
	\begin{minipage}[t]{0.45\linewidth}
		\centering
		\hspace*{-20pt}
		\includegraphics[height=5.5cm,width=7.5cm]{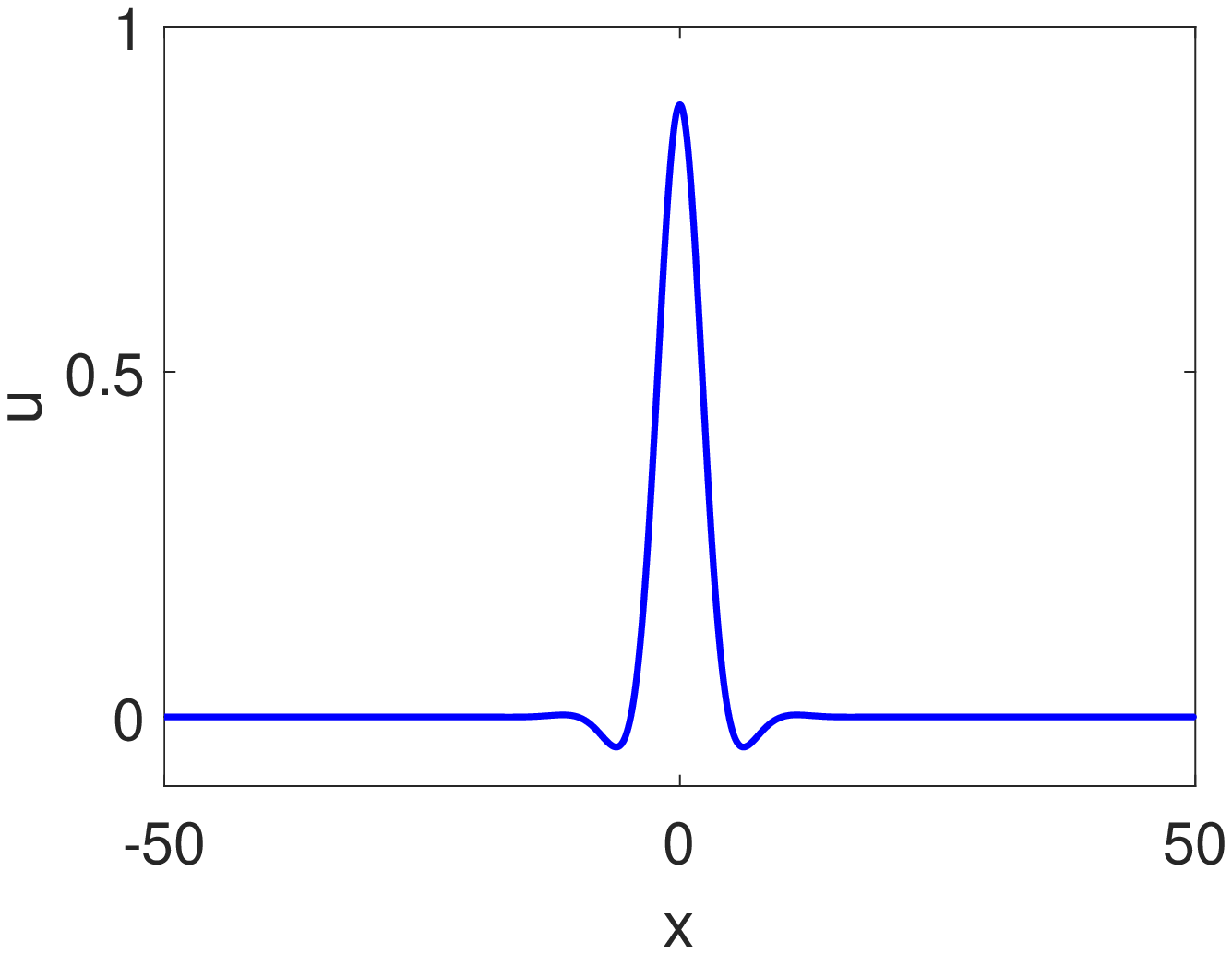}
	\end{minipage}%
	\hspace{3mm}
	\begin{minipage}[t]{0.45\linewidth}
		\centering
		\includegraphics[height=5.5cm,width=7.5cm]{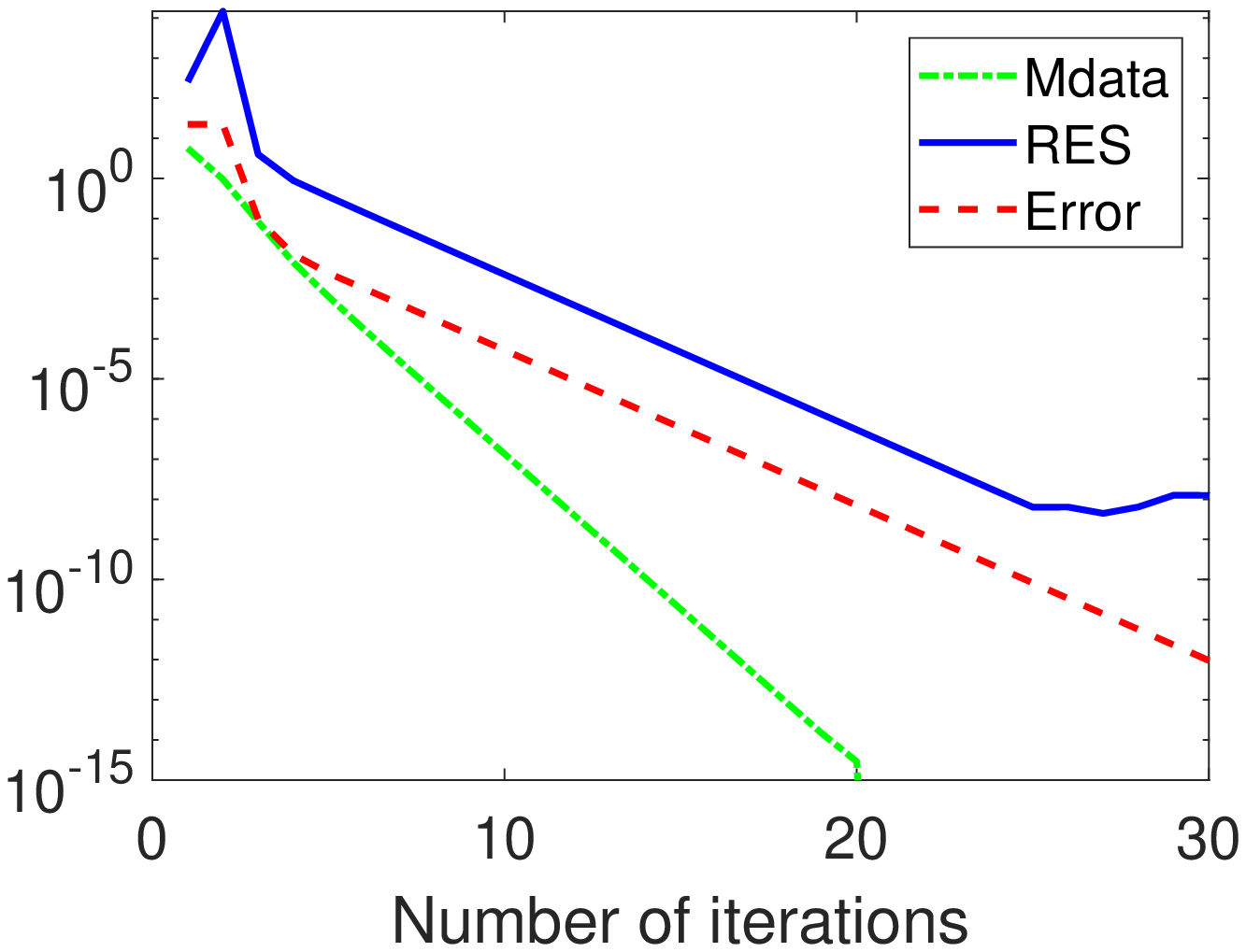}
	\end{minipage}
	
	\caption{The solitary wave solution generated by the  Petviashvili's method for \eqref{gBE}  with the parameters $\alpha=2$, $c=1.3, \kappa=\beta=1$ (left panel) and   the variation of the $Error(n)$, $|1-M_n|$ and $RES$ with the number of iterations in semi-log scale (right panel).}\label{figure3}
\end{figure}
The solitary wave profile generated by  Petviashvili's method for the parameters $\alpha=2$, $c=\sqrt{3}, \kappa=\beta=1$, where $c^2> \alpha>1$  is illustrated in the top left panel of Figure \ref{figure2}. The top right panel of Figure \ref{figure2} gives a closer look. In the bottom panel, we present the variation of    three different errors
with the  number of iterations in the semi-log scale.
The solitary wave profile generated by  Petviashvili's method for the parameters $\alpha=2$, $c=1.3, \kappa=\beta=1$, where $1< c^2< \alpha$ is illustrated in the  left panel of Figure \ref{figure3}. In the right panel of Figure \ref{figure3}, we present  the variation of    three different  errors
with the  number of iterations in semi-log scale.
As it is seen from Figures \ref{figure2} and \ref{figure3}, the solitary wave solution  has an oscillatory structure. For the parameters $ (i)~\alpha=2, c=\sqrt{3}, ~\kappa=\beta=1 $, $(ii)~\alpha=2, c=1.3, ~ \kappa=\beta=1$, the discriminant in \eqref{disc} is calculated approximately as
$-22.99, -4.56$, respectively. In the presence of the imaginary parts of the roots, the solitary wave solution has oscillatory asymptotics.  Therefore, the numerical results are compatible with the expected theoretical result.
\begin{figure}[!htbp] \label{figure4}
	\begin{minipage}[t]{0.45\linewidth}
		\centering
		\hspace*{-20pt}
		\includegraphics[height=5.5cm,width=7.5cm]{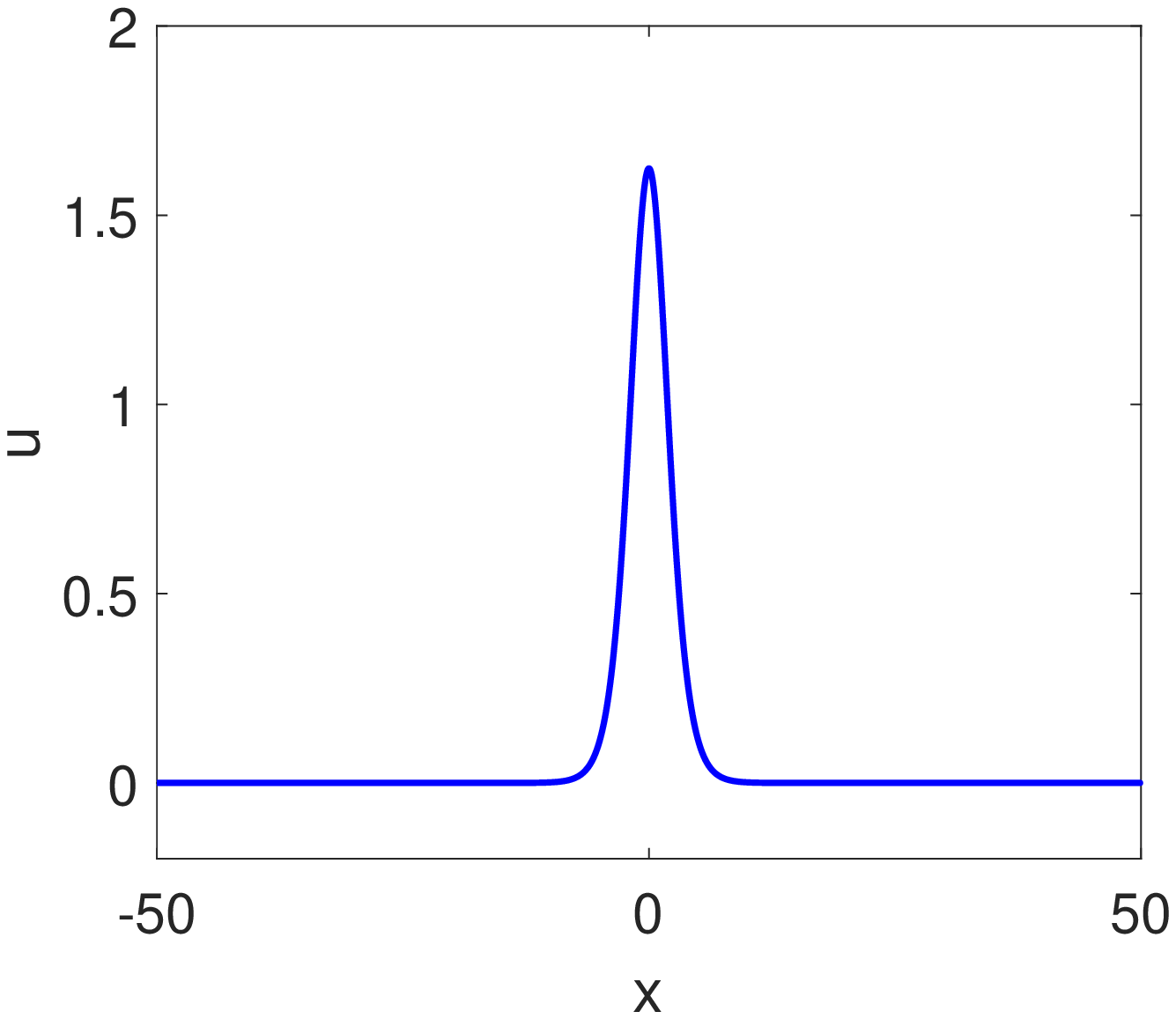}
	\end{minipage}%
	\hspace{3mm}
	\begin{minipage}[t]{0.45\linewidth}
		\centering
		\includegraphics[height=5.5cm,width=7.5cm]{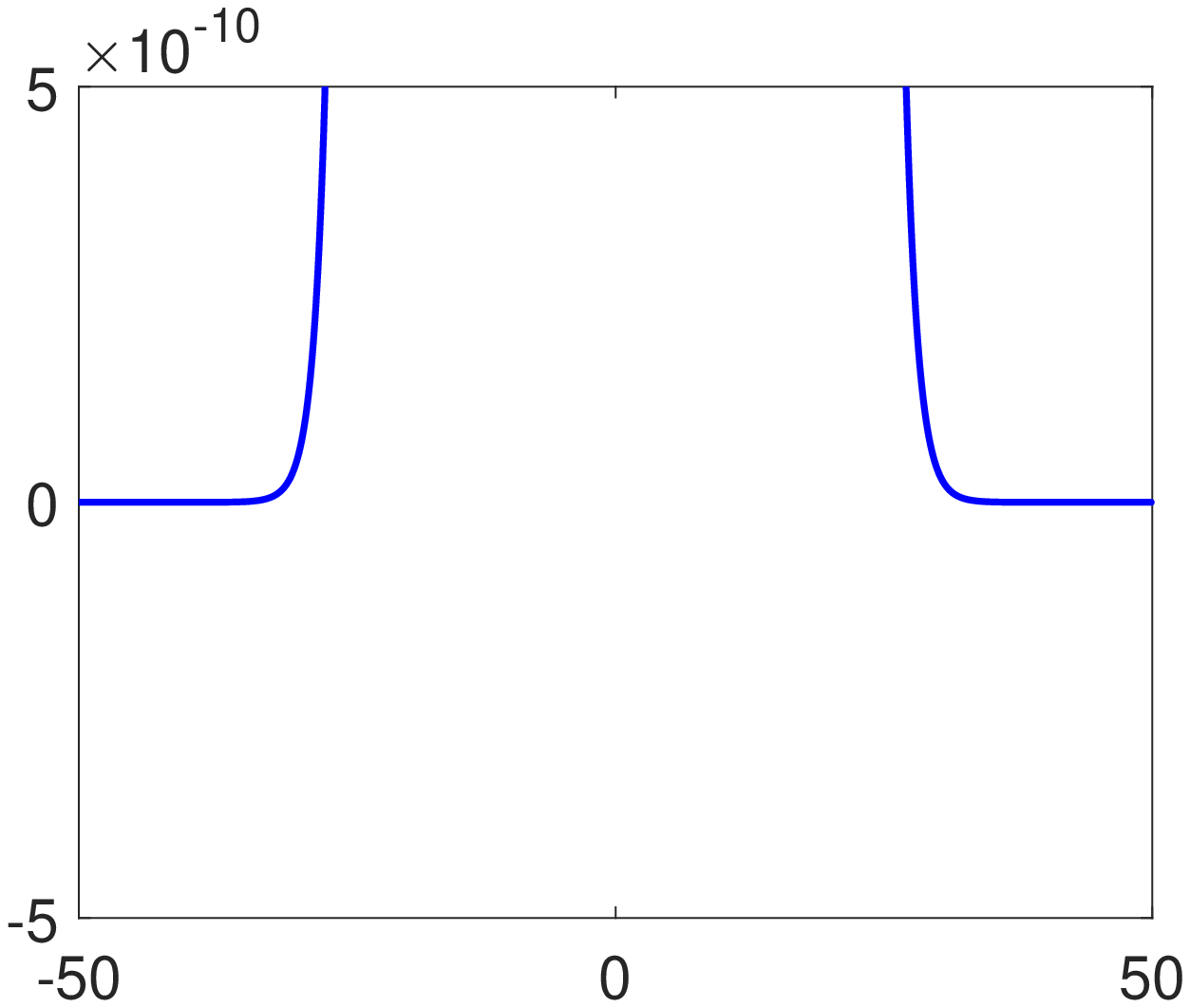}
	\end{minipage}
	\vspace{-5mm}
	\caption{The solitary wave solution generated by the  Petviashvili's method for \eqref{gBE}  with the parameters $\alpha=0.5, c=\sqrt{2.1}, \kappa=0.1, \beta=1$ (left panel),  the close-up look at the profile (right panel).} 
\end{figure}

\noindent
The solitary wave profile  for the parameters $\alpha=0.5, c=\sqrt{2.1}, \kappa=0.1, \beta=1$ is presented in the left panel of Figure \ref{figure4}.  Contrary to the above experiments, we observe that \eqref{gBE} has the monotone sech-type solitary wave profile. The  right panel of Figure \ref{figure4} gives a closer look. It takes  nonnegative values even if near zero. We observe that the equation \eqref{characteristic} has four real roots for the chosen parameters. Therefore, the oscillatory asymptotic profile is disappeared in this case.

\setcounter{equation}{0}
\section{Fourier pseudo-spectral  numerical method for  the gBq equation} \label{section-numericalmethod}
In order to investigate the time evolution of the solutions for \eqref{gBE},  we propose the  numerical method  combination of  a Fourier pseudo-spectral method for the space component
and a fourth-order Runge Kutta scheme (RK4) for time.
If the spatial period $[-L,L]$ is, for convenience,
normalized to $[0,2\pi]$ using the transformation
\mbox{$X=\pi(x+L)/L$}, the equation \eqref{gBE} becomes
\begin{eqnarray} \label{gBq_scaled}
	u_{tt}&=&\left(\frac{\pi }{L}\right)^2 u_{XX}- \left(\frac{\pi }{L}\right)^4 \alpha~ u_{XXXX}+\left(\frac{\pi }{L}\right)^2 u_{XXtt}- \left(\frac{\pi }{L}\right)^4 \kappa~ u_{XXXXtt}
	+(\frac{\pi }{L})^2\beta~(u^{p+1})_{XX}  \nonumber \\
	&&
\end{eqnarray}

\noindent
The interval $[0,2\pi]$ is divided into $N$ equal subintervals with grid spacing
$\Delta X=2\pi/N$, where the integer $N$ is even. The spatial grid points are given by
$X_{j}=2\pi j/N$,  $j=0,1,2,...,N$. The approximate solutions to
$u(X_{j},t)$ is denoted by $U_{j}(t)$. The discrete Fourier transform of the sequence
$\{U_{j} \}$
\begin{equation}\label{dft}
	\widetilde{U}_{k}={\cal F}_{k}[U_{j}]=
	\frac{1}{N}\sum_{j=0}^{N-1}U_{j}\exp(-ikX_{j}),
	~~~~-\frac{N}{2} \le k \le \frac{N}{2}-1~
\end{equation}
gives the corresponding Fourier coefficients. Similarly, $\{U_{j} \}$ can
be recovered from the Fourier coefficients by the inversion formula
for the discrete Fourier transform (\ref{dft})
\begin{equation}\label{invdft}
	U_{j}={\cal F}^{-1}_{j}[\widetilde{U}_{k}]=
	\sum_{k=-\frac{N}{2}}^{\frac{N}{2}-1}\widetilde{U}_{k}\exp(ikX_{j}),
	~~~~j=0,1,2,...,N-1.
\end{equation}
Here $\cal F$ denotes the discrete Fourier transform and
${\cal F}^{-1}$ its inverse.  These transforms are performed using the well-known software package FFT algorithm in Matlab.

\noindent
Applying the discrete Fourier transform to the equation \eqref{gBq_scaled}, the system of ordinary differential equations is given explicitly
\begin{eqnarray}
	&& (\widetilde{U}_k)_t=\widetilde{V}_k  \label{gBq_fourier1}  \\
	&& (\widetilde{V}_k)_t= -\frac{\Big\{ \big [(\pi k/L)^2+\alpha (\pi k/L)^4 \big ]  \widetilde{U}_k  +\beta (\pi k/L)^2
		(\widetilde{U^{p+1}})_k \Big\}}{1+(\pi k/L)^2+\kappa(\pi k/L)^4}
	\label{gBq_fourier2}
\end{eqnarray}
We use the fourth-order Runge-Kutta method to solve the   system of ordinary differential equations
\eqref{gBq_fourier1}-\eqref{gBq_fourier2} in time.
Finally, we find the approximate solution by using the inverse Fourier transform \eqref{invdft}.

\subsection{Time evolution of the single solitary wave}
The aim of  this subsection is to investigate the time evolution of the single solitary wave solution of \eqref{gBE}. First, we show that our proposed method is capable of high accuracy and
to confirm the convergence of the  scheme in space and time. The $L_\infty$-error norm is defined as
\begin{equation}
	L_{\infty}\mbox{-error}=\max_i |~u_i-U_i~|,
\end{equation}
where $u_i$ denotes the exact solution at $u(X_i,t)$.

\begin{figure}[!htbp]
	\centering
	\hspace*{-20pt}
	\includegraphics[height=5.5cm,width=7.5cm]{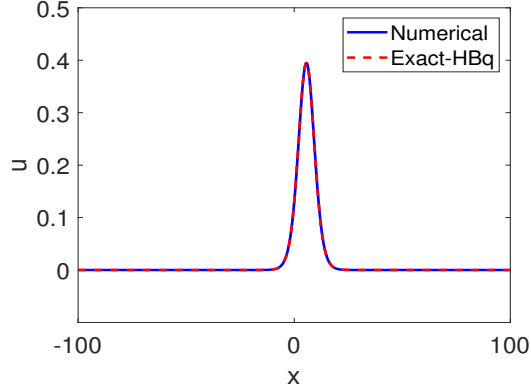}
	\caption{The solitary wave solution obtained by the Fourier pseudo-spectral scheme  for \eqref{gBE}  with $\alpha=0$ and the exact solitary wave solution for the HBq equation}\label{figure5}
\end{figure}

\noindent
In the first numerical experiment,  we take the initial data
\begin{eqnarray}\label{hbq-initial}
	&& u(x,0)=\frac{15}{38}\, \sech^4(\frac{x}{2\sqrt{13}}) \\
	&& v(x,0)= \frac{15 \sqrt{13}}{19\sqrt{133}} \sech^4(\frac{x}{2\sqrt{13}})
	\tanh(\frac{x}{2\sqrt{13}})
\end{eqnarray}
corresponding to the solitary wave solution  \eqref{soliterhbq} for $\eta_1=\eta_2=1, p=2$ initially
located at $x_0=0$. The experiment was run from $t=0$ to $t=5$ in the space interval $-100 \leq x\leq 100$
taking the number of  spatial grid points $N=2048$ and  the number of temporal grid  points $M=1000$.
To test the accuracy of the scheme \eqref{gBq_fourier1}-\eqref{gBq_fourier2}, the numerical solution obtained for
\eqref{gBE} with $\alpha=0$
is compared with the
exact solitary wave solution for the HBq equation \eqref{soliterhbq} at $t=5$ in Figure \ref{figure5}. As it is seen from Figure \ref{figure5}, the  numerical solution and the exact solution coincide very well. The  $L_\infty$-error norm at $t=5$ is $8.33 \times 10^{-16}$. It shows that our proposed method is capable of high accuracy.

\begin{figure}[!htbp]
	\begin{minipage}[t]{0.45\linewidth}
		\centering
		\hspace*{-20pt}
		\includegraphics[height=5.5cm,width=7.5cm]{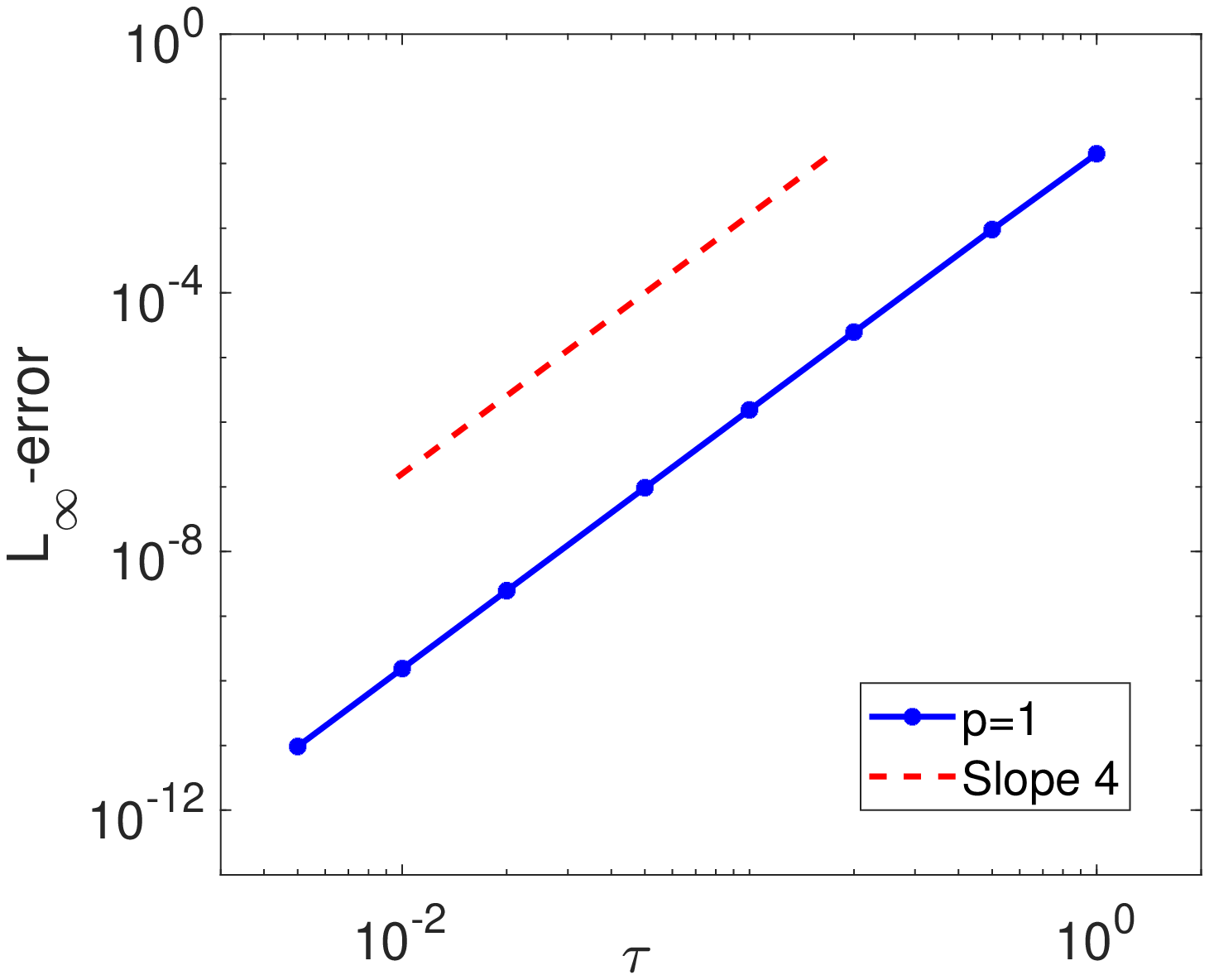}
	\end{minipage}%
	\hspace{3mm}
	\begin{minipage}[t]{0.45\linewidth}
		\centering
		\includegraphics[height=5.5cm,width=7.5cm]{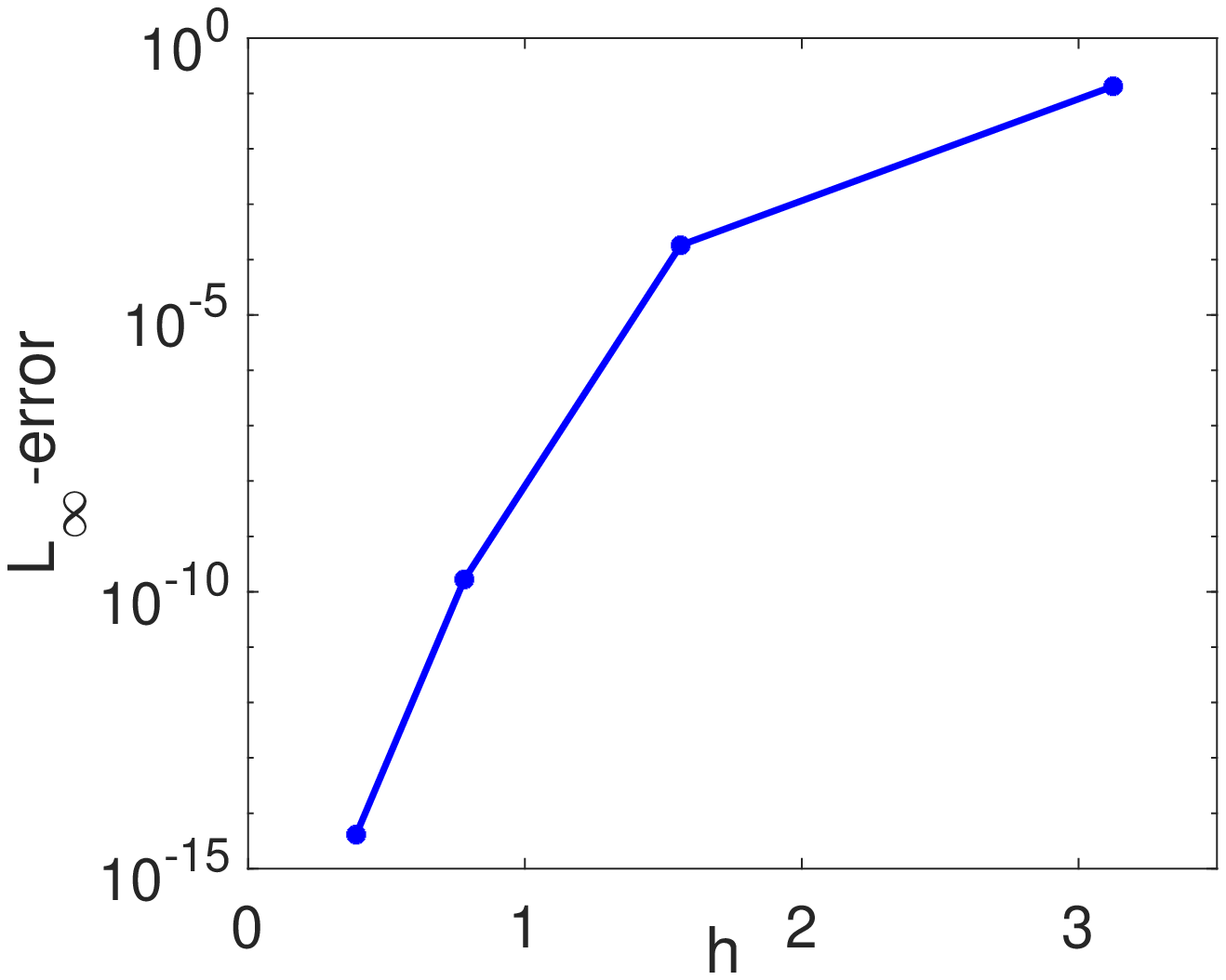}
	\end{minipage}
	
	\caption{The convergence rates in time calculated from
		the $L_{\infty}$-errors  (left panel) and  in space calculated from
		the $L_{\infty}$-errors (right panel).}   \label{ordertime}
\end{figure}

\begin{figure}[!htbp]
	\begin{minipage}[t]{0.45\linewidth}
		\centering
		\hspace*{-20pt}
     \centering
		\includegraphics[height=5.3cm,width=7.5cm]{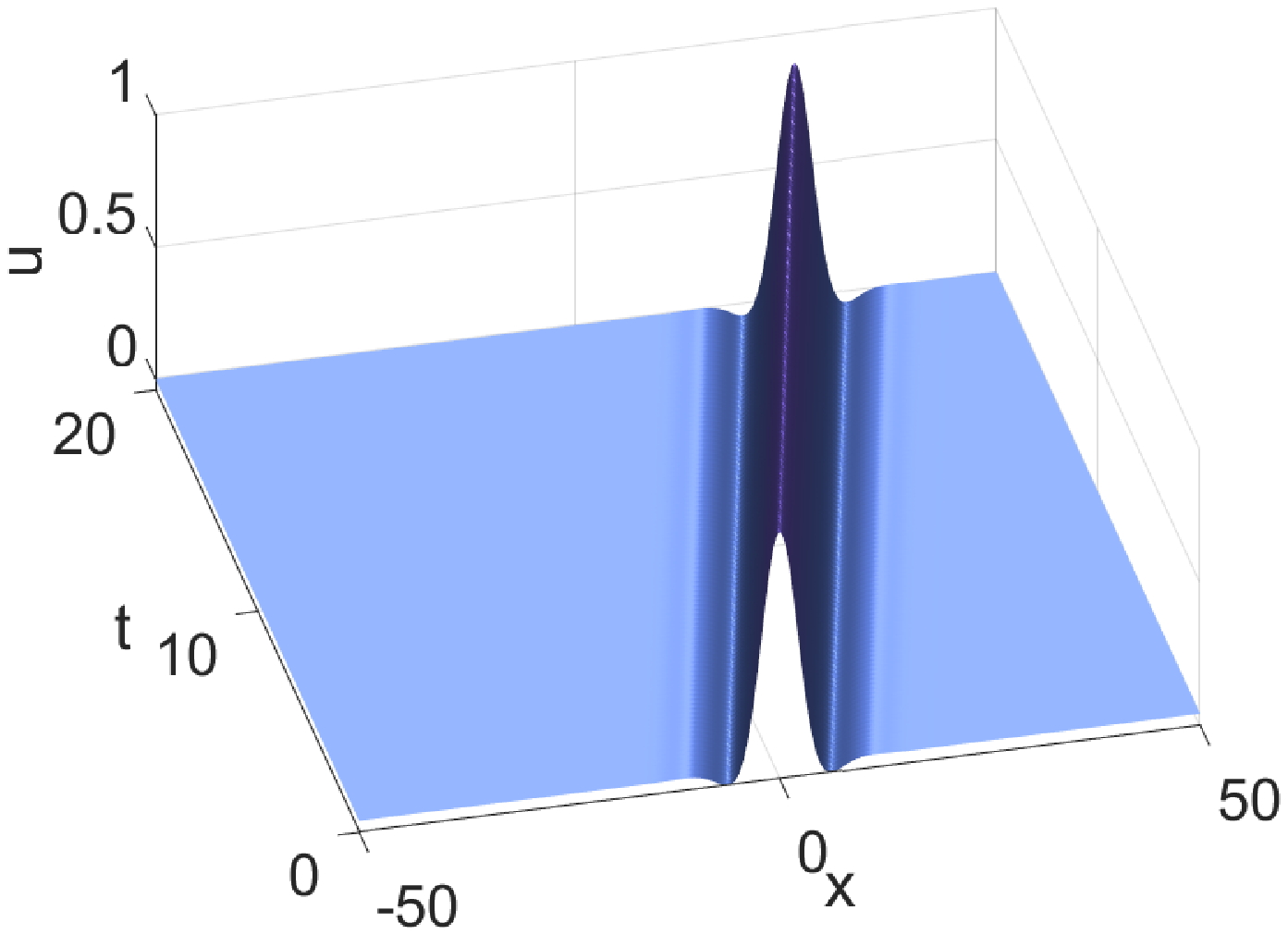}
	\end{minipage}%
	\hspace{3mm}
	\begin{minipage}[t]{0.45\linewidth}
		\centering
		\includegraphics[height=5.3cm,width=7.5cm]{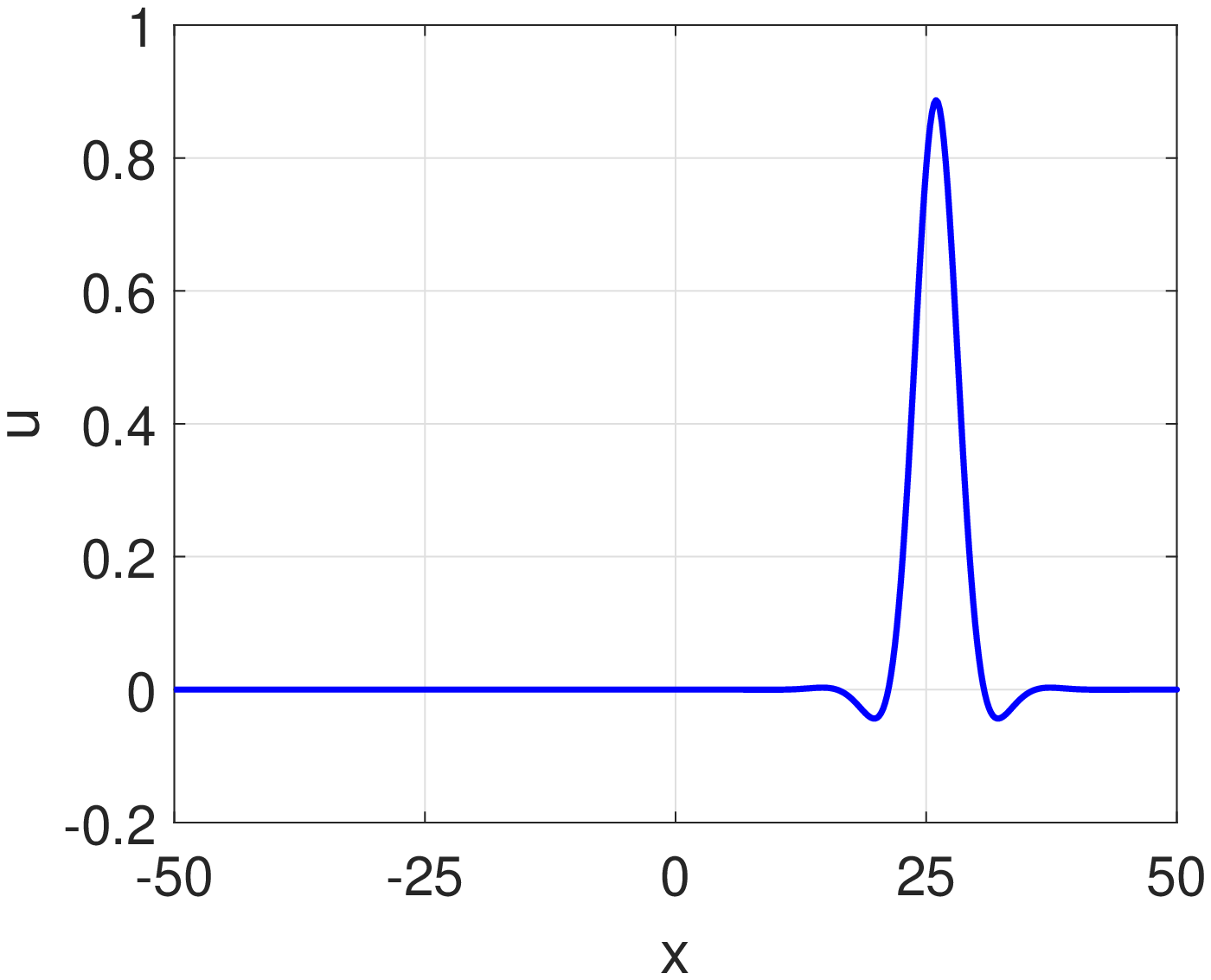}
	\end{minipage}
	\begin{minipage}[t]{0.45\linewidth}
		\centering
		\hspace*{-20pt}
		\includegraphics[height=5.3cm,width=7.5cm]{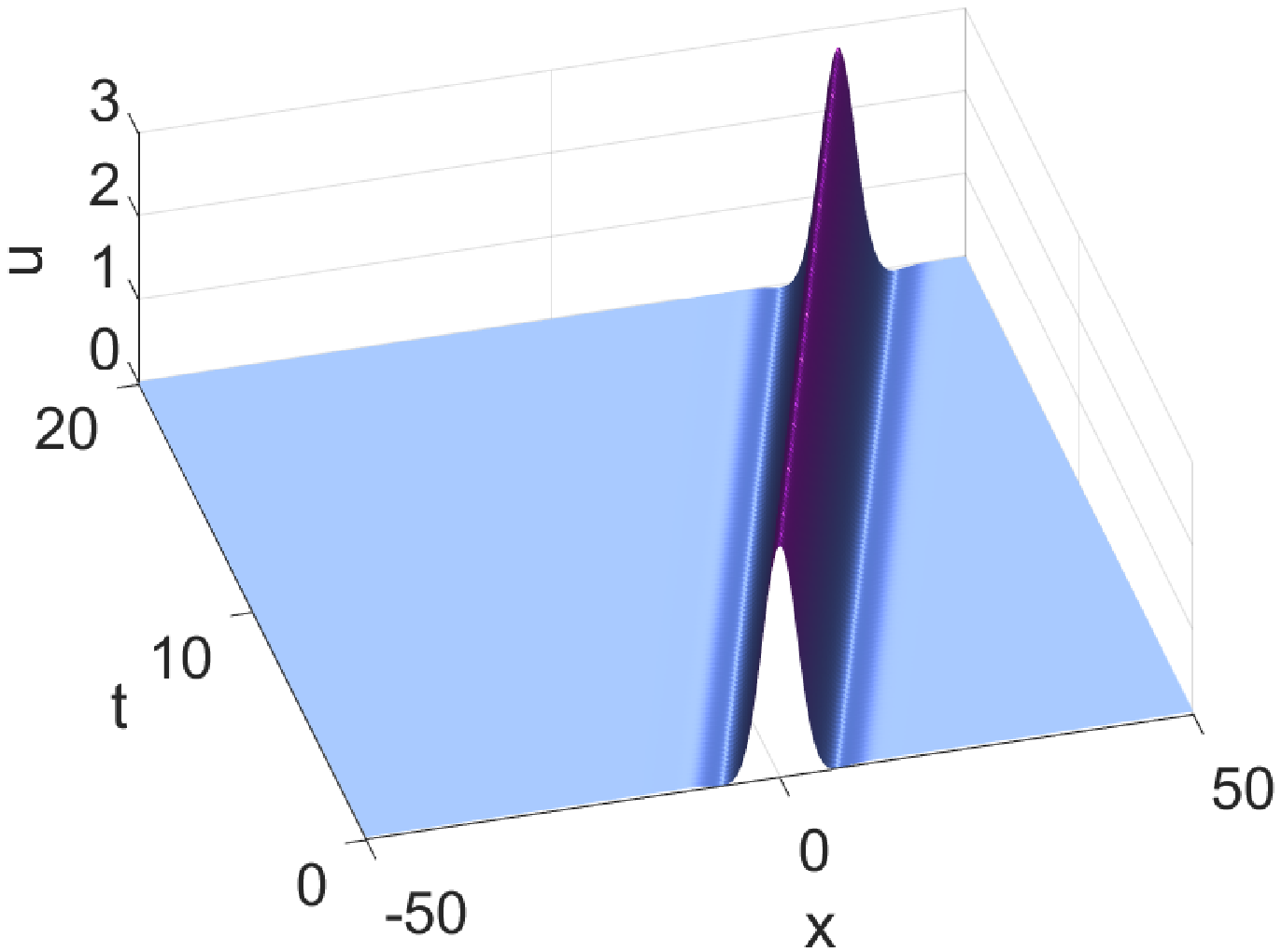}
	\end{minipage}%
	\hspace{14mm}
	\begin{minipage}[t]{0.45\linewidth}
		\centering
		\includegraphics[height=5.3cm,width=7.5cm]{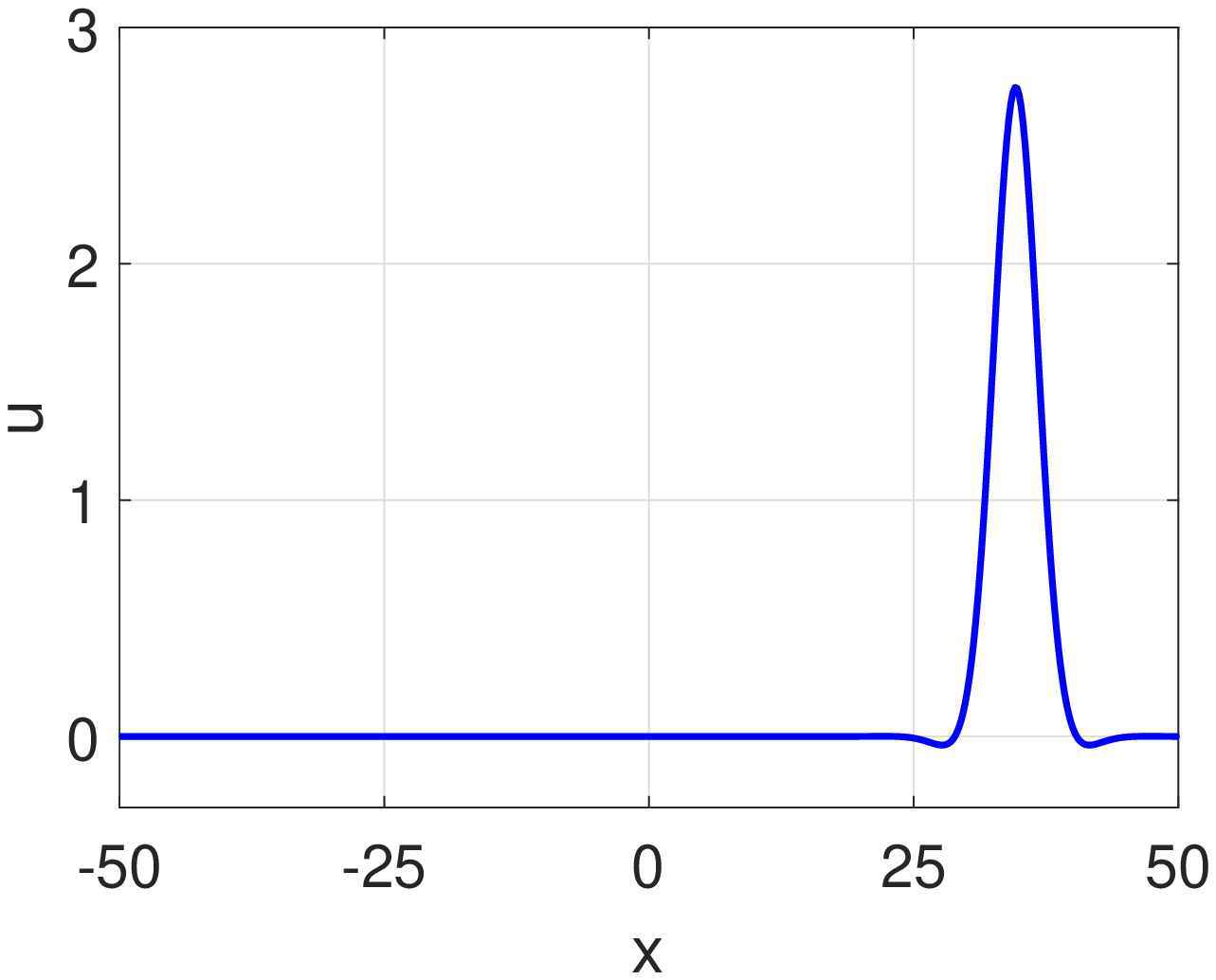}
	\end{minipage}
	\caption{Evolution of the solitary wave for the parameters $ \alpha=2, c=1.3$,  $\kappa=\beta=1 $ (top left), the profile at $t=20$ (top right)  and the evolution of the solitary wave for the parameters $\alpha=2, c=\sqrt{3}, \kappa=\beta=1$ (bottom left), the profile at $t=20$ (bottom right) .}\label{figure7}
\end{figure}

\begin{figure}[!htbp]
	\centering
	\hspace*{-20pt}
	\includegraphics[height=5.5cm,width=7.5cm]{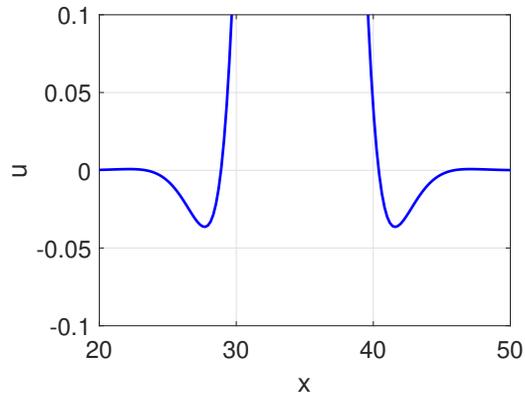}
	\caption{Magnification of the profile at $t=20$ for the parameters $\alpha=2, c=\sqrt{3}, \kappa=\beta=1$ }\label{figure75}
\end{figure}

\noindent
Our goal is to understand numerically how the higher-order effects of dispersion and nonlinearity affect the solitary wave solutions of \eqref{gBE} as time increases.
Therefore, we need initial conditions for $u(x,0)$ and $u_t(x,0)$. The generated solitary wave profile $Q_c$ using Petviashvili's iteration method for the parameters  $\alpha=2, c=1.3, \kappa=\beta=1$ is taken as $u(x,0)$. Since the localized solitary wave solution has the  form $u(x,t)=Q_c(x-ct)$, we get $u_t(x,t)=-c Q_c^{\prime}(x-ct)$. We set
$u_t(x,0)=-cQ_c^{\prime}(x)$. It is computed  by \mbox{$-c{\cal F}_{j}^{-1}[ik{\cal F}_{k}[Q_c(j)]]$} numerically.  Since the exact solitary wave solution is unknown,  the ``exact'' solitary wave solution $u^{\rm ex}$ is obtained numerically  with a very fine spatial step size  $N=1024$  and a very small time step $M=5000$   by using Fourier pseudo-spectral method.

\noindent
In order to test the temporal discretization errors for  the gBq equation  where the parameter $\alpha$ is different from zero, we fix the number of  spatial grid points $N=1024$ and solve \eqref{gBE} for different time step $\tau$. The convergence rates  calculated from the $L_{\infty}$-errors
at the terminating time $t=5$ are illustrated in the left panel  of Figure \ref{ordertime}. The computed convergence rates agree well
with the fact that Fourier pseudo-spectral method exhibits fourth-order convergence in time. In order to  test the spatial discretization errors,  we fix the time step $(\tau=0.001)$  such that the temporal error can be neglected, and solve \eqref{gBE} for different spatial step size $h$. The right panel of
Figure \ref{ordertime} shows  the variation of  $L_{\infty}$-errors with spatial step size.
The error decays very rapidly when the spatial step size decreases.
These results show that the numerical solution obtained using the
Fourier pseudo-spectral scheme converges rapidly to the accurate solution in space.

Figure \ref{figure7} shows evolution of the nonmonotone solitary wave for the parameters $ (i)~\alpha=2, c=\sqrt{3}$, $\kappa=\beta=1 $ and  $(ii)~\alpha=2, c=1.3, \kappa=\beta=1$ . The solitary wave profiles at the final time $t=20$ are also presented.
The experiment was run from $t=0$ to $t=20$ in the space interval $-100 \leq x\leq 100$
taking the number of  spatial grid points $N=1024$ and  the number of temporal grid  points $M=20000$. 
As it is  seen from the figures, the solitary wave emerges without any change in their shapes. The magnification of the profile at $t=20$  for the parameters $\alpha=2, c=\sqrt{3}, \kappa=\beta=1$ is depicted in Figure \ref{figure75}.  The solution still has the sign-changing part at the final time. We do not observe any dispersive tail as time increases. 

\subsection{Global and Blow-up Solutions}
In this subsection, we first test  the ability of the proposed method to investigate the global existence and blow-up solutions of the gBq equation. The analytical results  needed for the global existence or blow-up of the solutions for \eqref{gBE} have been discussed  in Section \ref{section-evolution} but there remains a gap between the global existence and blow-up intervals. Our aim is to give shed of light on the gap interval  neither a global existence nor a blow-up result is established.

For the numerical experiments in this subsection, the problem is solved on the interval $-100 \leq x \leq 100$  taking the spatial grid points as $N=2^{13}$. In the rest of the study, we consider \eqref{gBE} with cubic nonlinearity \mbox{$f(u)=-u^3$} setting the parameters $\alpha=\kappa=1$ and $\beta=-1$.  The steady-state solution of the equation
\begin{equation*}
	u-  u^{\prime\prime}-u^3=0
\end{equation*}
is given by $g(x)=\sqrt{2}\sech(x)$.  If we choose the
initial data
\begin{equation}
	u(x,0)=-\sqrt{2}A\sech(x)\tanh(x), \hspace*{30pt} v(x,0)=0,
	\label{inamp}
\end{equation}
where $A>0$, one gets
\begin{eqnarray*}
	d&=&J(g) =\frac12\|g\|_{H^1}^2-\frac14\|g\|_{L^4}^4
	=\frac43,\\
	E(u(x,0))&=&\frac12\left(\|u(\cdot,0)\|_{L^2}^2+\|v(\cdot,0)\|_{L^2}^2+ \|u_x(\cdot,0)\|_{L^2}^2
	+\|v_{x}(\cdot,0)\|_{L^2}^2 +\|(-\partial_x^2)^{-1/2}v(.,0)\|_{L^2}^2
	\right)
	\\
	&&
	\hspace*{15pt} -\frac14\|u(\cdot,0)\|_{L^4}^4 \\
	&=&
	\frac{4A^2}{35}(14-A^2),\\
	I(u(x,0))&=&\|u(.,0)\|_{H^1}^2-\|u(\cdot,0)\|_{L^4}^4=\frac{16A^2}{35}(7-A^2).
\end{eqnarray*}
Thus we obtain that
\begin{eqnarray*}
	E(u(x,0))&<&d\Leftrightarrow A\in(0,0.943)\cup(3.621,\infty)\\
	I(u(x,0)&< &0\Leftrightarrow A>\sqrt{7}\approx 2.646.
\end{eqnarray*}

\begin{figure}[!htbp]
	\begin{minipage}[t]{0.45\linewidth}
		\centering
		\hspace*{-20pt}
		\includegraphics[height=5.5cm,width=7.5cm]{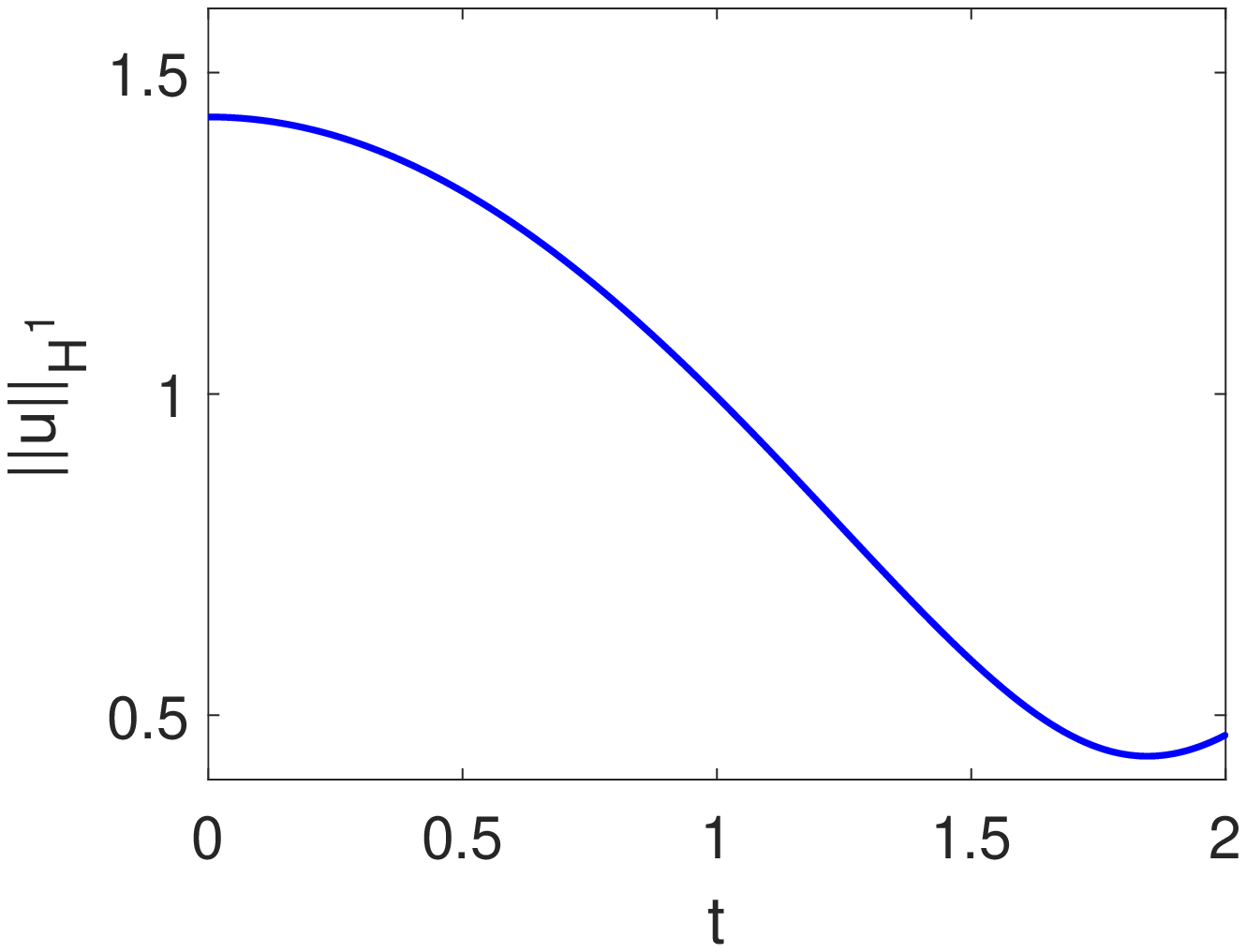}
	\end{minipage}%
	\hspace{3mm}
	\begin{minipage}[t]{0.45\linewidth}
		\centering
		\includegraphics[height=5.5cm,width=7.5cm]{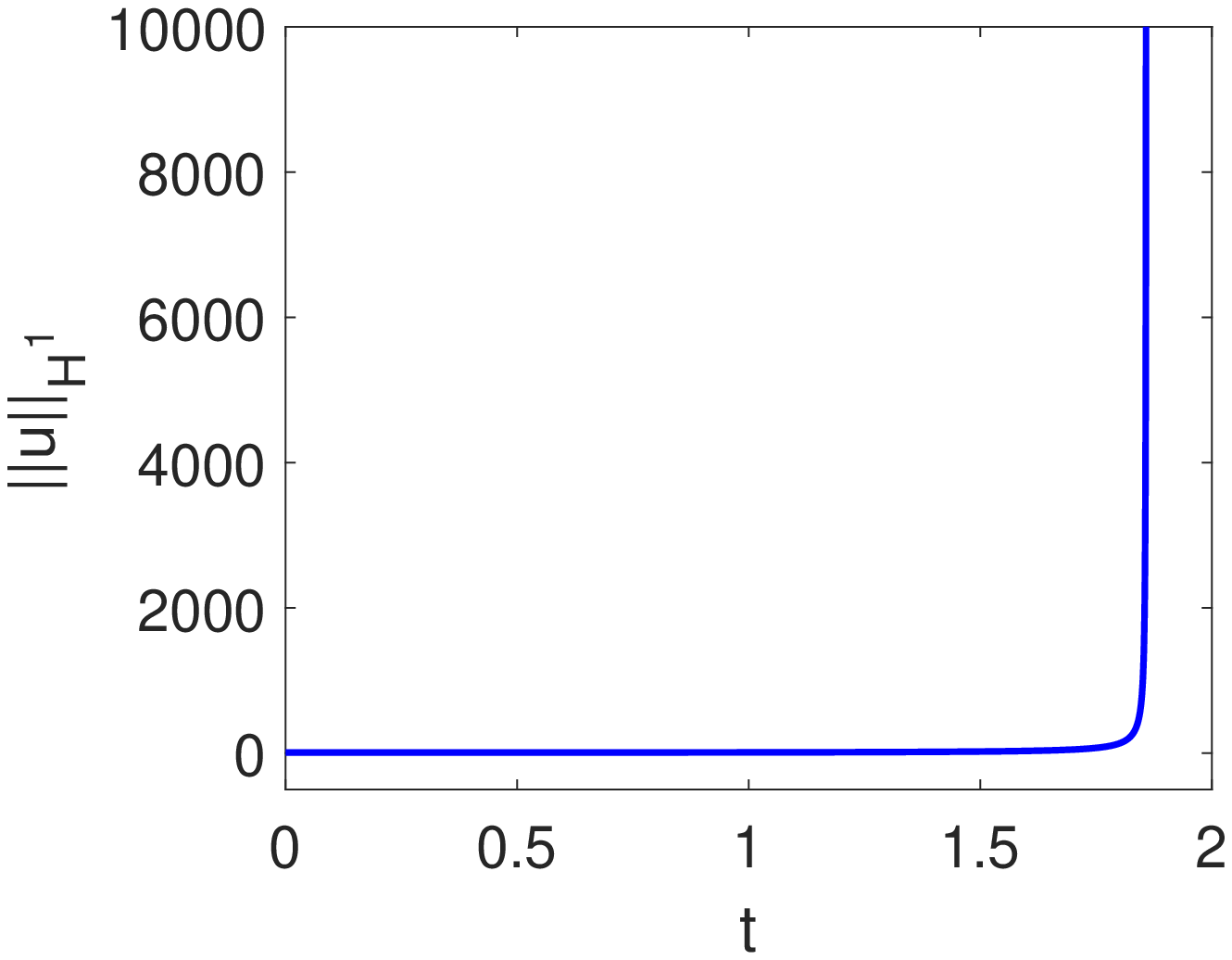}
	\end{minipage}
	
	\caption{The variation of the $H^1$-norm of the approximate solution   for $A=0.8$ (left panel) and $A=3.7$ (right panel).}   \label{figure9}
\end{figure}

\noindent
We present  the variation of the $H^1$-norm of the approximate solution  for the initial data \eqref{inamp} with $A=0.8$ and $A=3.7$ with time in Figure \ref{figure9}.  In the left panel of  Figure \ref{figure9},  the $H^1$-norm of the approximate solution
decreases  as time increases for  $A=0.8$. The numerical result indicates the global existence of the solution. This numerical result is also  compatible with the analytical result given in Theorem \ref{global-1}. In the right panel of  Figure \ref{figure9}, the $H^1$-norm of the approximate solution
increases  as time increases for  $A=3.7$. This is a strong indication that the solution  blows up in a finite time. This numerical result is in complete agreement with the analytical result given in Theorem \ref{blowup-1}.

\begin{figure}[!htbp]
	\begin{minipage}[t]{0.45\linewidth}
		\centering
		\hspace*{-20pt}
		\includegraphics[height=5.5cm,width=7.5cm]{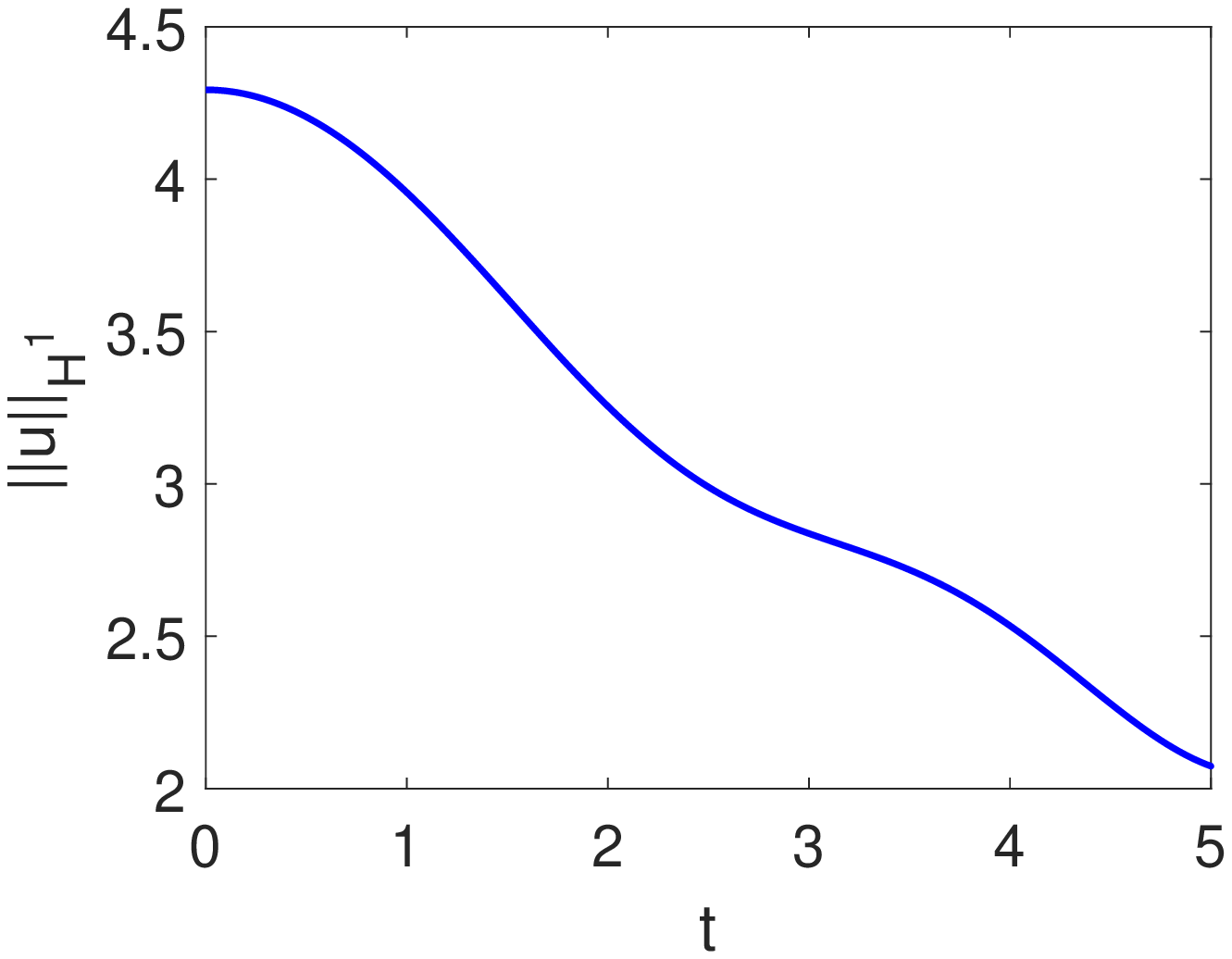}
	\end{minipage}%
	\hspace{3mm}
	\begin{minipage}[t]{0.45\linewidth}
		\centering
		\includegraphics[height=5.5cm,width=7.5cm]{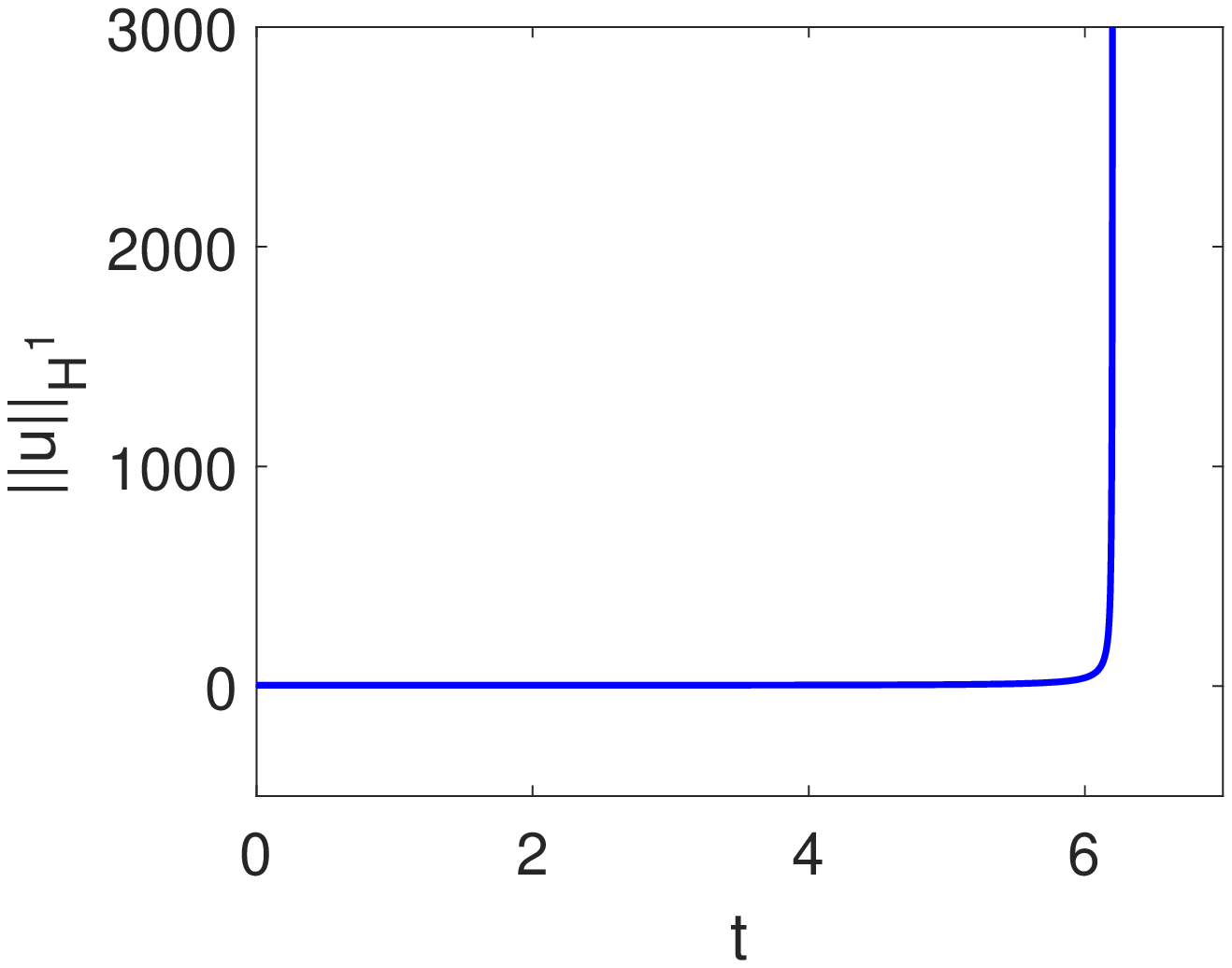}
	\end{minipage}
	
	\caption{The variation of the $H^1$-norm of the approximate solution with time  for the initial data \eqref{inamp} with $A=2.4$ (left panel) and $A=2.5$ (right panel).}   \label{figure10}
\end{figure}
\begin{figure}[!htbp]
	\begin{minipage}[t]{0.45\linewidth}
		\centering
		\hspace*{-20pt}
		\includegraphics[height=5.5cm,width=7.5cm]{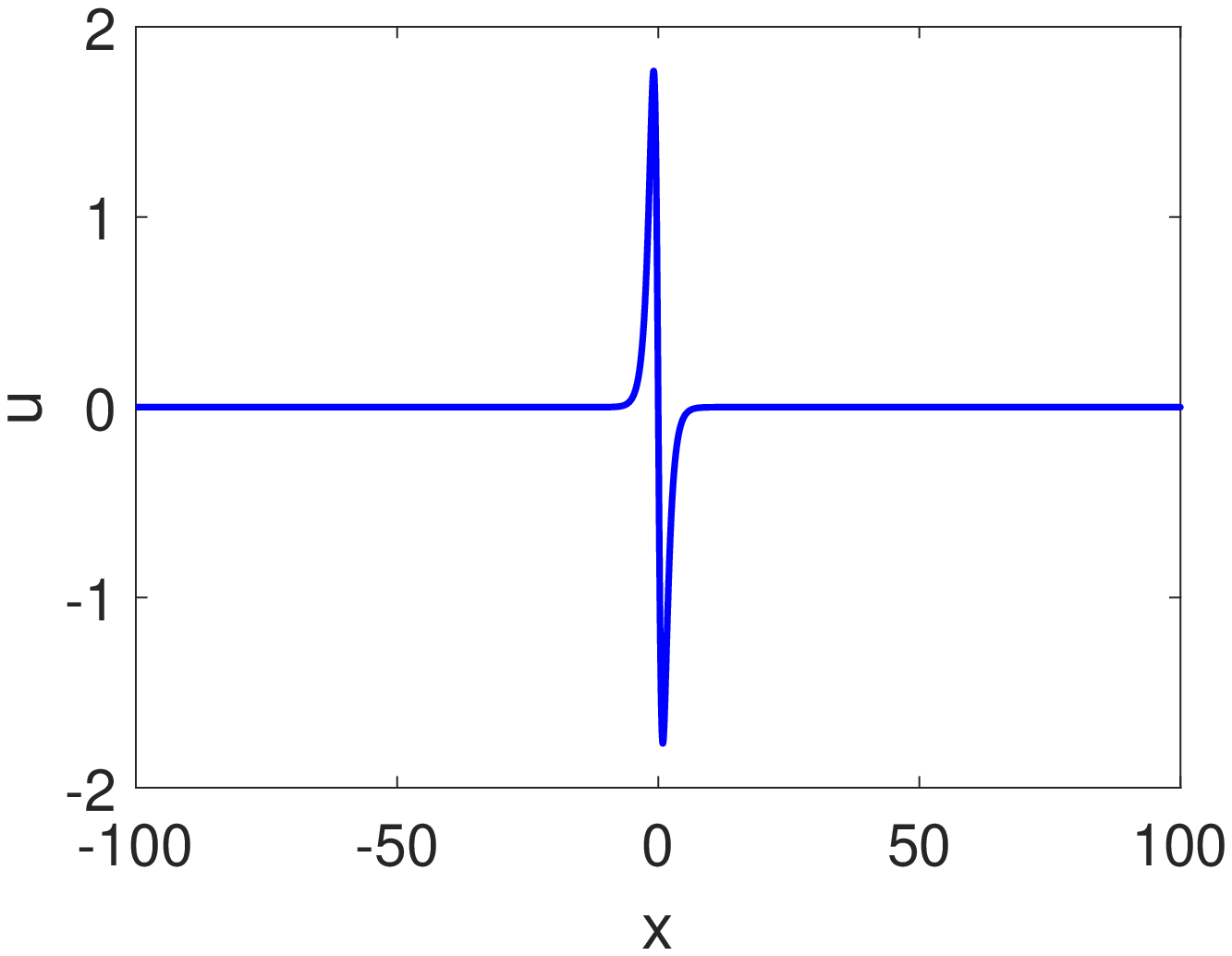}
	\end{minipage}%
	\hspace{3mm}
	\begin{minipage}[t]{0.45\linewidth}
		\centering
		\includegraphics[height=5.5cm,width=7.5cm]{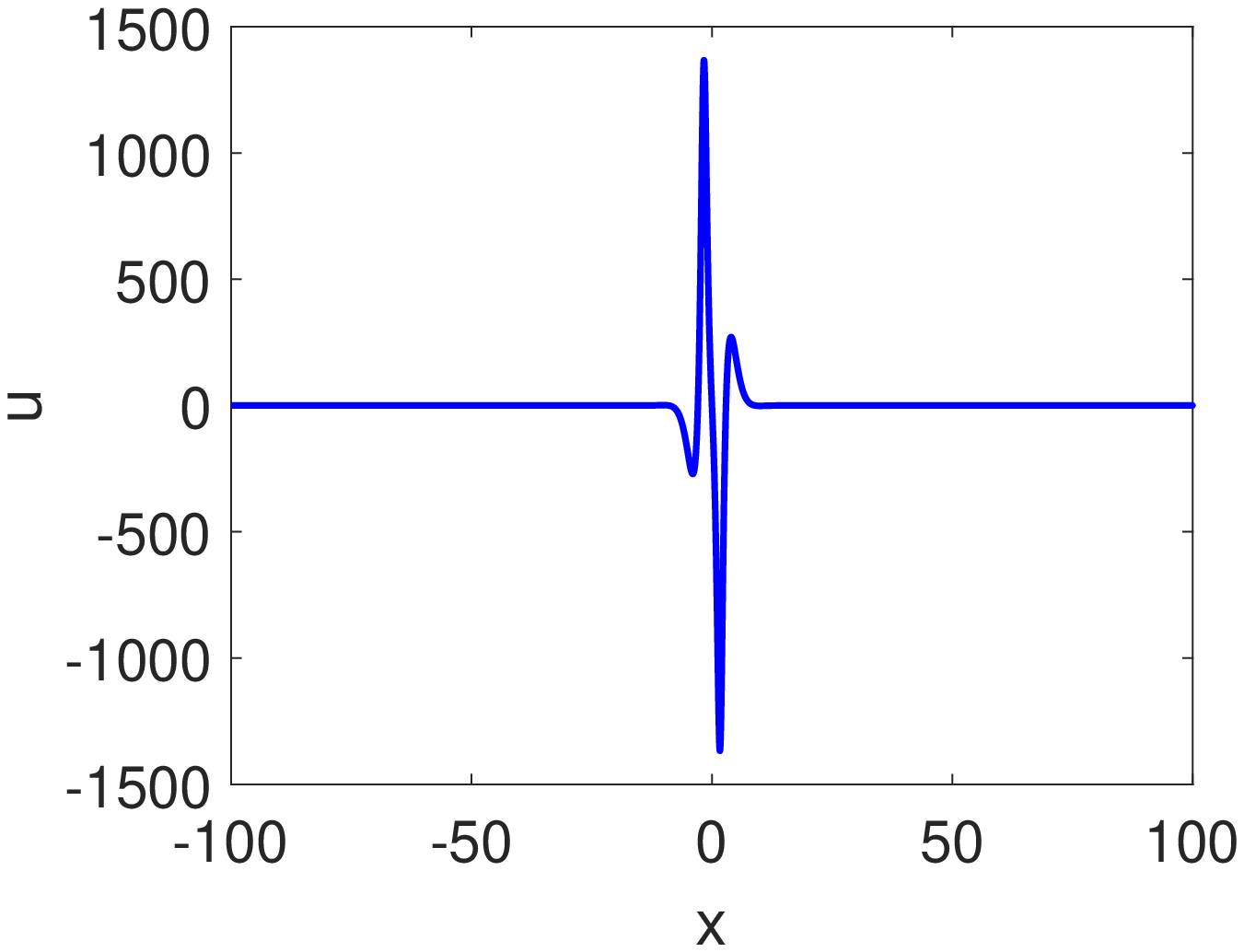}
	\end{minipage}
	\caption{The initial data \eqref{inamp} with $A=2.5$   (left panel) and the solution  profile at $t=6.2$ (right panel).}   \label{figure105}
\end{figure}

There is no analytical result for the gap interval $A\in[0.943,3.621]$. We performed lots of numerical experiments for different values of $A$. The numerical experiments indicate that there is a threshold value  $A^*\in (2.4,2.5)$ such that
the solution exists globally   for  $A \in [0.943, A^*)$ and it blows up in finite time for  $A \in [A^*,3.621]$.
The variation of  $H^1$-norm of the approximate solution with time is illustrated  for the initial data \eqref{inamp} with  $A=2.4$ and $A=2.5$ in Figure \ref{figure10}.
The $H^1$-norm of the approximate solution
decreases  as time increases for  $A=2.4$ which indicates the global existence of the solution.
For $A=2.5$,   the $H^1$-norm of the approximate solution
increases  as time increases. Figure \ref{figure105} shows the the initial data \eqref{inamp} with $A=2.5$ and the solution profile at $t=6.2$.    
It can be  seen that the solution  blows up in finite time.

\par In the second numerical experiment, we take the
initial data
\begin{equation}
	u(x,0)=v(x,0)=-\sqrt{2}A\sech(x)\tanh(x)
	\label{inamp2}
\end{equation}
where $A>0$. It yields
\begin{equation*}
	E(u(x,0))=A^2 \left(\frac{78}{15}-\frac{4}{35} A^2  \right),
\end{equation*}
and
\begin{eqnarray*}
	E(u(x,0))&<& d\Leftrightarrow A\in(0,0.508)\cup(6.726,\infty),  \\
	I(u(x,0)&< &0\Leftrightarrow A>\sqrt{7}\approx 2.646.
\end{eqnarray*}

\begin{figure}[!htbp]
	\begin{minipage}[t]{0.45\linewidth}
		\centering
		\hspace*{-20pt}
		\includegraphics[height=5.5cm,width=7.5cm]{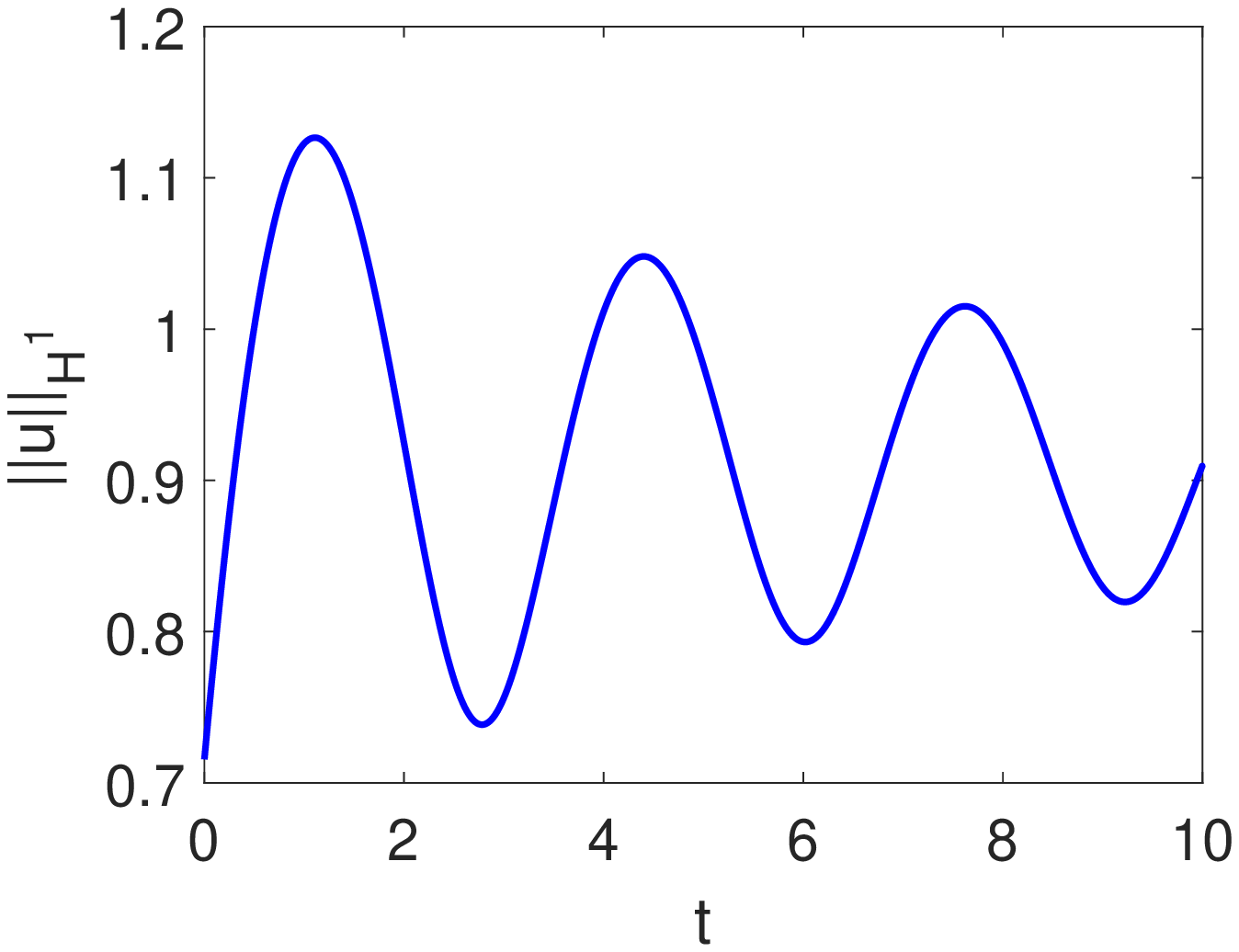}
	\end{minipage}%
	\hspace{3mm}
	\begin{minipage}[t]{0.45\linewidth}
		\centering
		\includegraphics[height=5.5cm,width=7.5cm]{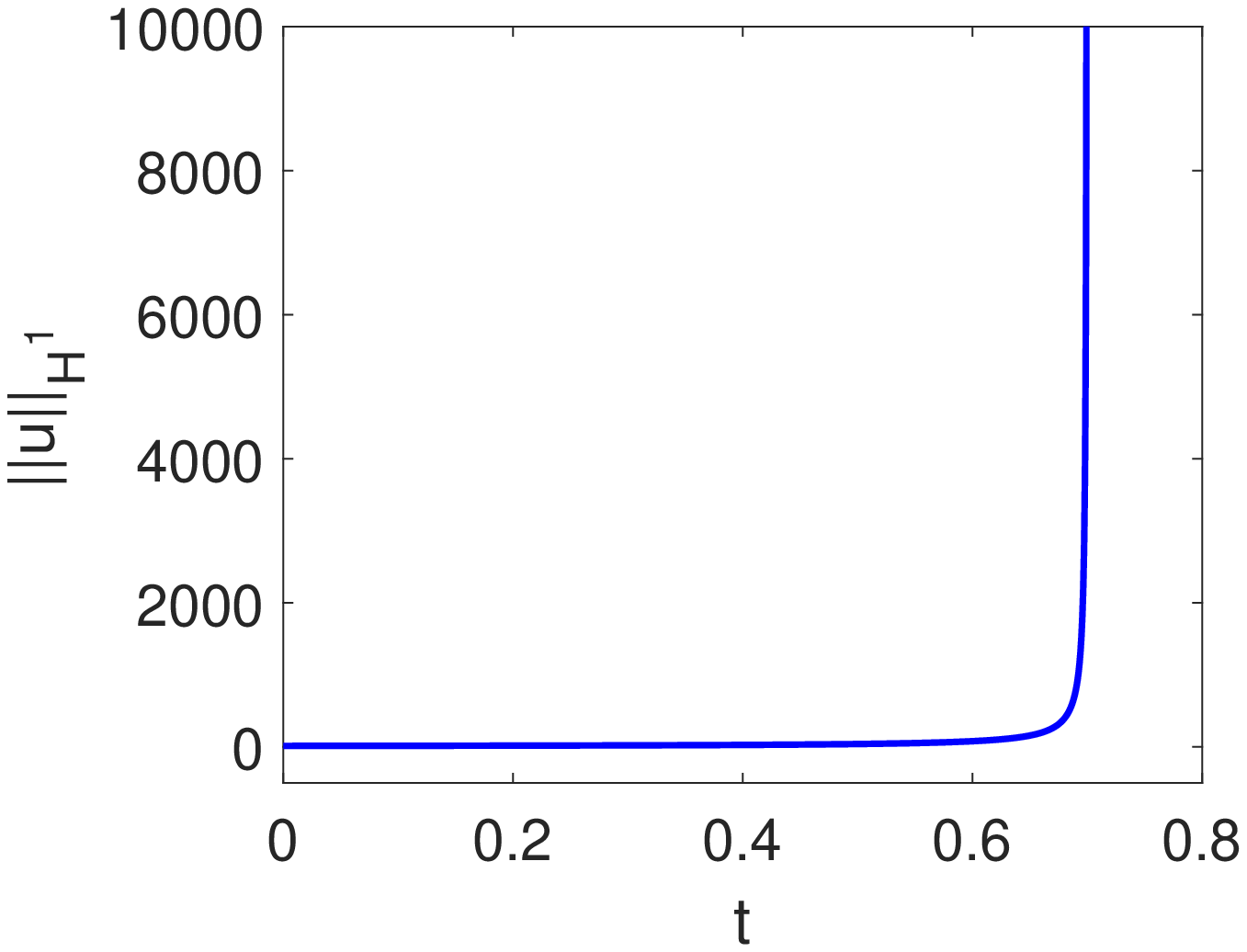}
	\end{minipage}
	
	\caption{The variation of the $H^1$-norm of the approximate solution   for $A=0.4$ (left panel) and $A=7$ (right panel).}   \label{figure11}
\end{figure}
\noindent
In the left panel of  Figure \ref{figure11},  the $H^1$-norm of the approximate solution
remains bounded for  $A=0.4$. It oscillates  as time increases.   The numerical result indicates the global existence of the solution. This numerical result is also  compatible with the analytical result given in Theorem \ref{global-1}
since  $E(u(x,0) <d$  and $I(u(x,0))>0$.
The $H^1$-norm of the approximate solution  increases  as time increases for  $A=7$. This is a strong indication that the solution  blows up in a finite time. This numerical result is in complete agreement with the analytical result given in Theorem \ref{blowup-1}. 
For the gap interval $A\in[0.508,6.726]$,  the numerical experiments indicate that there is a threshold value  $A^*\in (1, 1.1)$ such that
the solution exists globally   for the parameter $A \in [0.508, A^*)$ and it blows up in finite time for the parameter   $A \in [A^*,6.726]$.
The variation of the $H^1$-norm of the approximate solution is illustrated for $A=1$ and $A=1.1$ in Figure \ref{figure12}.  The initial data \eqref{inamp2} with $A=1.1$ and the evolution of the initial data
are depicted in Figure \ref{figure125}. This peak appears to blow-up.
\begin{figure}[!htbp]
	\begin{minipage}[t]{0.45\linewidth}
		\centering
		\hspace*{-20pt}
		\includegraphics[height=5.5cm,width=7.5cm]{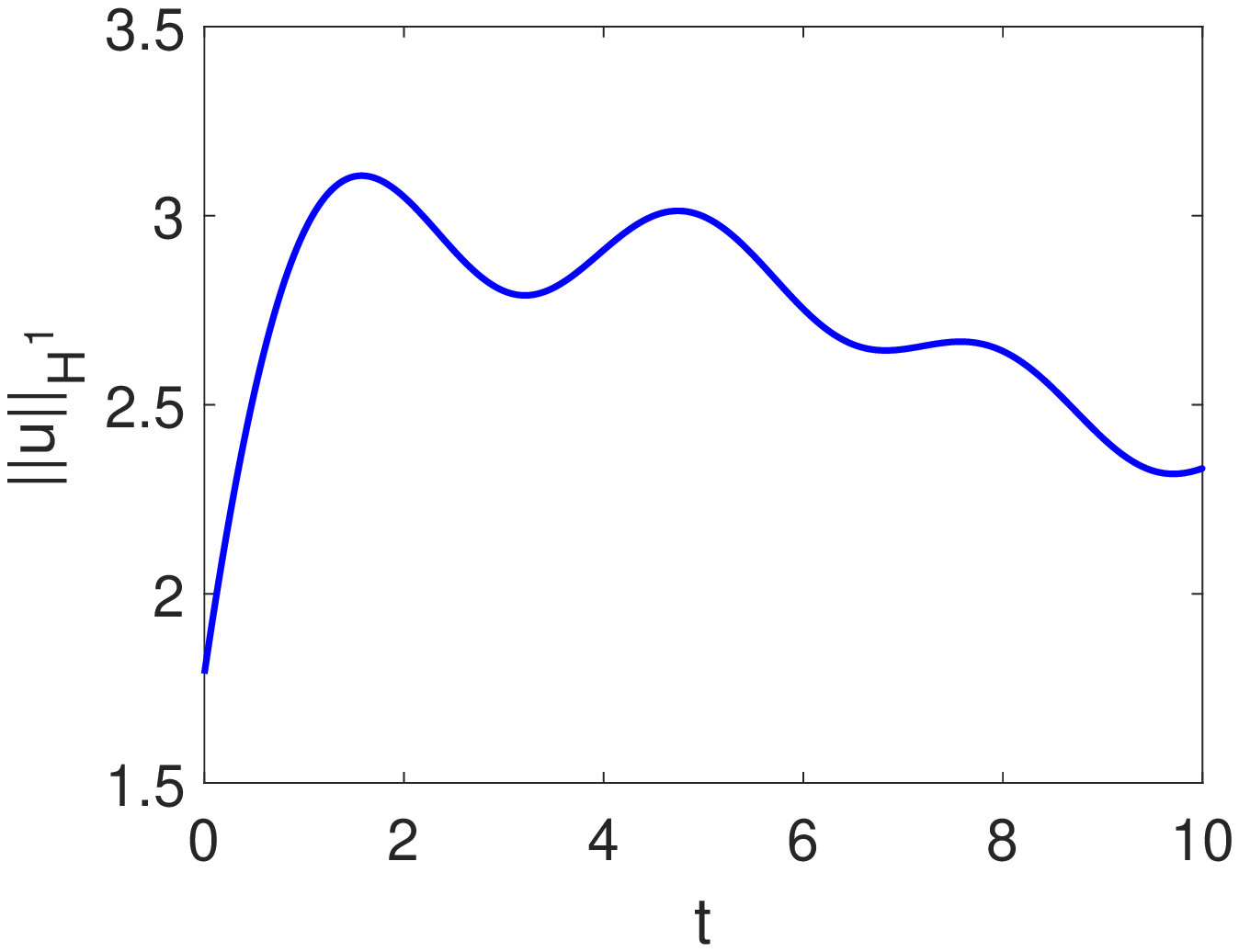}
	\end{minipage}%
	\hspace{3mm}
	\begin{minipage}[t]{0.45\linewidth}
		\centering
		\includegraphics[height=5.5cm,width=7.5cm]{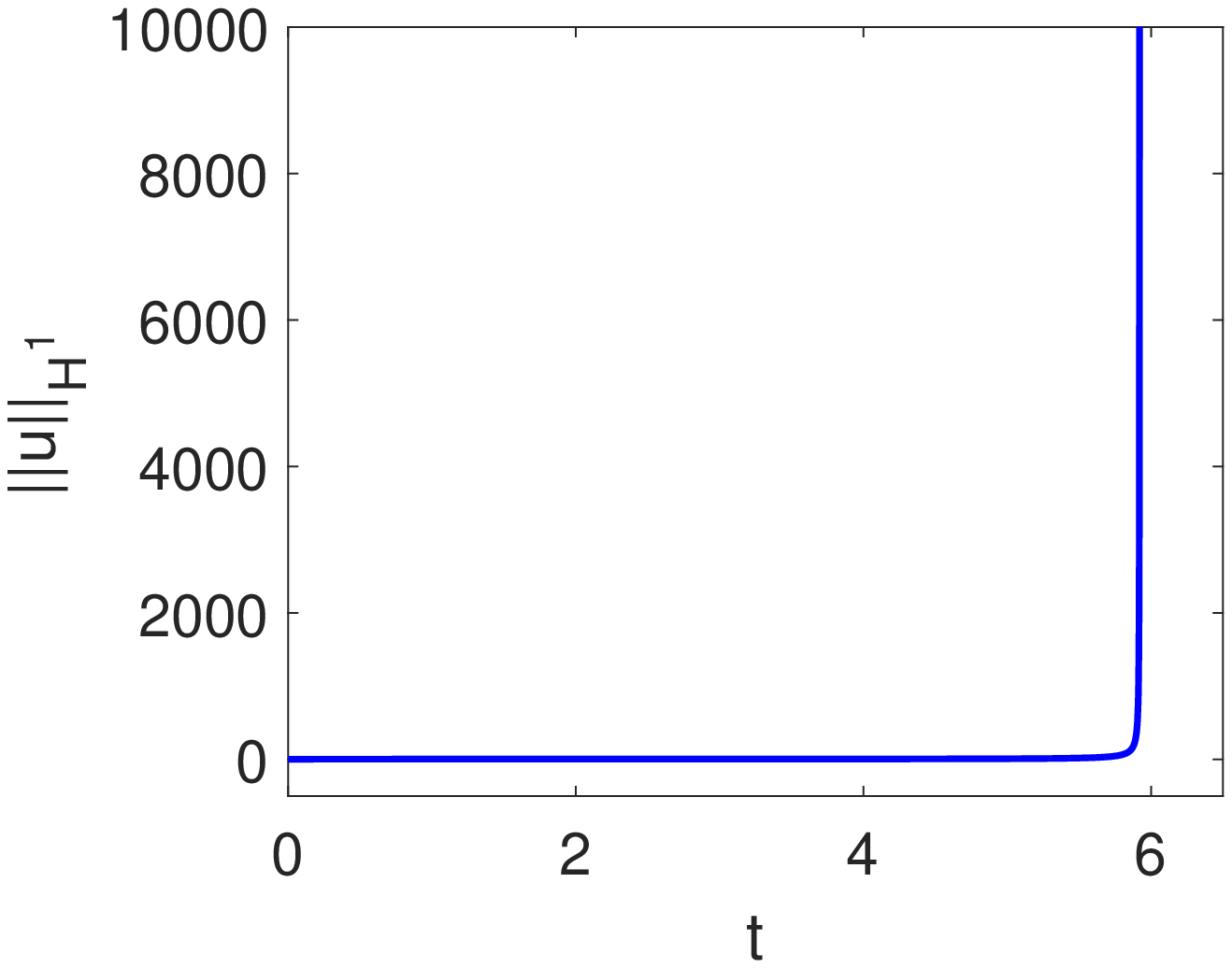}
	\end{minipage}
	
	\caption{The variation of the $H^1$-norm of the approximate solution with time  for the initial data \eqref{inamp2} with  for $A=1$ (left panel) and $A=1.1$ (right panel).}   \label{figure12}
\end{figure}

\begin{figure}[!htbp]
	\begin{minipage}[t]{0.45\linewidth}
		\centering
		\hspace*{-20pt}
		\includegraphics[height=5.5cm,width=7.5cm]{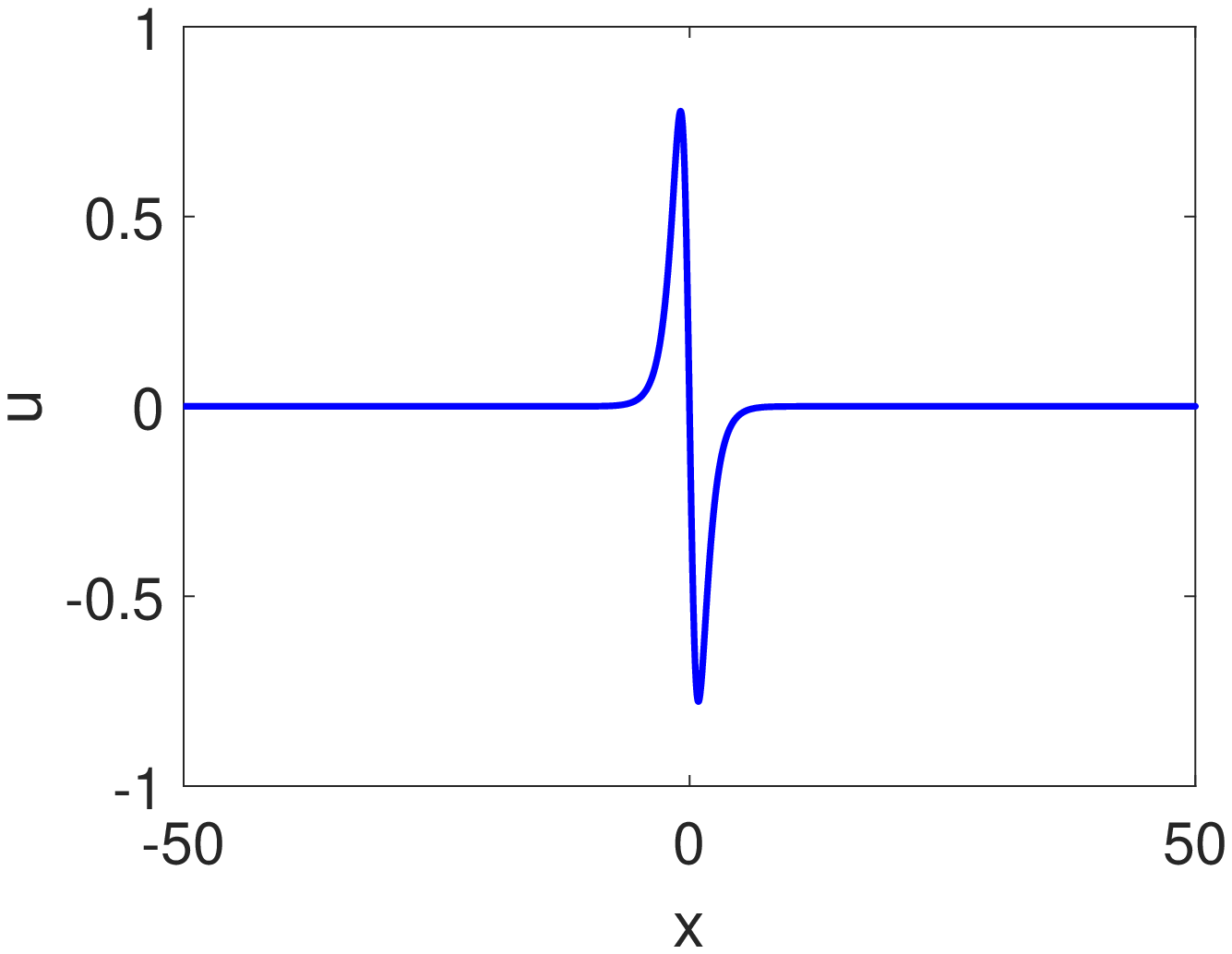}
	\end{minipage}%
	\hspace{3mm}
	\begin{minipage}[t]{0.45\linewidth}
		\centering
		\includegraphics[height=5.5cm,width=7.5cm]{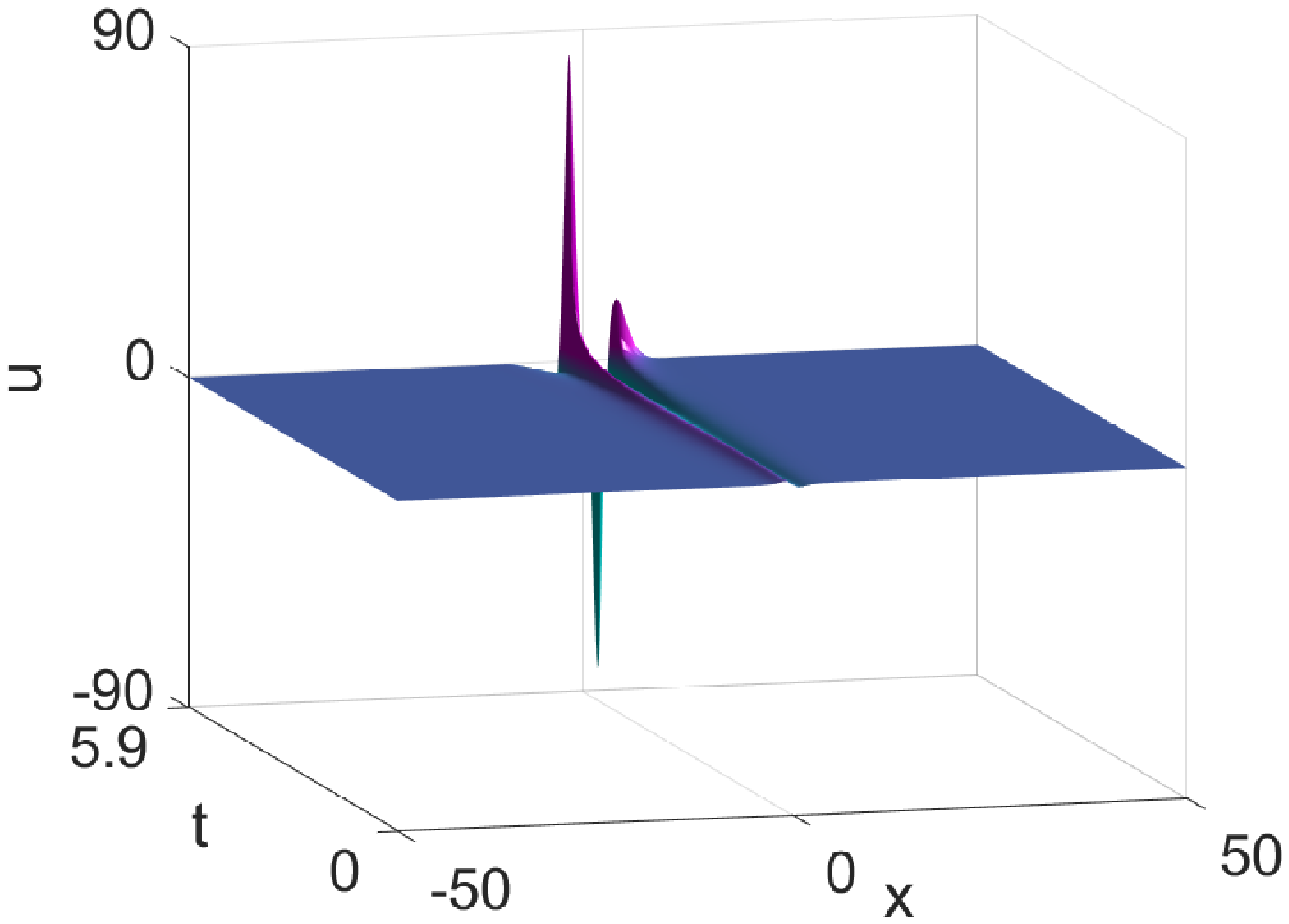}
	\end{minipage}
	\caption{The initial data \eqref{inamp2} with $A=1.1$ (left panel) and the evolution of the initial data (right panel).}   \label{figure125}
\end{figure}

\begin{figure}[!htbp]
	\begin{minipage}[t]{0.45\linewidth}
		\centering
		\hspace*{-20pt}
		\includegraphics[height=5.5cm,width=7.5cm]{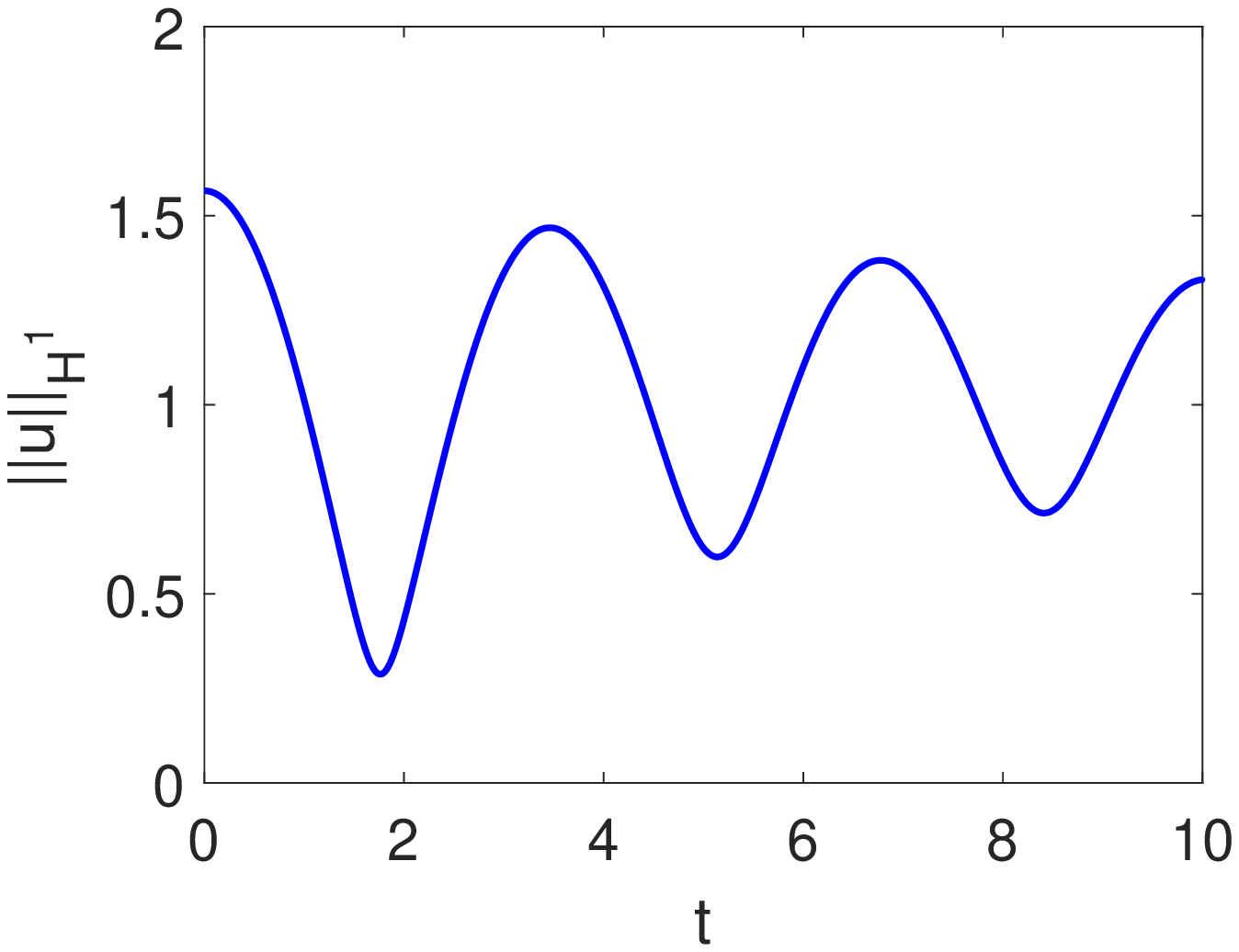}
	\end{minipage}%
	\hspace{3mm}
	\begin{minipage}[t]{0.45\linewidth}
		\centering
		\includegraphics[height=5.5cm,width=7.5cm]{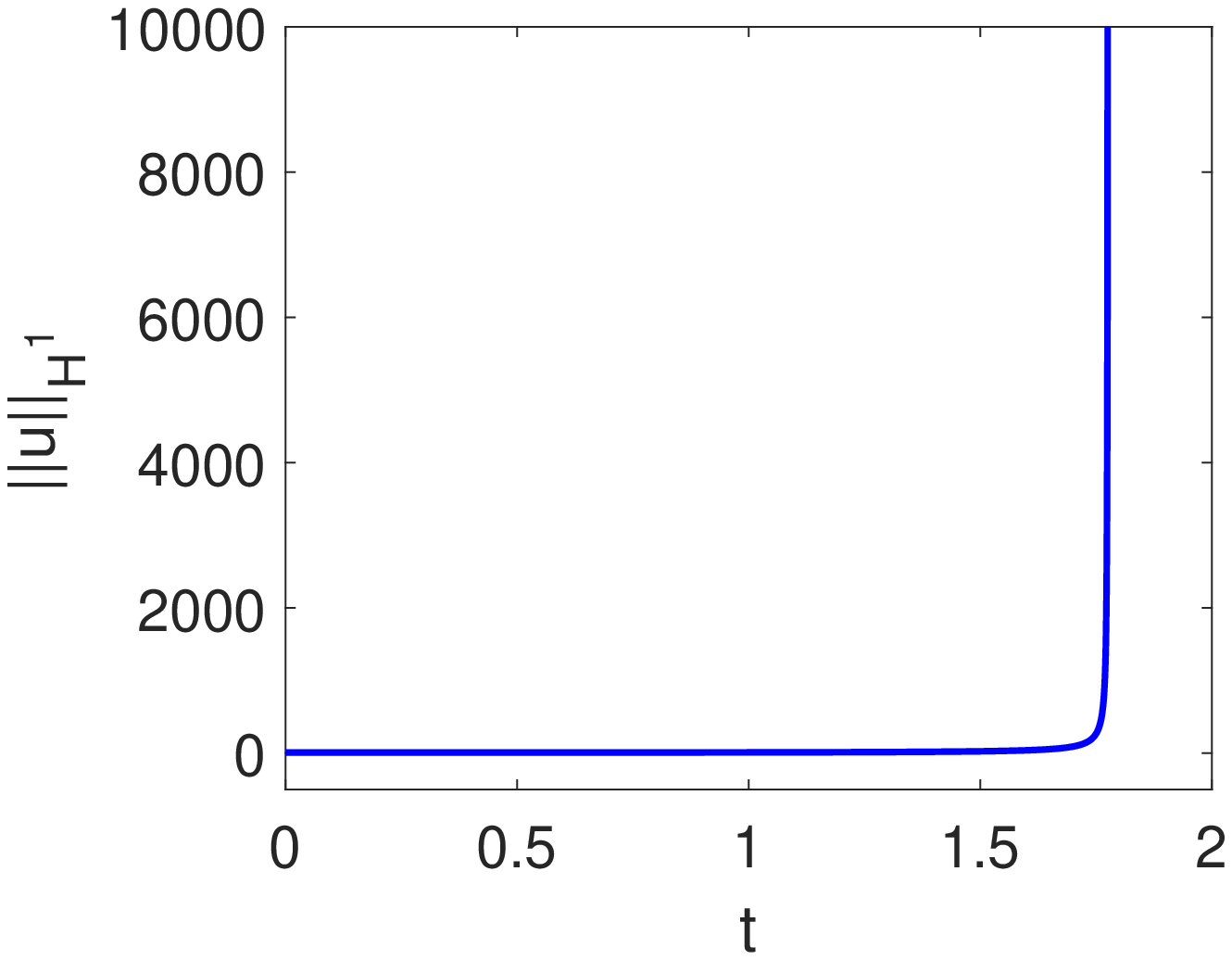}
	\end{minipage}
	\caption{The variation of the $H^1$-norm of the approximate solution with time  for the initial data \eqref{inamp3} with for $A=1.6$ (left panel) and $A=6.5$ (right panel).}   \label{figure13}
\end{figure}

\par In the last numerical experiment, we take the
initial data
\begin{equation}
	u(x,0)=A \Big[g(x)-\frac{1}{2}g(x-1)-\frac{1}{2}g(x+1) \Big],  \hspace*{30pt} v(x,0)=0,
	\label{inamp3}
\end{equation}
where $A>0$. It   yields
\begin{eqnarray*}
	E(u(x,0))&<& d\Leftrightarrow A\in(0,1.728)\cup(6.435,\infty),  \\
	I(u(x,0)&< &0\Leftrightarrow A>4.712.
\end{eqnarray*}

\noindent
The $H^1$-norm of the approximate solutions correspond to  $A=1.6$ and $A=6.5$ are depicted in Figure \ref{figure13}. The $H^1$-norm of the approximate solution
decreases with  oscillations  $A=1.6$. It remains bounded.   The numerical result indicates the global existence of the solution which is  compatible with the Theorem \ref{global-1}.
The $H^1$-norm of the approximate solution
increases  as time increases for  $A=6.5$. This is a strong indication of that the solution  blows up in finite time which is compatible with   Theorem \ref{blowup-1}.
\begin{figure}[!htbp]
	\begin{minipage}[t]{0.45\linewidth}
		\centering
		\hspace*{-20pt}
		\includegraphics[height=5.5cm,width=7.5cm]{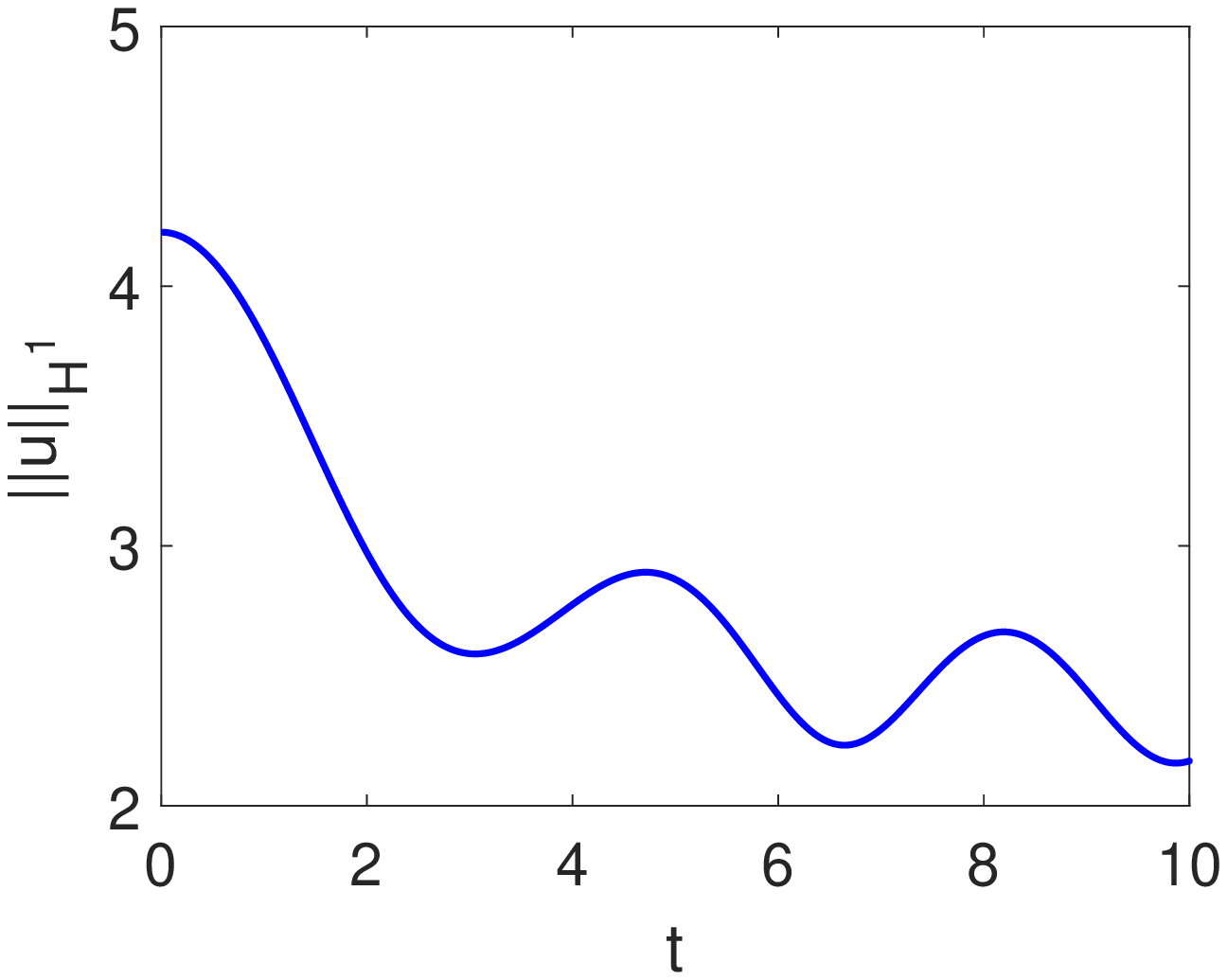}
	\end{minipage}%
	\hspace{3mm}
	\begin{minipage}[t]{0.45\linewidth}
		\centering
		\includegraphics[height=5.5cm,width=7.5cm]{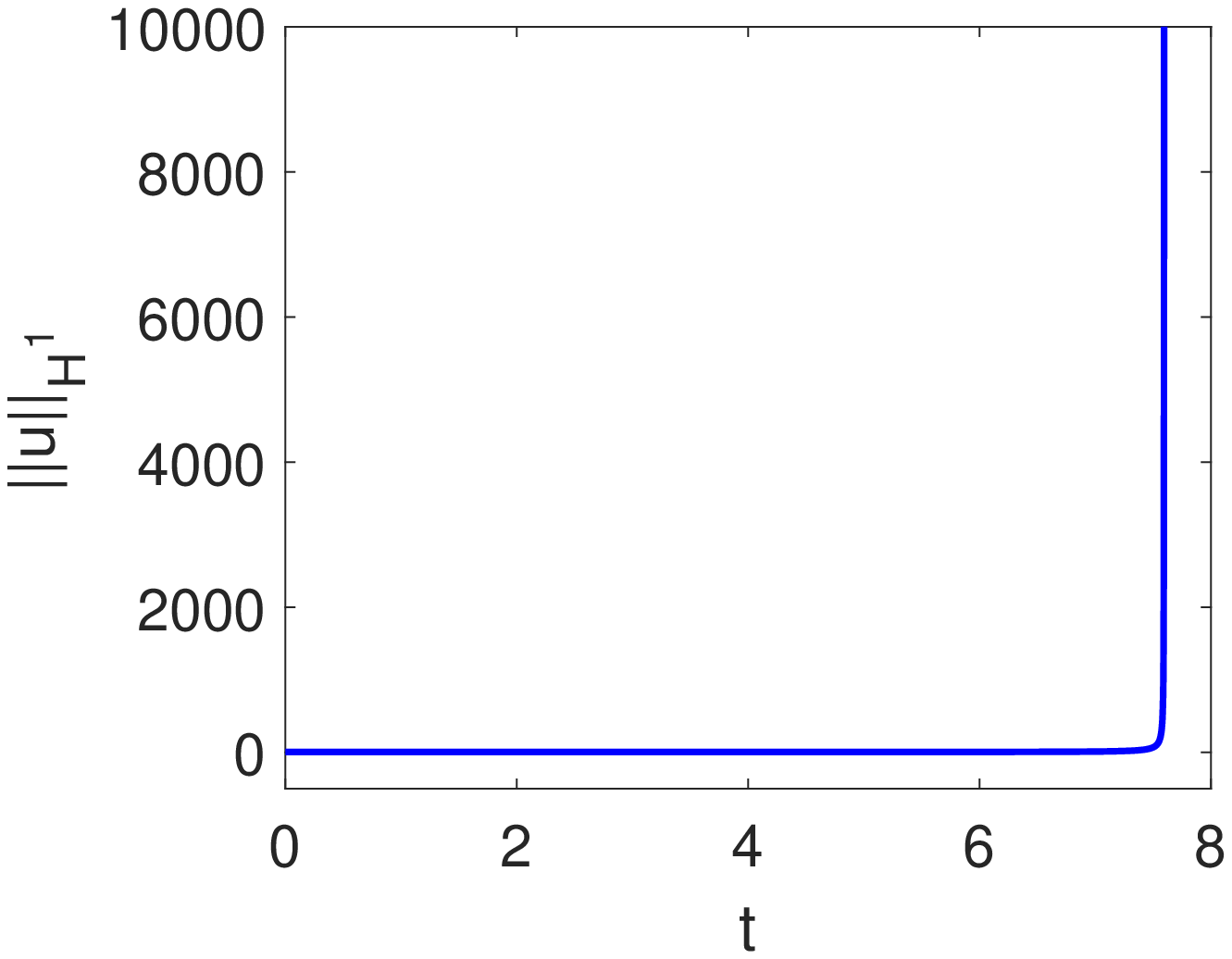}
	\end{minipage}
	
	\caption{The variation of the $H^1$-norm of the approximate solution with time  for the initial data \eqref{inamp3} with  for $A=4.3$ (left panel) and $A=4.4$ (right panel).}   \label{figure14}
\end{figure}

\begin{figure}[!htbp]
	\begin{minipage}[t]{0.45\linewidth}
		\centering
		\hspace*{-20pt}
		\includegraphics[height=5.5cm,width=7.5cm]{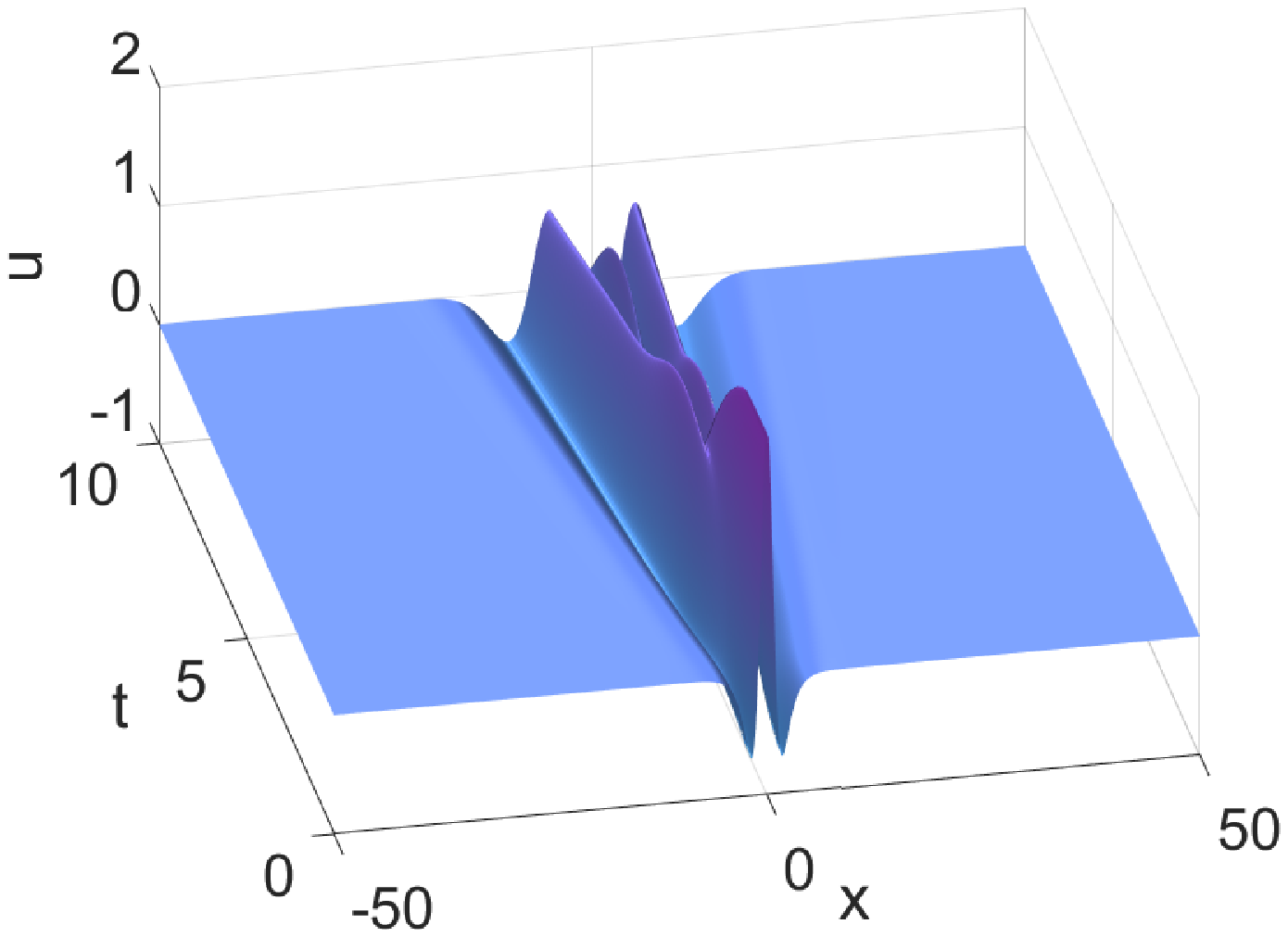}
	\end{minipage}%
	\hspace{3mm}
	\begin{minipage}[t]{0.45\linewidth}
		\centering
		\includegraphics[height=5.5cm,width=7.5cm]{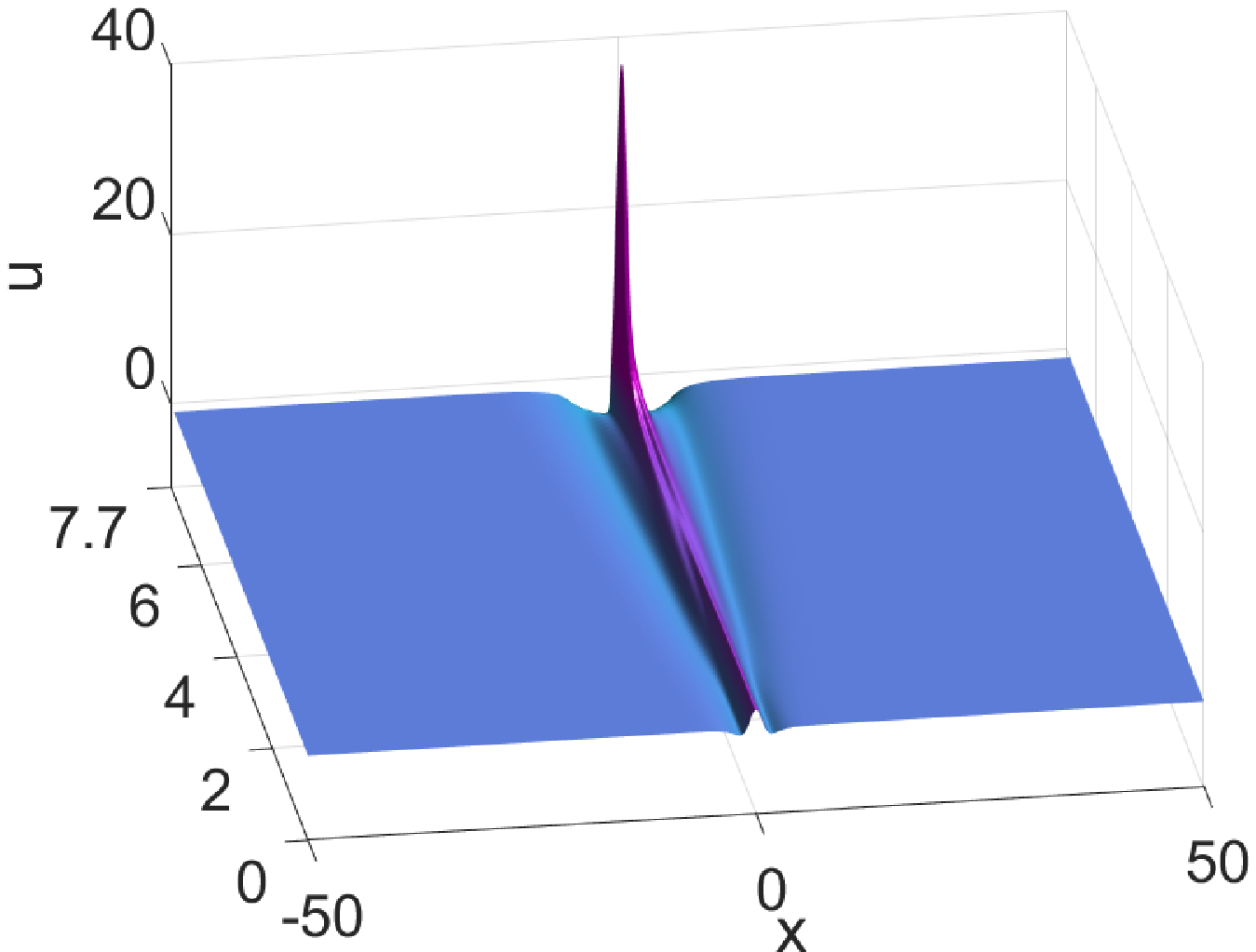}
	\end{minipage}
	
	\caption{Solution to the gBq equation for the initial data \eqref{inamp3} with $A=4.3$ (left panel) and $A=4.4$ (right panel).}   \label{figure145}
\end{figure}
\noindent

The variation of the $H^1$-norm of the approximate solution with time is illustrated  for $A=4.3$ and $A=4.4$ in Figure \ref{figure14}. In Figure \ref{figure145}, solutions to the gBq equation for the initial data \eqref{inamp3} with $A=4.3$  and $A=4.4$  are presented. 
It can be seen that the initial hump splits into several more strongly peaked humps for  the initial data \eqref{inamp3} with $A=4.3$,  whereas the initial hump evolves directly into a blow-up for  the initial data \eqref{inamp3} with $A=4.4$. For the gap interval $A\in[1.728,6.435]$,  the numerical experiments indicate that there is a threshold value  $A^*\in (4.3, 4.4)$ such that
the solution exists globally   for the parameter $A \in [1.728, A^*)$ and it blows up in finite time for the parameter  $A \in [A^*,6.435]$.

\section*{Conclusions}
This paper is focused on  the   Boussinesq equation with a general power-type nonlinearity. 
We first investigate from a theoretical point of view and derive the conditions under which the solutions of the initial value problem associated with this equation is global or blow up. Next, we extend our results to the Bessel potential spaces leading us to find the asymptotic behavior of the solutions.  Moreover,  the conditions  which guarantee the non-existence of solitary waves are obtained. We generate solitary wave solutions of the generalized Boussinesq equation by using the Petviashvili iteration method numerically. In order to investigate the time evolution of solutions to the generalized Boussinesq equation, the Fourier pseudo-spectral numerical method is proposed. After studying the time evolution of the single solitary wave,
we focus on the gap interval where neither a global existence nor a blow-up result has been established theoretically. Our numerical results successfully fill the gaps left by the theoretical ones.

\section*{Acknowledgements}
A. E. is supported by Nazarbayev University under Faculty Development Competitive Research Grants Program  for 2023-2025 (grant number 20122022FD4121).
G. M. M. is   supported by the Research Fund of the Istanbul Technical University (Project Number: 45000).

\section*{CONFLICT OF INTERESTS} It is stated by the authors that they have no competing interests.

\section*{DATA AVAILABILITY STATEMENT} 
Data sharing is not applicable to this article as no new data were created or analyzed in this study.

\bibliographystyle{elsarticle-num}
\biboptions{compress}

\end{document}